\documentclass[reqno]{amsart}
\usepackage[utf8]{inputenc}
\usepackage[T1]{fontenc}

\usepackage{ifthen}

\newif\ifdraftmode
\draftmodefalse

\newif\ifarxivversion
\arxivversiontrue

\newif\ifwithcomparison
\withcomparisontrue

\newif\ifwithappendixstable
\ifarxivversion
  \withappendixstabletrue
\else
  \withappendixstablefalse
\fi
  
\newif\ifwithappendixtwodim
\ifarxivversion
  \withappendixtwodimtrue
\else
  \withappendixtwodimfalse
\fi

\newif\ifwithdetails
\withdetailsfalse

\newif\iffinal
\finaltrue

\iffinal
\draftmodefalse
\withdetailsfalse
\fi

\usepackage{amssymb}
\usepackage{lmodern, microtype}
\usepackage{amscd, color}
\usepackage[shortlabels]{enumitem}
\usepackage{mathtools}
\usepackage{slashed}
\usepackage{accents}                            

\usepackage{tikz}
\usetikzlibrary{decorations.pathreplacing}
\usepackage{tikz-cd}
\usetikzlibrary{arrows}
\ifwithcomparison\else\usepackage{environ}\fi

\numberwithin{equation}{section}

\newcommand{\erhoehebox}[2]{\settoheight{\meineboxheight}{#2}\addtolength{\meineboxheight}{#1}\vbox to \meineboxheight{\vfil\hbox{#2}}}

\newcommand{\vertiefetextbox}[2]{\settoheight{\meineboxheight}{#2}\addtolength{\meineboxheight}{#1}\lower#1\vbox to \meineboxheight{\hbox{#2}\vfil}}

\newcommand{\bafrac}[3][.12\keyheightlength]{\mathchoice%
{\frac{\textstyle #2}{\erhoehebox{#1}{$\textstyle #3$}}}%
{\frac{\scriptstyle #2}{\erhoehebox{#1}{$\scriptstyle #3$}}}%
{\frac{\scriptscriptstyle #2}{\erhoehebox{#1}{$\scriptscriptstyle #3$}}}%
{\frac{\scriptscriptstyle #2}{\erhoehebox{#1}{$\scriptscriptstyle #3$}}}
}

\ifdraftmode
\newcounter{mnotecount}[section]

\fi

\ifdraftmode

\def\clcomm#1{{\color{blue}CL: #1}}
\else
\def\clcomm#1{}
\fi

\usepackage{xcolor}

\definecolor{darkgreen}{rgb}{0,0.6,0}
\definecolor{darkred}{rgb}{0.6,0,0}
\definecolor{darkblue}{rgb}{0,.1,.6}
\definecolor{darksharpblue}{rgb}{.4,.2,.8}
\definecolor{darkgray}{rgb}{0.3,.3,.3}
\definecolor{lightgreen}{rgb}{.9,1,.9}
\usepackage[%
   pdftitle={A quadratic lower bound for the number of minimal geodesics},
   pdfauthor={Bernd Ammann and Clara L\"oh},
   colorlinks, 
   citecolor=darkgreen, 
   linkcolor=darkblue,
   urlcolor=darkred,
   pdfpagelabels=true,
   unicode=true,
   pdfusetitle,
   pagebackref
 ]{hyperref}

\newcommand\darkblue[1]{\textcolor{darkblue}{#1}}

\ifdraftmode
\fi

\usepackage[normalem]{ulem}
\renewcommand\sout{\bgroup\markoverwith
{\textcolor{red}{\rule[0.7ex]{3pt}{1.4pt}}}\ULon}
 
\usepackage[shortcuts]{extdash} 

\let\Paragraph\S\let\S\relax

\def\al{{\alpha}}

\def\si{{\sigma}}

\let\epsilon\varepsilon
\def\ep{{\varepsilon}}

\let\theta\vartheta

\def\phi{{\varphi}}

\newcommand\ceq{\coloneqq}
\newcommand\qec{\eqqcolon}

\let\witi\widetilde
\let\wihat\widehat
\newcommand\inneres[1]{\accentset{\circ}{#1}}
\let\gdw\Leftrightarrow

\newcommand\diam{\mathop{\mathrm{diam}}}
\newcommand\image{\mathop{\mathrm{image}}}
\newcommand\upd{\mathrm{d}}

\newcommand\<{\langle}
\renewcommand\>{\rangle}

\newlength{\meineboxheight}
\newlength{\meineboxdepth}

\newlength{\templength}
\newcommand{\setzeboxnachlinksmath}[1]{\settowidth{\templength}{$#1$}\kern-\templength#1}

\newcommand\datver[1]{\def\datverp
{\par\boxed{\boxed{\text{Version: #1; Run: \today}}}}}\datver{0.1}

\newcommand{\NN}{\mathbb N}

\newcommand{\RR}{\mathbb R}

\newcommand{\ZZ}{\mathbb Z}

\newcommand{\Iso}{\operatorname{Iso}}

\let\ol\overline

\DeclareMathOperator{\length}{length}
\DeclareMathOperator{\conv}{conv}
\DeclareMathOperator{\convjoin}{conv-join}

\newcommand{\maC}{\mathcal C}

\newcommand{\maJ}{\mathcal J}

\newcommand{\maL}{\mathcal L}

\newcommand{\maN}{\mathcal N}

\newcommand{\maV}{\mathcal V}

\newcommand\Nmin{\maN_{\mathrm{min}}}

\newcommand\bullette{{\scalebox{0.5}{$\bullet$}}}
\newcommand\argu{\,\raisebox{.1\keyheightlength}{\bullette}\,}
\newcommand\arrgu{\argu,\argu}
\let\args\argu
\newtheorem{theorem}{Theorem}[section]
\newtheorem{proposition}[theorem]{Proposition}
\newtheorem{corollary}[theorem]{Corollary}

\newtheorem{lemma}[theorem]{Lemma}

\newtheorem{question}[theorem]{Question}
\newtheorem{openquestion}[theorem]{Open Question}

\theoremstyle{definition}
\newtheorem{definition}[theorem]{Definition}
\newtheorem{setup}[theorem]{Setup}

\newtheorem{remark}[theorem]{Remark}

\newtheorem{example}[theorem]{Example}

\newtheorem{history}[theorem]{Historical comment}

\ifwithcomparison

\newtheorem*{comparisonliterature*}{Comparison to the literature}
\else
\NewEnviron{comparisonliterature}{}{}
\NewEnviron{comparisonliterature*}{}{}
\fi

\newlist{enumalpha}{enumerate}{5}
\setlist[enumalpha]{label=\upshape{(\alph*)}}
\newlist{enumAlpha}{enumerate}{5}
\setlist[enumAlpha]{label=\upshape{(\Alph*)}}
\newlist{enumarab}{enumerate}{5}
\setlist[enumarab]{label=\upshape{(\arabic*)}}
\newlist{enumroman}{enumerate}{5}
\setlist[enumroman]{label=\upshape{(\roman*)}}
\newlist{enumRoman}{enumerate}{5}
\setlist[enumRoman]{label=\upshape{(\Roman*)}}
\newlist{enumgreek}{enumerate}{5}
\setlist[enumgreek]{label=\upshape{(\greekalpha*)}}
\newlist{enumGreek}{enumerate}{5}
\setlist[enumGreek]{label=\upshape{(\Greekalpha*)}}
\newlist{enumC}{enumerate}{5}
\setlist[enumC]{label=\upshape{(C.\roman*)}}

\author[B.~Ammann]{Bernd Ammann} \address{B. Ammann, Fakult\"at f\"ur
  Mathematik, Universit\"at Regensburg, Regensburg, Germany}
\email{bernd.ammann@mathematik.uni-regensburg.de}
\author[C.~L\"oh]{Clara L\"oh} \address{C. L\"oh, Fakult\"at f\"ur
  Mathematik, Universit\"at Regensburg, Regensburg, Germany}
\email{clara.loeh@mathematik.uni-regensburg.de}

\newcommand\arxiv[1]{\href{https://www.arxiv.org/abs/#1}{arXiv:~#1}}

\newcommand\doi[1]{\href{https://doi.org/#1}{\darkblue{DOI:~#1}}}

\newlength{\keyheightlength}

\newcommand{\stellle}[1]{\mathchoice%
  {\mbox{\lower.15\keyheightlength\hbox{\upshape{|}}}_{#1}}%
  {\mbox{\lower.12\keyheightlength\hbox{\upshape{|}}}_{#1}}%
  {\mbox{\lower.17\keyheightlength\hbox{$\scriptstyle|$}}_{#1}}%
  {\mbox{\lower.09\keyheightlength\hbox{\upshape{|}}}_{#1}}%
}
\newcommand{\steLLe}[1]{\mathchoice%
  {\lower.20\keyheightlength\hbox{|}\lower.38\keyheightlength\hbox{$\scriptstyle #1$}}%
  {\lower.16\keyheightlength\hbox{|}\lower.32\keyheightlength\hbox{$\scriptstyle #1$}}%
  {\lower.15\keyheightlength\hbox{$\scriptstyle|$}\lower.28\keyheightlength\hbox{$\scriptscriptstyle #1$}}%
  {\lower.13\keyheightlength\hbox{$\scriptscriptstyle|$}\lower.26\keyheightlength\hbox{$\scriptscriptstyle #1$}}%
}
\newcommand{\stelle}[1]{\mathchoice%
  {\lower.27\keyheightlength\hbox{\big|}\lower.57\keyheightlength\hbox{$\scriptstyle #1$}}%
  {\lower.24\keyheightlength\hbox{\big|}\lower.48\keyheightlength\hbox{$\scriptstyle #1$}}%
  {\lower.21\keyheightlength\hbox{|}\lower.42\keyheightlength\hbox{$\scriptscriptstyle #1$}}%
  {\lower.18\keyheightlength\hbox{|}\lower.39\keyheightlength\hbox{$\scriptscriptstyle #1$}}%
}
\newcommand{\Stelle}[1]{\mathchoice%
  {\lower.39\keyheightlength\hbox{\Big|}\lower.74\keyheightlength\hbox{$\scriptstyle #1$}}%
  {\lower.33\keyheightlength\hbox{\big|}\lower.66\keyheightlength\hbox{$\scriptstyle #1$}}%
  {\lower.30\keyheightlength\hbox{\big|}\lower.60\keyheightlength\hbox{$\scriptscriptstyle #1$}}%
  {\lower.25\keyheightlength\hbox{\big|}\lower.51\keyheightlength\hbox{$\scriptscriptstyle #1$}}%
}
\newcommand{\sstelle}[1]{\mathchoice%
  {\lower.39\keyheightlength\hbox{\bigg|}\lower1.2\keyheightlength\hbox{$\scriptstyle #1$}}%
  {\lower.32\keyheightlength\hbox{\Big|}\lower.74\keyheightlength\hbox{$\scriptstyle #1$}}%
  {\lower.32\keyheightlength\hbox{\big|}\lower.75\keyheightlength\hbox{$\scriptstyle #1$}}%
  {\lower.29\keyheightlength\hbox{\big|}\lower.63\keyheightlength\hbox{$\scriptstyle #1$}}%
}
\newcommand{\SStelle}[1]{\mathchoice%
  {\lower.40\keyheightlength\hbox{\Bigg|}\lower1.35\keyheightlength\hbox{$\scriptstyle #1$}}%
  {\lower.37\keyheightlength\hbox{\bigg|}\lower.98\keyheightlength\hbox{$\scriptstyle #1$}}%
  {\lower.35\keyheightlength\hbox{\Big|}\lower.90\keyheightlength\hbox{$\scriptscriptstyle #1$}}%
  {\lower.30\keyheightlength\hbox{\big|}\lower.72\keyheightlength\hbox{$\scriptscriptstyle #1$}}%
}
\date\today

\newcommand{\meineurl}[1]{\href{#1}{(URL)}}

\newcommand\eg{e.\kern2pt g.,\ \ignorespaces}
\newcommand\ie{i.\kern2pt e.,\ \ignorespaces}

\newcounter{proofstep}

\newcommand{\case}[2]{\par\leftskip0cm\noindent\emph{Case #1: }\emph{#2}\par\setlength{\parindent}{0cm}\leftskip\caseindent\ignorespaces}
\newlength{\caseindent}
\makeatletter
\newcommand{\xdashrightarrow}[2][]{\ext@arrow 0359\rightarrowfill@@{#1}{#2}}
\def\rightarrowfill@@{\arrowfill@@\relax\relbar\rightarrow}
\def\arrowfill@@#1#2#3#4{%
  $\m@th\thickmuskip0mu\medmuskip\thickmuskip\thinmuskip\thickmuskip
   \relax#4#1
   \xleaders\hbox{$#4#2$}\hfill
   #3$%
}
\makeatother
\newcommand{\dashto}{\xdashrightarrow{\kern10mm}}

\let\itemref\ref
\newcommand\ititemref[1]{\ref{#1}}
\newcommand\Ad[1]{Ad~#1:}
\newcommand\Adref[1]{\Ad{\ititemref{#1}}}

\def\qand{%
  \quad\text{and}\quad}

\let\tab\=
\let\=\undefined
\newcommand\={\,=\,}
\let\N\NN
\let\Q\QQ
\let\R\RR
\let\Z\ZZ
\let\C\CC
\def\fa#1{%
  \forall_{#1}\;\;\;}
\def\exi#1{%
  \exists_{#1}\;\;\;}

\newcommand\HZR[1][M]{H_1(#1;\ZZ)_{\RR}}

\DeclareMathOperator{\derham}{dR}
\newcommand\HdR{H_{\derham}^1(M)}

\DeclareMathOperator{\stable}{st}
\DeclareMathOperator{\sign}{sign}
\makeatletter
\newcommand\snorm{\bBigg@{0.8}}

\makeatother

\newcommand\innorm[4][flex]{\csname #1l\endcsname\|#2\csname#1r\endcsname\|_{#3}^{#4}}
\newcommand\onlynorm[2][flex]{\csname #1l\endcsname\|#2\csname#1r\endcsname\|}
\newcommand\stabnorm[2][flex]{\innorm[#1]{#2}{\stable}{}}
\newcommand\stabnormdual[2][flex]{\innorm[#1]{#2}{\stable}{*}}
\newcommand\norm[2][flex]{\onlynorm[#1]{#2}}

\newcommand\groupnorm[1]{N(#1)}
\newcommand\gnorm[3][flex]{\innorm[#1]{#2}{g}{#3}}

\newcommand\Cinftypw{C^\infty_{\mathrm{pw}}}
\newcommand\pscurve{$\Cinftypw$-curve\ignorespaces}
\newcommand\pscurves{$\Cinftypw$-curves\ignorespaces}

\def\ucov#1{\widetilde{#1}}
\def\ucoeverle#1{\tilde{#1}}
\newcommand{\homcovlower}[1]{\mathchoice%
  {\mbox{\lower.25\keyheightlength\hbox{$\textstyle #1$}}}%
  {\mbox{\lower.2\keyheightlength\hbox{$\textstyle #1$}}}%
  {\mbox{\lower.09\keyheightlength\hbox{$\scriptstyle #1$}}}%
  {\mbox{\lower.05\keyheightlength\hbox{$\scriptscriptstyle #1$}}}%
}
\newcommand{\homcovLower}[1]{\mathchoice%
  {\mbox{\lower.5\keyheightlength\hbox{$\textstyle #1$}}}%
  {\mbox{\lower.4\keyheightlength\hbox{$\textstyle #1$}}}%
  {\mbox{\lower.18\keyheightlength\hbox{$\scriptstyle #1$}}}%
  {\mbox{\lower.1\keyheightlength\hbox{$\scriptscriptstyle #1$}}}%
}

\newcommand\homZcov[1]{{\widehat{#1}^{\,\homcovlower\ZZ}}}
\newcommand\homRcov[1]{{\widehat{#1}^{\,\homcovlower\RR}}}
\newcommand\homZcovtorus[1]{{\widehat{T^{#1}}^{\,\homcovLower\ZZ}}}
\newcommand\homRcovtorus[1]{{\widehat{T^{#1}}^{\,\homcovLower\RR}}}
\newcommand\homcov[1]{\widehat{#1}}

\newcommand\rest[2][0]{\steLLe{[#1,#2]}}
\newcommand\hsi[2][i]{h\bigl(\sigma_{#1}\rest{#2}\bigr)}
\newcommand\hsirel[2][0]{h\bigl(\sigma_i\rest[#1]{#2}\bigr)}

\DeclareMathOperator{\Eop}{E}
\DeclareMathOperator{\VEop}{VE}
\newcommand\Emin{\Eop_{\min}}
\newcommand\VEmin{\VEop_{\min}}
\DeclareMathOperator{\CS}{CS}

\DeclareMathOperator{\mix}{mix}
\newcommand\Apm{A^{+\cup -}}
\newcommand\Atot{A_{\mix}}

\newcommand\freehom[2][flex]{\csname #1l\endcsname[#2\csname#1r\endcsname]_{\mathrm{free}}}
\newcommand\freehomn[1]{[#1]_{\mathrm{free}}}
\DeclareMathOperator{\forget}{forget}

\newcommand\gHed{g_{\mathrm{Hed}}}
\newcommand\dHed{d_{\mathrm{Hed}}}
\newcommand\inupto[2][1]{\in\left\{#1,\ldots,#2\right\}}

\begin{document}
\title[A lower bound for minimal geodesics]%
      {A quadratic lower bound for the number\\ of minimal geodesics}

\begin{abstract}
  A minimal geodesic on a Riemannian manifold is a geodesic defined
  on~$\RR$ that lifts to a globally distance minimizing curve on the
  universal covering. Bangert proved that there is a lower bound for
  the number of geometrically distinct minimal geodesics of closed
  Riemannian manifolds that is linear in the first Betti number, using
  the stable norm unit ball on the first homology. We refine this
  method to obtain a quadratic lower bound.  For example, on the
  $3$-dimensional torus with an arbitrary Riemannian metric we improve
  the lower bound from $3$ to~$15$.  We distinguish between different
  types of minimal geodesics and we show that our lower estimate for
  the number of homologically non-homoclinic minimal geodesics is
  sharp.
\end{abstract}

\date{\today}

\maketitle

\section{Introduction}
\subsection{Guiding question of the article}

We study the following classical geometric counting problem
(see Section~\ref{sec:mingeod} for the precise definitions):

\begin{question}\label{question.count}
  How many geometrically distinct minimal geodesics does a
  closed connected Riemannian manifold need to have?
\end{question}

It is well-known that a minimal geodesic exists if and only if
$\pi_1(M)$ is infinite (Proposition~\ref{prop.infinite.pi}).
Bangert proved that a closed connected Riemannian manifold
has at least~$\dim_\R H_1(M;\R)$ minimal geodesics~\cite{bangert:90}. 
For the circle, there clearly is exactly one. For $2$-dimensional
tori~\cite{hedlund:1932} (see Example~\ref{exa.torus.asymptotes}) and closed
connected surfaces of genus at least $2$~\cite{morse:1924,klingenberg:1971}
(see Example~\ref{ex.higher-genus}),
there are always infinitely many. Similarly, if $M$ is a closed
connected Riemannian manifold with non-elementary Gromov-hyperbolic
fundamental group, then on the universal covering~$\ucov M$ each two
different points of the boundary of~$\ucov M$ at infinity can be
joined by a minimal geodesic, and thus there are uncountably many
minimal geodesics
on~$M$~\cite[Section~7.5]{gromov_hyperbolic_groups:87}.

In contrast, in dimension at least~$3$, examples 
with very few minimal geodesics have been
constructed by Hedlund~\cite[Section~9]{hedlund:1932}
and others~\cite[Section~5]{bangert:90}\cite{ammann:diploma}.
Hedlund's work is sometimes misunderstood and
it is erroneously claimed that he constructed metrics with only three 
geometrically distinct minimal geodesics on the $3$-dimensional torus.
Although it is easy to see that Hedlund's
construction does not have this property, it was not clear  
whether one might find other Riemannian metrics on the $3$-torus that have only three minimal geodesics.

We answer this question negatively: We show the existence of at
least~$15$ geometrically distinct minimal geodesics on the $3$-dimensional torus with an arbitrary Riemannian metric, see Corollary~\ref{Cor.n-torus}.

\subsection{A quadratic lower bound}

To state our main result, we introduce the following constants:

\begin{definition}\label{def:Emin}
  Let $b \in \N$ and let $\CS(b)$ be the set of all
  centrally symmetric polytopes in~$\R^b$ with non-empty interior.
  If $P \in \CS(b)$, we write~$V(P)$ and $E(P)$ for the number
  of vertices and edges of~$P$, respectively.
  We set
  \begin{align*}
    \Emin(b)
    & := \min_{P \in \CS(b)} E(P),
    \\
    \VEmin(b)
    & := \min_{P \in \CS(b)} \left(\frac{V(P)}2 + E(P)\right)\,.
  \end{align*}
\end{definition}
Note that the minimal number of vertices is $2b$, and thus $\VEmin(b)\geq  b+ \Emin(b)$. We do not know whether there are $b\in \N$ with $\VEmin(b)> b+ \Emin(b)$, see  
Remark~\ref{rem:csimproved}~\ref{rem:csimproved.i} below. 

\begin{theorem}\label{thm:main}
 Let $M$ be a closed connected Riemannian manifold and let
 $b:= \dim_\R H_1(M;\R)$. Then $M$ admits at least $\VEmin(b)$
 geometrically distinct minimal geodesics.
 Moreover, 
 \[ \VEmin(b) \geq
    \begin{cases}
       2 \cdot b^2 - b& \text{if $b \leq 3$}
       \\
       b^2 + 2\cdot b +1 & \text{if $b \geq 4$}
       .
    \end{cases}  
 \]
\end{theorem}
We briefly outline the argument (which is presented in full detail in
Section~\ref{sec:newmingeod}, Section~\ref{sec:cs}, and
Subsection~\ref{subsec:proofthmmain}). See also Subsections~\ref{subsec:def.convex.bodies} and~\ref{subsec:dual.norm} for precise definitions of exposed edges and exposed points. 
The key tools in the proof are the (closed)
unit ball~$B \subset H_1(M;\R)$ of the stable norm on~$H_1(M;\R)$ and
the construction of minimizing geodesics on the universal
covering~$\ucov M$ of~$M$ from sequences of finite-length minimizing
geodesics.

Bangert proved that every antipodal pair of exposed points
of~$B$ leads to a minimal geodesic whose asymptotic direction
is controlled by the underlying exposed point~\cite[Theorem~4.4]{bangert:90}. 
If~$B$ has infinitely many exposed points, then this already
shows that~$M$ has infinitely many geometrically distinct
minimal geodesics.

If~$B$ has only finitely many exposed points, then~$B$ is a compact
convex centrally symmetric polytope in~$H_1(M;\R)$, and the exposed
points are the vertices of this polytope.  
We refine Bangert's method to construct additional minimal geodesics, whose asymptotic behaviour
is controlled by the exposed edges of~$B$. 
Moreover, we show that these are indeed geometrically distinct. 
To complete the proof, we give a lower bound for the number of vertices and edges in centrally symmetric
compact convex polytopes, which gives the claimed estimate for~$\VEmin$.

\begin{remark}\label{rem:csimproved} ${}$
  \begin{enumalpha}
    \item\label{rem:csimproved.i} 
          We have $\VEmin(b) \geq b + \Emin(b)$ (Corollary~\ref{cor:VEmin}).
          For~$\Emin$, we conjecture the lower bound~$\Emin(b) \geq 2 \cdot b\cdot (b-1)$ for all $b$. 
          The conjecture is equivalent to saying that the minimum in $\Emin$ is attained by the cross-polytope.
          The \emph{cross-polytope} is, by definition, the unit norm ball of the $\ell^1$-norm. 
          The conjecture is equivalent to $\VEmin(b) = b + \Emin(b)$ (Corollary~\ref{cor:VEmin}).
    \item\label{rem:csimproved.ii}  In a forthcoming work~\cite{ammann.loeh_edge}, we will show that $\Emin(b) \geq b\cdot (b+2)$ for even~$b \geq 4$ and that this is sharp for~$b =4$. 
           For odd $b\geq 5$ will show that $\Emin(b) \geq b\cdot (b+2)-1$. 
  \end{enumalpha}
\end{remark}

\enlargethispage{2\baselineskip}

\begin{corollary}\label{Cor.n-torus}
  Let $n \in \N_{\geq 3}$ and let $g$ be a Riemannian metric on the
  $n$-torus~$T^n$.  Then $(T^n,g)$ admits at least
  \[
  \begin{cases}
    15
    & \text{if $n = 3$}
  \\
  (n+1)^2
  & \text{if $n \geq 4$}
  \end{cases}
  \]
  geometrically distinct minimal geodesics.
\end{corollary}
\begin{proof}
  We apply Theorem~\ref{thm:main} and use that $\dim_\R H_1(T^n;\R) = n$. 
\end{proof}

Thus, each Riemannian $3$-torus admits at least~$15$ geometrically
distinct minimal geodesics.  In particular, this applies to the
Hedlund
examples~\cite[Section~9]{hedlund:1932}\cite[Section~5]{bangert:90}
on~$T^3$ and shows that the corresponding exercise in the book by
D.~Burago, Y.~Burago, and
S.~Ivanov~\cite[Exercise~8.5.16]{burago.burago.ivanov:01} is not
solvable.  The latter is already evident from Bangert's study of
Hedlund metrics~\cite[Section~5]{bangert:90}. Our result implies that
there is also no other Riemannian metric on the $3$-torus that admits
only three geometrically distinct minimal geodesics.

Similarly, in combination with
Remark~\ref{rem:csimproved}~\ref{rem:csimproved.ii}, we get for the
$4$-dimensional torus that besides the four minimal geodescis detected
by Bangert~\cite[Theorem~4.4]{bangert:90} there are at least $24$
minimal geodesics of of another type (namely ``homologically
non-homoclinic and $\R$\=/ho\-mo\-lo\-gi\-cal\-ly minimal''); in
particular, the latter minimal geodesics are geometrically distinct
from the ones detected by Bangert. We will also discuss that our bound
is sharp as a lower bound on the number of homologically
non-homoclinic and $\R$\=/ho\-mo\-lo\-gi\-cal\-ly minimal geodesics.

\subsection{A refinement}

Theorem~\ref{thm:main} can be strengthened, by distinguishing between 
different types of minimal geodesics.
This leads to the following refinement, stated in
Theorem~\ref{thm:main.2}.  The terms ``homologically homoclinic'',
``homologically heteroclinic'' and ``homologically non-homoclinic'' will
be defined in Definition~\ref{def.homolo.homocli};
``$\R$\=/ho\-mo\-lo\-gi\-cal\-ly minimal'' geodesics are introduced in
Section~\ref{subsec.homologically-minimal}.  The minimal geodesics
obtained by Bangert are homologically homoclinic and
$\R$\=/ho\-mo\-lo\-gi\-cal\-ly minimal.
 
\begin{theorem}[refinement of Theorem~\ref{thm:main}]\label{thm:main.2}
  Let $M$ be a closed connected Riemannian manifold and let $b :=
  \dim_\R H_1(M;\R)$.  Then, at least one of the following statements
  holds:
  \begin{enumalpha}
  \item\label{thm:main.alt.1}
    There are infinitely many geometrically
    distinct minimal geodesics that are homologically homoclinic and
    $\R$\=/ho\-mo\-lo\-gi\-cal\-ly minimal.
  \item\label{thm:main.alt.2} There are at least~$\Emin(b)$
    geometrically distinct minimal geodesics that are homologically
    non-homoclinic and $\R$\=/ho\-mo\-lo\-gi\-cal\-ly minimal, and
    there are at least~$b$ geometrically distinct geodesics that are
    homologically homoclinic and $\R$\=/ho\-mo\-lo\-gi\-cal\-ly
    minimal.
  \end{enumalpha}
  If the stable norm unit ball is not a polytope, then 
  \ititemref{thm:main.alt.1} is satisfied. If the stable norm
  unit ball is a polytope~$B$, then one can replace~$\Emin(b)$
  by~$E(B)$ and the number of homologically homoclinic minimal geodesics
  is at least~$1/2\cdot V(B)$.
\end{theorem}

We prove this in a stronger form in Proposition~\ref{prop:twomingeod.stronger}
and Section~\ref{subsec:proofthmmain.2}.

\subsection{Discussion of sharpness}\label{subsec:introsharp}

Let us discuss the sharpness of the previous results. We assume in
this section, that the norm ball for the stable norm is a polytope, as
otherwise Bangert's results already provide an infinite number of
geometrically distinct minimal geodesics. Moreover, we assume that we
are not in the case~\ititemref{cor:twomingeod.stat.i} of
Proposition~\ref{prop:twomingeod.stronger} in which there are uncountably
many geometrically distinct homologically homoclinic minimal
geodesics.

Our construction of minimal geodesics makes substantial use of the
fact that we are working with vertices and edges of the stable norm
unit ball. We do not expect that higher-dimensional faces in the boundary
lead to further minimal geodesics.

It is natural to ask whether the lower bounds in
Theorems~\ref{thm:main} and~\ref{thm:main.2} are sharp, i.e., whether
on a given closed connected manifold~$M$ there exist Riemannian
metrics~$g$ whose number of minimal geodesics attains this lower bound.
\begin{itemize}
\item As discussed above, the bound~$\VEmin$ is far from being
  optimal if one considers minimal geodesics on manifolds with
  non-elementary Gromov-hyperbolic fundamental group.
\end{itemize}
The sharpness is better controlled if one only considers
$\R$\=/ho\-mo\-lo\-gi\-cal\-ly minimal geodesics and special
conditions (see Section~\ref{subsec:Hedlund.examples} for 
details on the Hedlund examples):
\begin{itemize}
\item In the special case $\dim M\geq 3$ and~$b=2$, we explain in
  Example~\ref{exam.hedlund.b2} the construction of Hedlund metrics for which the lower bound in
  Theorem~\ref{thm:main.2}~\itemref{thm:main.alt.2} on homologically
  non-homoclinic and $\R$\=/ho\-mo\-lo\-gi\-cal\-ly minimal geodesics is attained. 
  In this case, $\Emin(2)=4$ and the Hedlund metrics defined
  with respect to two generators of~$\HZR$ have precisely four 
  non-homoclinic $\R$\=/ho\-mo\-lo\-gi\-cal\-ly minimal geodesics.
\item Similar constructions are possible in the case~$\dim M\geq 3$
  and~$b=2$, for Hedlund metrics with respect
  to~$\gamma_1,\ldots,\gamma_k\in\HZR$ with~$k\geq 3$ for~$\HZR$. We
  assume that $\gamma_1,\dots, \gamma_k$ span $H_1(M,R)$ (as a vector
  space). Then, the convex hull~$B$ of~$\{\pm\gamma_j\mid j \in
  \{1,\dots,k\}\}$ is a convex centrally symmetric $2\ell$-gon, for
  some~$\ell\in\{2,\ldots,k\}$. An associated Hedlund metric will have
  $B$ as the stable norm unit~ball. By removing some of the~$\gamma_i$ we
  can achieve~$k=\ell$. Then the number of homologically
  non-homoclinic $\R$\=/ho\-mo\-lo\-gi\-cal\-ly minimal geodesics is
  precisely $E(B)=2k$. Thus we obtain a sharp estimate for
  homologically non-homoclinic $\R$\=/ho\-mo\-lo\-gi\-cal\-ly minimal
  geodesics in this situation as well; see
  Example~\ref{exam.hedlund.b2.extended} for details.
\item If we strengthen the hypotheses of the two previous items to
  $\pi_1(M)\cong \Z^2$, then a geodesic is minimal if and only if it
  is $\R$\=/ho\-mo\-lo\-gi\-cal\-ly minimal. Thus, in this case, all
  statements above also hold for ``homlogically non-homoclinic
  minimal geodesics'' instead of
  ``homologically non-homoclinic $\R$\=/ho\-mo\-lo\-gi\-cal\-ly minimal geodesics''.
\item Sharpness is problematic for~$b\geq 3$ and in the homoclinic
  case as the Hedlund examples exhibit unwanted ``side shift
  effects'', e.g., minimal geodesics of the type of 
  Example~\ref{exam.hedlund.b2} \itemref{exam.hedlund.b2.1}
  and~\itemref{exam.hedlund.b2.2}, see also Subsection~\ref{subsec:compar_BR}. 
  This affects the type of geodesics detected by Bangert's method.
\item Even in the case of the $3$-dimensional torus~$T^3$ the minimal
  number of minimal geodesics is unknown. The known types of 
  Hedlund examples have infinitely many minimal geodesics due to
  ``side shift effects'', but we expect that more refined construction
  methods can reduce them to $15$~homoclinic and $12$~heteroclinic
  minimal geodesics. In contrast to this, Bangert's and our lower
  bounds show the existence of three homologically
  homoclinic and $12$~homologically non-homoclinic
  minimal geodesics.  As the heteroclinic geodesics in such Hedlund
  examples are homologically non-homoclinic, the lower bound for the
  number of homologically non-homoclinic minimal geodesics is thus optimal. See again Subsection~\ref{subsec:compar_BR} for related statements.
\end{itemize}

A related problem is to study the consequences of equality in our
estimates.  We recall that a finite number of
minimal geodesics is only possible if the stable norm unit ball is a
polytope (Remark~\ref{rem:finpolytope}).

\begin{theorem}[Section~\ref{subsec:proofEminequal}]\label{thm.equality.gen}
  If $M$ is a closed connected Riemannian manifold with~$b = \dim_\R\  H_1(M;\R)$ and precisely~$b + \Emin(b)$ minimal geodesics, then the
  stable norm unit ball of~$M$ is a cross-polytope and $\Emin(b) = 2
  \cdot b^2 - 2 \cdot b$.  
  Furthermore, then all minimal geodesics are $\R$\=/ho\-mo\-lo\-gi\-cal\-ly minimal and homologically exposed;
  among them $b$ are homologically homoclinic and $\Emin(b)$ are
  homologically heteroclinic.
\end{theorem}

In particular, the first phrase of the theorem implies: If the conjectured
lower bound~$\Emin(b) \geq 2 \cdot b\cdot (b-1)$ does not hold, then
there would be at least~$b + \Emin(b)+1$ minimal geodesics.

\begin{theorem}[Section~\ref{subsec:proofVEminequal}]\label{thm.equality.polytope-fixed}
  Let $M$ be a closed connected Riemannian manifold with~$b = \dim_\R H_1(M;\R)$ whose stable norm unit ball~$B$ is a polytope.  
  We assume that $M$ has exactly $V(B)/2 + E(B)$ geometrically distinct minimal geodesics. 
  Then all of them are $\R$\=/ho\-mo\-lo\-gi\-cal\-ly
  minimal and homologically exposed; among them, $V(B)/2$ are
  homologically homoclinic and $E(B)$ are homologically heteroclinic.
\end{theorem}

\subsection{Open problems}

The following problem seems to be open:

\begin{openquestion}
  Do there exist closed connected Riemannian manifolds~$M$
  that satisfy~$\dim_\R H_1(M;\R) \geq 2$ and have only finitely
  many geometrically distinct minimal geodesics?
\end{openquestion}

As discussed above, we expect that Hedlund examples on the $3$-torus
can be improved in a way such that one can show that they have at
most~$12+15=27$ geometrically distinct minimal geodesics.

Informally speaking, the real strength of the Hedlund
examples~\cite[Section~9]{hedlund:1932}\cite[Section~5]{bangert:90}\cite{ammann:diploma,ammann:97}
lies in the fact that they have ``few'' minimal geodesics in an
asymptotic sense. Many of these examples have an infinite number
of minimal geodesics, however with only finitely many asymptotic
types.

In order to turn this into a precise statement, we define what
it means that a geodesic~$\tau:\R\to M$ is ``asymptotic in the
future/past direction'' to the geodesic~$\gamma:\R\to M$. For
simplicity, we supose that both geodesics are parametrized by arclength.
We define~$T^\pm\gamma$ as
\[ T^\pm\gamma
   \ceq
       \bigl\{\dot\gamma(t)\bigm| t\in \R\bigr\}
       \cup \bigl\{-\dot\gamma(t)\bigm| t\in \R\bigr\}
\,,
\]
which is a subset of the unit tangent bundle~$SM$. We choose a
Riemannian metric on the unit tangent bundle~$SM$, which defines a
distance function~$d^{SM}$ on~$SM$. Finally, for~$v\in SM$ we
write~$d^{SM}(v,T^\pm\gamma)\ceq \inf\{d^{SM}(v,w)\mid w\in
T^\pm\gamma\}$.

We say that $\tau$ is \emph{asymptotic in the future direction
  (resp. past direction)} to~$\gamma$ if $d^{SM}(\dot\tau(t),
T^\pm\gamma)$ converges to~$0$ for~$t\to+\infty$ (resp.\ for~$t\to
-\infty$).

We then obtain the asymptotic version of Question~\ref{question.count}:

\begin{question}
  Let $(M,g)$ be a closed connected Riemannian manifold.  What is the
  minimal number~$\mu=\mu(M,g)$ of minimal
  geodesics~$\gamma_1, \ldots,\gamma_\mu$ on~$M$ such that every other minimal
  geodesic is asymptotic in every direction to
  one of the~$\gamma_i$\;?  We set~$\mu(M,g):=\infty$ if such a finite set
  of minimal geodesics does not exist.
\end{question}

The proof of Bangert's lower estimate for the minimal number of
minimal geodesics~\cite[Theorem~4.4]{bangert:90} immediately gives the
stronger statement
\[ \mu(M,g)\geq \dim_\R  H_1(M;\R)
\,.
\] 
For the Hedlund examples, the equality~$\mu(M,g)=\dim_\R H_1(M;\R)$
is known to be attained in the following cases:
\begin{itemize}
\item for $M=T^3$~\cite[Section~9]{hedlund:1932},
\item for $M=T^n$ with~$n\geq 3$~\cite[Section~5]{bangert:90},
\item for compact quotients of Heisenberg groups~\cite{ammann:diploma}. 
\end{itemize}
We conjecture that on every manifold~$M$ with virtually nilpotent
fundamental group a Hedlund type construction yields a Riemannian
metric~$g$ with~$\mu(M,g)=\dim_\R H_1(M;\R)$; see~\cite{ammann:97} for
further discussion. The construction of such a metric on~$T^3$ also
seems to be the intended goal of the previously mentioned exercise by
Burago, Burago, and
Ivanov~\cite[Exercise~8.5.16]{burago.burago.ivanov:01}.

Summarizing we see that for fundamental groups ``close'' to abelian
groups, the Hedlund examples show that Bangert's bounds are optimal in
the asymptotic sense.  In contrast, one can show that for
non-elementary Gromov-hyperbolic fundamental groups we always
have~$\mu(M,g)=\infty$. For example, on closed manifolds~$(M,g)$ with
nonpositive sectional curvature, all geodesics (defined on~$\R$) are
minimal, and for different and non-antipodal unit vectors~$v$,~$w$
with the same base point, the geodesics $t\mapsto \exp(tv)$ and
$t\mapsto \exp(tw)$ are not asymptotic to each other.

\subsection*{Organisation of this paper}

Section~\ref{sec:mingeod} contains preliminaries on minimal geodesics,
geometric Hurewicz maps, and the Jacobi map, and defines $\R$- and
$\Z$\=/ho\-mo\-lo\-gi\-cal\-ly minimal geodesics.

We recall basics on the stable norm in
Section~\ref{sec:stablenorm}. In Section~\ref{sec:findist}, we show
that the stable norm and the minimal length have finite distance, a
result that essentially follows from D.~Burago's work
\cite{burago:soviet:92}, but that does not seem to be worked out in
the literature.

Homological asymptotes are discussed in
Section~\ref{sec:asymptotes}. The set of terminal/initial asymptotes
will allow us to distinguish different types of minimal geodesics, \eg
homologically homoclinic and homologically heteroclinic ones.

Section~\ref{sec:newmingeod} contains the refined constructions of
minimal geodesics. The counting of vertices/edges in centrally
symmetric polytopes is given in Section~\ref{sec:cs}.
Section~\ref{sec:proofs} contains the proofs of
Theorem~\ref{thm:main}, Theorem~\ref{thm:main.2},
Theorem~\ref{thm.equality.gen}, and
Theorem~\ref{thm.equality.polytope-fixed}.
A variety of examples in dimension~$2$ and constructions of Hedlund type metrics are collected in Section~\ref{sec:examples}.  We will report in Subsection~\ref{subsec:Lagrangian} about alternative methods to detect minimal geodesics using results from Lagrangian mechanics.
Finally we will include a comparison to work by Bolotin and Rabinowitz, see Subsection~\ref{subsec:compar_BR}.

For the sake of completeness, we recall basics on minimal geodesics in
Appendix~\ref{app.minimzing.geodesics}\ifwithappendixstable{} and we spell out the
relation between the stable norm and the mass norm on~$H_1(\args;\R)$
in Appendix~\ref{app:stable-norm}\fi{}.
\ifwithappendixtwodim  In Appendix~\ref{app:twodim} we summarize additional well-known facts about minimal geodesics on surfaces; we do not intend to include this in the published version.
\fi

Readers interested only in the basic Theorem~\ref{thm:main} need not read the full article: After familiarizing themselves 
with the basic concepts and results in Sections~\ref{sec:mingeod} and~\ref{sec:stablenorm}, 
they should read Section~\ref{sec:findist}, Subsections \ref{subsec:def.hom.asymp}, \ref{subsec:classical-elem-asymp} and~\ref{subsec:asympt.hom.min}, Section~\ref{sec:newmingeod} until Corollary~\ref{cor:twomingeod}, Proposition~\ref{prop:cslowerbound} 
and finally the proof of Theorem~\ref{thm:main} in Subsection~\ref{subsec:proofthmmain}.

\subsection*{Acknowledgements}
Bernd Ammann thanks Victor Bangert for bringing his attention to this
problem and for many interesting discussions. We also thank an anonymous referee for pointing us to the Lagrangian appraoch sketched in Subsection~\ref{subsec:Lagrangian}.

\section{Minimal geodesics, curves, and homology}\label{sec:mingeod}

Before we recall the notion of minimal geodesics, the construction of
geometric Hurewicz maps on curves, and the Jacobi map, we will fix
some conventions.

\subsection{Conventions}

We use $\N := \{1,2,\dots\}$ and $\N_0 := \{0\} \cup \N$. All
manifolds in this article are assumed to be non-empty.

Geodesics are always assumed to be parametrized by arclength.
We use the convention that \emph{a curve in a (smooth) manifold~$M$} is
a smooth map~$\gamma:I\to M$ from an interval~$I$ (usually of positive length)
to~$M$. An exeception from this notion is the notion of \emph{piecewise
smooth curves}, which are defined as continuous maps~$\gamma:I\to M$
for which finitely many real numbers~$a_1<a_2<\ldots <a_k$ with~$k\in
\NN_0$ exist such that we have the following, using $a_0\ceq-\infty$ and $a_{k+1}\ceq +\infty$:
\[\text{for all~$i\in \{0,1,\ldots,k\}$, the
curve~$\gamma_{I\cap [a_i,a_{i+1}]}$ is smooth.}\] 
In particular, every curve is a
piecewise smooth curve, but not vice versa. We use the abbreviation
``\pscurve'' for piecewise smooth curves. 
If -- in exceptional cases -- curves may have lower regularity (\eg just continuous), this will be mentioned explicitly.
A curve $\gamma$ is \emph{regular} if $\dot\gamma(t)\neq 0$ for all $t$ in the domain.

A \emph{$\Cinftypw$-loop} in~$M$ is a $\Cinftypw$-path~$\gamma:[a,b]\to
M$ with~$\gamma(a)=\gamma(b)$. A \emph{loop} is defined as a smooth
$\Cinftypw$-loop. Note that for a loop we do not require~$\dot\gamma(a)=\dot\gamma(b)$. 
A \emph{geodesic loop} is a loop that is a geodesic. 
A \emph{closed curve} is a smooth loop~$\gamma\colon [a,b]\to M$ that extends to a (smooth) periodic curve~$\gamma\colon \R\to M$ of period length~$b-a$. 
We will often view them as smooth maps~$S^1\to M$. A closed curve is called \emph{simple} if $\gamma\stelle{[a,b]}$ is injective.
If a closed curve~$\gamma\colon [a,b]\to M$ is a geodesic, it is called a \emph{closed geodesic}, and this is equivalent to saying, that 
$\gamma$ is a geodesic loop with~$\dot\gamma(a)=\dot\gamma(b)$. 
A \emph{broken geodesic}
is defined as a \pscurve{} whose smooth pieces are geodesics.

All differential forms are assumed to be smooth. 
In a normed space the ``unit ball'' always denotes the \emph{closed} ball of radius~$1$ centered in~$0$.

\subsection{Based and free homotopies of loops}\label{subsec:based-and-free-homotopies}

Let $X$ be a path-connected topological space and let~$x_0 \in X$.
As a set, $\pi_1(X,x_0)$ is the set~$[(S^1,1), (X,x_0)]_*$ of
(continuous) loops in~$X$ based at~$x_0$ modulo pointed (continuous)
homotopies. We will denote the based homotopy class of a based loop~$\gamma$
by~$[\gamma]_*$.

For geometric considerations it is often better to work with
\emph{free homotopies} of loops instead. Let $[S^1,X]$ be 
the set of (continuous) loops in~$X$ modulo (continuous) homotopies;
in this case, we do not require constraints on the basepoint --
neither for the loops, nor for the homotopies. The free homotopy
class of a loop~$\gamma$ is denoted by~$[\gamma]$.
Forgetting the basepoint defines a well-defined map~$\forget_{x_0}
\colon \pi_1(X,x_0) \to [S^1,X]$, given by~$[\gamma]_* \mapsto [\gamma]$.
As $X$ is path-connected, the map~$\forget_{x_0}$ is surjective.

Conjugacy classes of~$\pi_1(X,x_0)$ will sometimes be called
\emph{$\pi_1$-conjugacy classes}. The $\forget_{x_0}$-preimages of
singletons in~$[S^1,X]$ are precisely the $\pi_1$-conjugacy classes
in~$\pi_1(X,x_0)$. If $\gamma$ is a $C^0$-loop, then
$(\forget_{x_0})^{-1}(\{[\gamma]\})$ is the \emph{$\pi_1$-conjugacy
  class represented by~$\gamma$}.

If $X$ is a connected smooth manifold, then we may replace continuous
curves and homotopies by curves and homotopies of regularity $C^k$,
$C^\infty$ or~$\Cinftypw$ without changing the results.

\subsection{Minimal geodesics}

Minimal geodesics are Riemannian geodesic lines that lift to
``metric'' geodesic lines on the Riemannian universal covering: 

\begin{definition}[minimizing/minimal geodesic]
  Let $M$ be a closed connected Riemannian manifold
  and let $\ucov M$ be its Riemannian universal covering.
  \begin{itemize}
  \item
    A (Riemannian) geodesic~$\gamma \colon I \to \ucov M$
    on~$\ucov M$, defined on a closed interval~$I \subset \R$
    and parametrized by arclength, is \emph{minimizing} if it
    is metrically isometric, i.e., if
    \[ \fa{s,t \in I} d\bigl(\gamma(t), \gamma(s)\bigr) = |t - s|
    \]
    holds, where $d \colon \ucov M \times \ucov M \to \R_{\geq 0}$
    denotes the metric induced by the Riemannian metric on~$\ucov M$.
  \item
    A geodesic~$\gamma \colon I \to M$ on~$M$ is \emph{minimal}
    if~$I = \R$ and one (whence every) lift of~$\gamma$ to~$\ucov M$
    is minimizing.
  \item
    Geodesics on~$M$ are \emph{geometrically equivalent} if there
    exists a reparametrization that transforms one into the other.
    Otherwise they are called \emph{geometrically distinct}.
  \end{itemize}
\end{definition}

\begin{example}
  Let $M := S^1 \times N$, where $(N,h)$ is a simply connected closed
  Riemannian manifold. We view $S^1$ as the unit circle in~$\C$,
  parametrized by~$t\mapsto \exp(it)$. We choose a smooth function
  $f\colon N\to [1,\infty)$ that attains its minimal value~$1$ only in
  single point~$p\in N$. We define the Riemannian metric~$g= f^2
  dt^2 + h$ on~$M$, i.e., the warped product metric with warping
  function~$f$. One easily sees that $t\mapsto
  \bigl(\exp(it),p\bigr)$ is a minimal geodesic in~$(M,g)$, and that
  this is up to geometric equivalence the only one. By construction,
  $\dim_\R H_1(M;\R) = 1$.
\end{example}

We will construct minimal geodesics with the following lemma: 

\begin{lemma}[limiting geodesics]\label{lemma:limitinggeodesic}
  Let $M$ be a closed connected Riemannian manifold, and let $\wihat
  M\to M$ be a Riemannian covering. Let $(\sigma_i \colon [a_i,b_i]
  \to \wihat M)_{i \in \N}$ be a sequence of minimizing geodesics
  on~$\wihat M$ with the following properties:
  \begin{itemize}
  \item We have~$\lim_{i \to \infty} a_i = -\infty$
    and $\lim_{i \to\infty} b_i = \infty$.
  \item The sequence~$(\dot\sigma_i(0))_{i \in \N}$
    converges to some~$v_\infty \in T\wihat M$.
  \end{itemize} 
  Then
  \begin{align*}
    \sigma_\infty \colon \R
    & \to \wihat M
    \\
    t
    & \mapsto \exp_{\wihat M} (t \cdot v_\infty)
  \end{align*}
  is a minimizing geodesic and the $(\sigma_i)_{i \in \N}$ converge on
  each compact interval uniformly (in the $C^\infty$-topology)
  to~$\sigma_\infty$.
\end{lemma}

A proof is provided in Appendix~\ref{app.minimzing.geodesics}.

\subsection{Curves and homology}\label{subsec:curveshom}

The geometric Hurewicz map is defined via integration.  In contrast
with the topological Hurewicz map for continuous \emph{loops}, the
geometric Hurewicz map depends on how we represent de Rham cohomology
classes by forms. However, this dependence will have no relevance
for our following arguments. By~$\HZR \subset H_1(M;\R)$, we denote
the integral part of~$H_1(M;\R)$, i.e., the image of~$H_1(M;\Z)$ under
the change-of-coefficients map from $\Z$ to~$\R$. Then $\HZR$ is a
lattice in~$H_1(M;\R)$.

\begin{setup}\label{setup:bases}
  Let $M$ be a closed connected Riemannian manifold and let $b\ceq \dim_\R H_1(M;\R)$.  
  Let~$(\beta_1, \dots, \beta_b)$ be a
  $\Z$-module basis of~$\HZR$, which is also an $\R$-basis of
  $H_1(M;\R)$.  Let $\alpha^1, \dots, \alpha^b$ be closed $1$-forms
  on~$M$ such that $([\alpha^1], \dots, [\alpha^b])$ is the basis of~$\HdR$
  dual to~$(\beta_1,\dots, \beta_b)$, viewed
  as a vector space basis for~$H_1(M;\R)$.
\end{setup}

\begin{definition}[geometric Hurewicz map]
  In the situation of Setup~\ref{setup:bases},
  for each \pscurve{}~$\gamma \colon I \to M$, we
  let $h(\gamma)$ denote the unique element of~$H_1(M;\R)$
  with
  \[ \fa{j \in \{1,\dots,b\}} \int_\gamma \alpha^j
  =  \bigl\langle [\alpha^j], h(\gamma)\bigr\rangle
  \,.
  \]
  Here, $\langle \args,\!\args\rangle$ denotes
  the integration pairing between $\HdR$ and~$H_1(M;\R)$.
\end{definition}

\begin{remark}[additivity]\label{rem:hurewiczproperties}
  In the situation of Setup~\ref{setup:bases}, using the additivity of
  integration along curves, we obtain: If $\gamma$ and $\gamma'$ are
  \pscurves{} on~$M$ such that the endpoint of $\gamma$ is the start
  point of $\gamma'$, then
  \[ h(\gamma * \gamma') = h(\gamma) + h(\gamma')
  \qand
  h (\overline \gamma) = - h(\gamma)
  \,,
  \]
  where ``$*$'' denotes the concatenation of curves (if well-defined)
  and ``$\overline{\phantom{\gamma}}$'' denotes the orientation reversal of
  curves.
\end{remark}

\begin{remark}[loops]\label{rem:hurewiczint}
  In the situation of Setup~\ref{setup:bases}, the integration duality
  between~$H_1(M;\R)$ and~$\HdR$ shows: If $\gamma$ is a
  $\Cinftypw$-loop in~$M$, then $h(\gamma)$ is the singular homology
  class represented by the loop~$\gamma$; in particular, $h(\gamma)$
  then lies in~$\HZR$.  Furthermore, we have for every closed $1$-form
  $\eta\in \Omega^1(M)$:
  \[ \int_\gamma \eta
  =  \bigl\langle [\eta], h(\gamma)\bigr\rangle
  \,.
  \]
\end{remark}

\begin{remark}\label{rem:hurewicznonint}
  Again we assume the situation of Setup~\ref{setup:bases}. If
  $\eta\in\Omega^1(M)$ is a linear combination of
  $\alpha^1,\ldots,\alpha^b$, then we have for every
  \pscurve{} $\gamma\colon [a,b]\to M$:
    \[ \int_\gamma \eta
    =  \bigl\langle [\eta], h(\gamma)\bigr\rangle
    \,.
    \]
However, in general, this property does \emph{not} hold for arbitrary
closed $1$-forms~$\eta$ when~$\gamma$ is not a loop.  For example, let
$f:M\to \R$ be smooth with $f(\gamma(a))=0$ and $f(\gamma(b))=1$, and
let $\eta\ceq df$. Then, $[\eta]=0$ but
\[ \int_\gamma \eta=1\neq 0
=  \bigl\langle [\eta], h(\gamma)\bigr\rangle
\,.
\]
\end{remark}

\begin{definition}[Jacobi map]\label{def:jacobi}
  In the situation of Setup~\ref{setup:bases}, in addition,
  we pick a basepoint~$x_0$ in~$M$ and a lift~$\ucov x_0 \in \ucov M$
  of~$x_0$.
  For~$j \in \{1,\dots,b\}$, let $\ucov \alpha^j$ denote the
  pull-back of~$\alpha^j$ to~$\ucov M$. Then, there is a
  unique smooth function~$J^j \colon \ucov M \to \R$ with
  \[ \upd J^j = \ucov \alpha^j
  \qand
  J^j(\ucov x_0) = 0
  \,.
  \]
  The \emph{Jacobi map} is defined as the sum
  \[ J \ceq \sum_{j=1}^b J^j \cdot \beta_j \colon \ucov M \to H_1(M;\R)
  \,.
  \]
\end{definition}

\begin{remark}\label{rem:jacobiproperties}
  In the situation of Definition~\ref{def:jacobi}, the
  Jacobi map~$J \colon \ucov M \to H_1(M;\R)$ is smooth.
  Moreover, the Jacobi map is related to the geometric Hurewicz map
  as follows: If $\gamma \colon [a,b] \to M$ is a smooth
  curve and $\ucov \gamma \colon [a,b] \to \ucov M$ is
  a lift to~$\ucov M$, then
  \[ h(\gamma) = J\bigl( \ucov \gamma(b)\bigr) - J \bigl( \ucov \gamma(a)\bigr)
  \,.
  \]
  Indeed, for all~$j \in \{1,\dots,b\}$, we have
  \begin{align*}
    \bigl\langle [\alpha^j], h(\gamma) \bigr\rangle
    & = \int_\gamma \alpha^j
    = \int_{\ucoeverle \gamma} \ucov \alpha^j = \int_{\ucoeverle \gamma} d J_j
    = J_j \bigl(\ucov \gamma(b)\bigr) - J_j \bigl(\ucov \gamma(a)\bigr)
    \\
    & = \Bigl\langle [\alpha^j], J\bigl(\ucov \gamma(b)\bigr) - J\bigl(\ucov \gamma(a)\big)\Bigr\rangle\,.
  \end{align*}
\end{remark}

\begin{remark}
  It follows from Remark~\ref{rem:hurewiczint} and the additivity in
  Remark~\ref{rem:hurewiczproperties} that the Jacobi map $J \colon
  \ucov M \to H_1(M;\R)$ descends to to a map $\slashed J: M\to
  H_1(M;\R)/\HZR$. The map $\slashed J$ is Gromov's
  ``Jacobi mapping''~\cite[4.21]{gromov:99}
  and $H_1(M;\R)/\HZR$ is called the ``Jacobi variety'';
  literally, only the case that $H_1(M;\Z)$ is torsionfree is treated,
  but the extension to the general case is straightforward. 
\end{remark}

\subsection{$\R$\=/ho\-mo\-lo\-gi\-cal\-ly minimal and $\Z$\=/ho\-mo\-lo\-gi\-cal\-ly minimal geodesics}\label{subsec.homologically-minimal}

We will now introduce two stronger versions of minimality for
geodesics.

For this purpose let $\Gamma\ceq \pi_1(M)$ be the fundamental group of
a closed connected Riemannian manifold~$(M,g)$ and let
$[\Gamma,\Gamma]$ be its commutator subgroup. 
The Hurewicz map factors to an isomorphism~$\Gamma/[\Gamma,\Gamma]\to H_1(M;\Z)$. 
Thus $\homZcov M\ceq \ucov M/[\Gamma,\Gamma]$ is a normal covering of $M$ with
deck transformation group (canonically isomorphic
to)~$H_1(M;\Z)$. 
Similarly, if $T$ is the torsion subgroup of
$H_1(M;\Z)$, then the change-of-coefficients map $H_1(M;\Z)\to
H_1(M;\R)$ induces an isomorphism $H_1(M;\Z)/T\to \HZR$.  
The deck transformation group of the normal covering~$\homRcov M\ceq \homZcov M /
T$ of~$M$ is (canonically isomorphic to)~$\HZR$.  
The Jacobi map $J\colon \ucov
M\to H_1(M;\R)$ factors to maps $\homZcov J\colon\homZcov M\to
H_1(M;\R)$ and $\homRcov J\colon\homRcov M\to H_1(M;\R)$. 
We thus obtain the following sequence of covering maps:
\[ \ucov M
\to \homZcov M
\to \homRcov M
\to M\,.
\]

\begin{definition}
[$\R$\=/ho\-mo\-lo\-gi\-cal\-ly and $\Z$\=/ho\-mo\-lo\-gi\-cal\-ly minimal geodescis]\label{def:homologmin}
\hfil
Let~$M$ be a closed connected Riemannian manifold. 
\begin{itemize}
\item
    A geodesic~$\gamma \colon \R \to M$ on~$M$ is \emph{$\Z$\=/ho\-mo\-lo\-gi\-cal\-ly minimal}
    if~one (whence every) lift of~$\gamma$ to~$\homZcov M$
    is minimizing.
\item
    A geodesic~$\gamma \colon \R \to M$ on~$M$ is \emph{$\R$\=/ho\-mo\-lo\-gi\-cal\-ly minimal}
    if~one (whence every) lift of~$\gamma$ to~$\homRcov M$
    is minimizing.
\end{itemize}
\end{definition}
For a geodesic $\R\to M$, ``$\R$\=/ho\-mo\-lo\-gi\-cal\-ly minimal''
implies ``$\Z$\=/ho\-mo\-lo\-gi\-cal\-ly minimal'', and the latter one
implies ``minimal''.  The minimal geodesics constructed in
Theorem~\ref{thm:main.2} (and thus also the ones in
Theorem~\ref{thm:main}) are in fact ``$\R$\=/ho\-mo\-lo\-gi\-cal\-ly
minimal''.  Also, the slightly larger class of
$\Z$\=/ho\-mo\-lo\-gi\-cal\-ly minimal geodesics has many properties of
minimal geodesics, especially the ones related to the stable norm,
hold in this larger class (see \eg
Proposition~\ref{prop.hom.min.sphere}).

\begin{comparisonliterature*}
  Gromov calls~$\homRcov M$ the \emph{abelian covering} and denotes it
  by~$\witi M_{\mathrm{Ab}}$~\cite[4.22$_+$-4.24$_+$]{gromov:99}.
  Bangert considers~$\homZcov M$ and denotes it by~$\bar
  M$~\cite[Sec.~3]{bangert:90}.
\end{comparisonliterature*}

\section{The stable norm}\label{sec:stablenorm}

The stable norm on first homology is the homogenisation of the
quasi-norm given by the minimal length of representing curves. We
recall this construction and basic estimates for the stable norm on
the first real homology group~$H_1(M;\R)$. Our approach is similar to
the one of Bangert~\cite{bangert:90}, which again relies on similar notions
and statements by Gromov~\cite[Chap.~4.C.]{gromov:99}.

In the next section (Theorem~\ref{thm:stabnormN}), we will show that
the stable norm~$\stabnorm{\argu}$, restricted to the integral
lattice~$\HZR$, coincides up to a bounded error with the minimal
length quasi-norm~$N: \HZR\to\R_{\geq 0}$. This approximation result
is much stronger than the ones proven by Bangert and Gromov.

\subsection{Quasi-norms and their homogenisation}

We recall basic notions and properties for general quasi-norms
on abelian groups.

\begin{definition}[quasi-norm]\label{def:quasinorm}
  Let $V$ be an abelian group. 
  A \emph{quasi-norm} on~$V$ is
  a map~$N \colon V \to \R_{\geq 0}$ with the
  following properties:
  \begin{itemize}
  \item \emph{Positive definiteness.}
    We have for all~$x \in V$ that $N(x) = 0$ if and
    only if~$x = 0$.
  \item \emph{Symmetry.} For all~$x \in V$, we
    have~$N(-x) = N(x)$.
  \item \emph{Quasi-triangle inequality.}
    There exists a~$\Delta \in \R_{\geq 0}$ such that for
    all~$x,y \in V$ we have
    \[ N(x + y) \leq N(x) + N(y) + \Delta
    \,.
    \]
\end{itemize}
There is a least constant~$\Delta_{\min}$, satisfying the inequality
above for all~$x,y\in V$, and this constant will be called the
$\triangle$-discrepancy of~$N$. 
Note that $N(0)=0$ implies~$\Delta_{\min}\geq 0$.

We will occasionally need further properties of quasi-norms, and more
generally of maps~$N \colon V \to \R_{\geq 0}$, where~$V$ is a
$K$-module with~$K \in \{\N_0, \Z,\R\}$:
\begin{itemize}
\item  
  A map~$N \colon V \to \R_{\geq 0}$ is
  \emph{homogeneous over $K$} if we have for all~$x\in V$
  and all~$k \in K$ that
  \[ N(k \cdot x) = |k| \cdot N(x)
  \,.
  \]  
\item  
  The map $N \colon V \to \R_{\geq 0}$ is called \emph{subhomogeneous (over $\N$)},
  if we have for all~$x\in V$ and all~$n \in \N$ that
  \[ N(k \cdot x) \leq |k| \cdot N(x)
  \,.
   \]  
 \item The map $N \colon V \to \R_{\geq 0}$ satisfies the
   \emph{doubling property}, if there exists a~$D \in \R_{\geq 0}$
   such that for all~$x \in V$ we have
   \[ 2 \cdot N(x) \leq N(2 \cdot x) + D
   \,.
  \]
  Again, there is a least constant~$D_{\min}$ for which this inequality
  holds for all~$x\in V$. This least constant will be called the
  \emph{doubling constant of $N$}.
 \end{itemize}
\end{definition}

\begin{proposition}[homogenisation of quasi-norms]
  \label{prop:homog}
  Let $N \colon V \to \R_{\geq 0}$ be a quasi-norm with
  $\triangle$-discrepancy~$\Delta$ on the abelian group~$V$. For
  each~$x \in V$, the limit
  \[ \|x\| := \lim_{k \to \infty} \frac{N(k \cdot x)}{k}
  \]
  exists and we have
  \[ \|x\| = \inf_{k \in \N} \frac{N(k\cdot x)+\Delta}{k}
  \,.
  \]
  The map~$\norm\argu \colon V \to \R_{\geq 0}$
  is homogeneous over~$\Z$
  and satisfies the triangle inequality; we call~$\norm\argu$
  the \emph{homogenisation of~$N$}.
  If, additionally, $N$ is subhomogeneous, then 
  \[ \fa{x \in V}
  \|x\| \leq N(x)
  \,.
  \]
\end{proposition}

One may easily construct examples for which~$\norm\argu$ is no longer
positive definite. Positive definiteness is lost, e.g., on all
non-trivial torsion elements of~$V$.

\begin{proof}
  Let $x \in V$. 
  The sequence~$(s_k)_{k \in \N}$ with $s_k\ceq N(k\cdot x) + \Delta$ is subadditive, \ie $s_{k+\ell}\leq s_k+ s_\ell$ for all $k,\ell\in \N$. 
  By the Fekete lemma~\cite[\Paragraph~I.3.1]{polyaszego}, the
  limit~$\lim_{k \to \infty} (N(k\cdot x) + \Delta)/k$ exists and
  \[ \lim_{k \to \infty} \frac{N(k\cdot x) + \Delta}{k}
  = \inf_{k \in \N} \frac{N(k\cdot x) + \Delta}{k}
  \,.
  \]
  In particular, also~$\lim_{k \to \infty} N(k\cdot x)/k$ exists
  and
  \[ \lim_{k \to \infty} \frac{N(k\cdot x)}{k}
  = \inf_{k \in \N} \frac{N(k\cdot x) + \Delta}{k}
  \,.
  \]
  
  By construction, $\norm\argu$ is homogeneous over~$\N$ and
  $\norm\argu$ inherits symmetry from~$N$.
  Therefore, $\norm\argu$ is homogeneous over~$\Z$.

  For the triangle inequality, let $x,y \in V$.
  Then, the quasi-triangle inequality for~$N$ shows that
  \begin{align*}
  \|x+y\|
  & = \lim_{k \to \infty} \frac{N(k\cdot (x+y))}{k}
  \\
  & \leq \lim_{k \to \infty} \frac{N(k \cdot x) + N(k \cdot y) + \Delta}{k}
  = \lim_{k\to\infty} \frac{N(k \cdot x)}k + \lim_{n\to\infty} \frac{N(k \cdot y)}k + 0
  \\
  & = \|x\| + \|y\|\,.
  \end{align*}

  Finally, let $N$ be subhomogeneous. Then,
  \[ \| x \| = \lim_{k \to \infty} \frac{N(k \cdot x)}{k}
  \leq \lim_{k \to \infty} \frac{|k| \cdot N(x)}{k}
  = N(x)
  \,,
  \]
  as claimed.
\end{proof}

\begin{proposition}\label{prop:qndoubling}
  Let $N \colon V \to \R_{\geq 0}$ be a quasi-norm on the abelian
  group~$V$ satisfying the doubling property with doubling constant
  $D$.
  Then the homogenisation~$\|\argu\|$ of~$N$ satisfies
  \[ \fa{x \in V}   N(x) \leq \|x\| + D
  \,.
  \]
\end{proposition}
\begin{proof}
  Let $x \in V$. From the doubling property, we inductively
  obtain
  \[ \fa{k \in \N}
     2^k \cdot N(x)
     \leq N(2^k \cdot x) + \sum_{j=0}^{k-1} 2^j \cdot D
     \leq N(2^k \cdot x) + 2^k \cdot D
     \,.
  \]
  Taking the limit over the subsequence of binary powers
  shows therefore that
  \begin{align*}
    N(x)
    & = \lim_{k\to\infty} \frac{2^k \cdot N(x)}{2^k}
    \leq \lim_{k \to \infty} \frac{N(2^k \cdot x) + 2^k \cdot D}{2^k}
    \\
    & = \|x\| + D\,.
    \qedhere
  \end{align*}
\end{proof}

For the sake of completeness, we record a discrete
example:

\begin{remark}[word lengths on finitely generated abelian
    groups~\protect{\cite[Exc.~8.5.8]{burago.burago.ivanov:01}}]
  \label{rem:wordlength}
  \hfil\\   
  Let $V$ be a finitely generated free abelian group, let $S \subset
  V$ be a finite generating set of~$V$, and let $N$ be the word norm
  on~$V$ associated with~$S$. Then $N$ is subhomogeneous and $N$ has
  the doubling property: Indeed, let $k := \#S$ and $S = \{s_1,\dots,
  s_k\}$.  For~$x \in V$ we choose an $N$-minimal
  $S$-representation~$2 \cdot x = \sum_{j=1}^k a_j \cdot s_j$ of~$2
  \cdot x$ with~$a_1,\dots, a_k \in \Z$.  We write~$a_j = 2 \cdot b_j
  + r_j$ with~$b_j \in \Z$ and $r_j \in \{0,1\}$ and consider the
  quasi-half
  \[ x' := \sum_{j=1}^k b_j \cdot s_j
  \]
  of~$2 \cdot x$.
  We obtain
  \begin{align*}
  2 \cdot N(x)
  & \leq 2 \cdot N(x') + 2 \cdot N(x-x')
  \\
  & \leq 2 \cdot \sum_{j=1}^k |b_j| + 2 \cdot k
  = \sum_{j=1}^k |a_j - r_j| + 2 \cdot k
  \leq \sum_{j=1}^k |a_j| + k + 2 \cdot k
  \\
  & = N(2\cdot x) + 3 \cdot \# S\,,
  \end{align*}
  which proves that $N$ has the doubling property.
  Proposition~\ref{prop:qndoubling}
  and Proposition~\ref{prop:homog} show that $N$ and its
  homogenisation have uniformly finite distance on~$V$.

  This argument can be extended to word norms on all finitely
  generated abelian groups. However, in contrast with the
  homogenisation process for quasi-morphisms, word norms on
  general amenable groups need not have finite distance to their
  homogenisation: For example, the centre of the integral
  $3$-dimensional Heisenberg group is distorted with respect
  to word norms.
\end{remark}

\subsection{The minimal length quasi-norm}

We first fix notation for the minimal length
of representing curves, which defines a quasi-norm, and then introduce the
stable norm via homogenisation and extension.

\begin{definition}\label{def:N}
  Let $M$ be a closed connected Riemannian manifold. We define the
  \emph{minimal length map} as the function
  \begin{align*}
    N \colon \HZR & \to \R_{\geq 0}
    \\
    x & \mapsto
    \inf
    \bigl\{ \maL (\gamma)
    \bigm| \text{$\gamma$ is a loop on~$M$ with~$h(\gamma) = x$}
    \bigr\}\,,
  \end{align*}
  where $\maL(\gamma)$ denotes the length of the curve~$\gamma$.
\end{definition}

\begin{remark}\label{rem:defN}
  As this definition is based on loops, only the canonical part of the
  Hurewicz map is used.  In the definition, we used our convention
  that loops are supposed to be smooth. However, let us explain why
  the value~$N(x)$ remains unchanged if we weaken this condition to
  $\Cinftypw$-curves, to (piecewise) $C^1$-curves, or even to
  $C^0$-curves. It follows from standard smoothing techniques for
  loops~\eg \cite[\Paragraph 16 and~17]{milnor::morse-theory}
  that the inclusion of the space of smooth loops into the space curves
  of regularity $\Cinftypw$, (piecewise) $C^1$ or~$C^0$ is a homotopy
  equivalence if all these spaces are equipped with their natural
  topologies. The homotopy inverses can be constructed by
  approximations by broken geodesics and rounding off the
  corners. These constructions can be performed in such a way that they do not
  increase lengths; therefore, the definition of~$N$ does not depend
  on the choice of curve regularity.

  Furthermore, the infimum in the definition of~$N$ is attained (see
  Lemma~\ref{lem.loop.attain.inf}). If the infimum is attained in some
  curve~$\gamma$ and if the class~$[\gamma]\in \HZR$ is non-zero, then
  $\gamma$ can be reparametrized (preserving the orientation) to a closed
  geodesic. Such a closed geodesic may be parametrized by arclength
  as~$\gamma:[a,a+N(x)]\to M$ for every~$a\in \R$. In the case
  $[\gamma]=0$, the curve~$\gamma$ is then constant. 
  This implies that $groupnorm$ is positive definite.

  In particular: If $\gamma \colon I \to M$ is a $C^1$-loop
  with~$\maL(\gamma) = N\bigl(h(\gamma)\bigr)$, then every lift~$I \to
  \homRcov M$ of~$\gamma$ to~$\homRcov M$ is minimizing. Hence,
  sequences of such geodesics can be used to construct minimal
  geodesics on~$M$ via Lemma~\ref{lemma:limitinggeodesic}. This will
  be the starting point for the constructions in the proof of Theorem~\ref{thm:main}.
  \end{remark}

\begin{comparisonliterature*}
  The map $N$ above is denoted as~$f$ by Bangert~\cite{bangert:90}.

  Gromov does not treat the norm~$N$, but a variant called ``length''
  defined on~$H_1(M;\Z)$ instead of~$\HZR$, using loops in a fixed
  basepoint~$v_0\in M$~\cite[Chap.~4.C.]{gromov:99}. More precisely,
  for every class in~$\pi_1(M,v_0)$, Gromov considers the infimum
  ($=$ minimum) of lengths of loops based at~$v_0$ representing this
  given class, leading to a function~$\pi_1(M,v_0)\to\R$.  By
  minimizing this function over preimages of the Hurewicz map~$\pi_1(M,v_0)\to H_1(M;\Z)$,
  Gromov obtains a map~$\length\colon
  H_1(M;\Z)\to \R_{\geq 0}$ that is positive definite, symmetric,
  satisfies the triangle inequality (\ie $\Delta_{\min}=0$ in the
  quasi-triangle inequality), and is subhomogeneous. There is a
  constant $C$ with the property: if $\bar\alpha$ denotes the image of
  $\alpha\in H_1(M;\Z)$ in $\HZR$, then 
  \[ N(\bar\alpha)\leq \length(\alpha)\leq N(\bar\alpha)+C
  \,.
  \]
  Then $\stabnorm{\bar\alpha}$ (in our notation) coincides with
  $\|\alpha_\R\|$ in Gromov's notation; however, these two norms are
  slightly differently introduced.
\end{comparisonliterature*}

\subsection{The stable norm on first homology}\label{subsec:stable-norm}

We introduce the stable norm on $H_1(M;\R)$ as the homogenization of
the minimal length map~$N:\HZR\to \R_{\geq 0}$.  Our approach to the
stable norm is close to Bangert's~\cite[Sec. 2]{bangert:90}. It can be
shown that the norm defined this way is equivalent to the classical
stable norm by Federer~\cite[\Paragraph 3]{federer.VP}, also called
the \emph{mass} of a real homology
class~\cite[4.15--4.17]{gromov:99}\ifwithappendixstable{}; see
Definition~\ref{def:federer-mass} for~$K=\R$\fi{} for a definition in
easier notation. The equivalence follows from Gromov's work~\cite[4.18
  and~4.20$\frac12$bis$_+$]{gromov:99} by using results from
  Federer~\cite[\Paragraph 5]{federer.VP}.
\ifwithappendixstable{}For the sake of being self-contained
we present the proof of this equivalence in Appendix~\ref{app:stable-norm}.\fi

\begin{proposition}[the stable norm]\label{prop:stabnormconstruction}
  Let $M$ be a closed connected Riemannian manifold. Then:
  \begin{enumarab}
  \item\label{i:Nisquasinorm}
    The minimal length map~$N \colon \HZR \to \R_{\geq 0}$
    from Definition~\ref{def:N} is a subhomogeneous
    quasi-norm on~$\HZR$ in the sense of
    Definition~\ref{def:quasinorm} (with $\triangle$-discrepancy
    at most~$2 \cdot \diam M$). We refer to~$N$ as the
    \emph{minimal length quasi-norm}.
  \item\label{i:Nhomogenisation}
    The homogenisation (in the sense of Proposition~\ref{prop:homog})
    of~$N$ on~$\HZR$ extends uniquely to a norm~$\stabnorm{\argu}$
    on~$H_1(M;\R)$, the \emph{stable norm}.
  \end{enumarab}
\end{proposition}

The minimal length quasi-norm also satisfies a doubling property, as
we will see in Proposition~\ref{prop:doub.prop}.

\begin{proof}
  \Adref{i:Nisquasinorm}
  This can be checked on representatives, using the flexibility in
  regularity provided by Remark~\ref{rem:defN}; for the quasi-triangle
  inequality, we connect short loops by paths of length at most~$\diam
  M$.

  \Adref{i:Nhomogenisation} We define $\stabnorm{\argu}:H_1(M;\R)\to
  \R_{\geq 0}$ as the homogenisation of $N$, as described in
  Proposition~\ref{prop:homog}.  In particular, $\stabnorm{\argu}$ is
  homogeneous and satisfies the triangle inequality. Symmetry
  follows from the symmetry of $N$. It remains to prove
  $\stabnorm{x}>0$ for all~$x\in H_1(M;\R)\setminus\{0\}$.
  
  We fix an inner product~$\<\arrgu\>_{H_1}$ on~$H_1(M;\R)$, which
  induces a norm~$\|\argu\|_{H_1}$, and thus a length
  functional~$\maL^{H_1}$ for curves.  
  Let $J\colon \ucov M\to H_1(M;\R)$ be the Jacobi map (Definition~\ref{def:jacobi}). 
  This map is differentiable with $\upd J=\sum_{i=1}^b \alpha^i\beta_i\in \Omega^1(M)\otimes H_1(M;\R)$. 
  Let $\bigl\|\upd J\steLLe{p}\bigr\|_{\mathrm{Op}}$ denote the operator norm of
  \[\upd J\steLLe{p}\colon(T_pM,g_p)\to (H_1(M;\R),\<\arrgu\>_{H_1})\,.\] 
  As $\upd J\steLLe{p}$ depends continuously on~$p$ and as it is invariant under the deck
  transformation group, which acts cocompactly, we see that the
  supremum in
  \[ \mathcal{K}\ceq \sup_{p\in \ucov M}\bigl\|\upd J\steLLe{p}\bigr\|_{\mathrm{Op}}
  \] 
  is attained and thus finite.
  For every finite $\Cinftypw$-curve~$\gamma$, we get
  \[ \mathcal{L}^{H_1}(J\circ \gamma)\leq \mathcal{K}\cdot \mathcal{L}^g(\gamma)
  \,.
  \] 
  This immediately implies that for every~$x\in \HZR$ we obtain 
  \[ \|x\|_{H_1}\leq \mathcal{K}\cdot N(x)
  \,.
  \]
  We apply homogenisation, which does not affect~$\|\argu\|_{H_1}$ as
  it is already homogeneous, and get
  \[ \frac1{\mathcal{K}}\cdot\|x\|_{H_1}\leq \stabnorm{x}
  \,.
  \] 
  Positive definiteness on the left-hand side implies positive
  definiteness on the right-hand side.
\end{proof}

\begin{example}\label{example:torus.basics}
  Let $M=T^n$ be an $n$-dimensional torus, which we will equip with two
  kinds of metrics~$g$. We identify $T^n$ with $\mathbb{R}^n/\Gamma$
  where~$\Gamma$ is a lattice, \ie a discrete cocompact subgroup
  of~$\mathbb{R}^n$. In both cases we use the notation of
  Setup~\ref{setup:bases} and the construction of the Jacobi map in
  Definition~\ref{def:jacobi}.  It is convenient to choose the 
  basepoint~$x_0=0$ and the forms~$\alpha^i=dx^i$, where
  $x=(x^1,\ldots,x^n)$ are the standard coordinates of~$\ucov
  M=\mathbb{R}^n$. Then we get $J^i(x)=x^i$, \ie $J$ corresponds
  to the identity map under the canonical isomorphism~$H_1(T^n;\R) \cong\R^n$.
  Moreover, we have a compatible canonical 
  identifications~$H_1(T^n;\Z)\cong\HZR\cong\Gamma$.
  \begin{enumarab}
  \item\label{example:torus.flat} If $g$ is flat, then we can identify
    $(T^n,g)$ with~$\mathbb{R}^n/\Gamma$, where $\R^n$ carries the
    Euclidean metric.  Then $N$ and $\stabnorm{\argu}$ coincide with
    the Euclidean norm. The stable norm unit ball is thus the standard
    ball.
     
    Every geodesic on $(T^n,g)$ can be extended to~$\R$ and this
    extension lifts to a straight affine line in~$\mathbb{R}^n$.
    Parametrized by arclength, these geodesics are minimal.
  \item\label{example:torus.hedlund} Let $n\geq 3$.  For the
    Hedlund metrics~$g$ on~$T^n$, the stable norm is the
    $\ell^1$-norm with respect to the chosen basis of~$H_1(T^n;\R)$~\cite[Prop.~5.8]{bangert:90}
    (the argument given by Bangert for~$n=3$ can easily be extended
    to~$n\geq 3$).  In particular, the stable norm unit ball has
    the combinatorial type of the $n$-dimensional cross-polytope.
\end{enumarab}
\end{example}

\begin{remark}
  If $M$ is a closed connected Riemannian manifold,
  then the stable norm unit ball~$B \subset H_1(M;\R)$
  is convex, compact, and centrally symmetric. 
  Several types of convex bodies are possible, \eg polytopes, smooth domains, combinations of flat and smooth boundaries and much more.
  But not every convex, compact, and centrally symmetric body will arise as the stable norm unit ball for a Riemannian metric, see Remark~\ref{rem:normrealization}.

  However, if $\dim M \geq 3$ and $P$ is a convex compact centrally
  symmetric polyhedron in~$H_1(M;\R)$ whose vertices have a rational direction (with
  respect to~$\HZR$), then there exists a Riemannian
  metric on~$M$ whose stable norm unit ball coincides with~$P$. 
  Suitable metrics are provided by the Hedlund metrics,
  introduced on arbitrary closed connected manifolds of dimension at
  least~$3$ in the first author's diploma thesis~\cite{ammann:diploma}.
  The fact that every stable norm unit ball 
  as above may be achieved by such Hedlund metrics follows
  immediately from the necessary conditions on $\R$\=/ho\-mo\-lo\-gi\-cal\-ly
  minimal geodesics in this diploma thesis~\cite[IV.1.e Homolog
    minimale Geodätische]{ammann:diploma}; in fact, Bangert's minimal
  geodesics are $\R$\=/ho\-mo\-lo\-gi\-cal\-ly minimal geodesics and thus
  these methods determine the norm ball of the Hedlund
  examples~\cite[Korollar~IV.1.16, Bemerkung~IV.1.18]{ammann:diploma}.

  Hedlund metrics on closed connected manifolds were re-introduced
  later~\cite{babenkobalacheff,jotz} and the statements of this remark
  have been worked out explicitly in these articles. 
  We will give more details on Hedlund examples in Subsection~\ref{subsec:Hedlund.examples}.
\end{remark}

\begin{remark}\label{rem:normrealization}
  For a given closed manifold~$M$, one can ask whether there is a Riemannian metric on~$M$ such that the given norm is its stable norm.
  This question is wide open, even for tori. 
  However, some obstruction are known. 
  Points in $H_1(M;\R)\cong\HZR\otimes_\Z\R\cong (\HZR\otimes_\Z)\otimes_\Z\R$ are called irrational if their coefficients are linearly independent over $\Q$.

  On the $2$-torus, Bangert \cite{bangert:94} has proved that the stable norm is differentiable at every irrational point. 
  If the stable norm is differentiabile in a rational point, then the torus is foliated by minimal geodesics in this direction.

  In dimension larger than~$2$, there is more flexibility to obtain a given norm, but some obstructions remain. 
  On $n$-tori, $n\geq 3$, Burago, Ivanov, and Kleiner \cite{burago.ivanov.kleiner:97} have proved that in irrational points, the stable norm has a directional derivative in at least one non-radial direction. This implies that vertices of the stable norm ball have rational directions. In particular if $B$ is a polytope, then all of its vertices lie in rational directions.
\end{remark}

\subsection{Some definitions for convex bodies}\label{subsec:def.convex.bodies}

Let us recall some standard terminology for convex bodies.

Let us assume that $B$ is a closed convex subset of a finite-di\-men\-sion\-al $\R$-vector space~$V$ with~$0\in\inneres{B}$. 
An \emph{exposed face} of~$B$ is a non-empty subset~$F\subset B$ 
for which there exists an affine hyperplane~$H \subset V$ with
\[ H \cap \partial B = F
\qand H \cap \inneres B = \emptyset
\,;
\]
in this case, $H$ is called a \emph{supporting hyperplane} for the exposed face~$F$ of~$B$.  
An \emph{exposed vertex} of~$B$ is a point~$x \in \partial B$ such that
$\{x\}$ is an exposed face of~$B$. 
An \emph{exposed edge} of~$B$ is an exposed face of dimension~$1$.

If $H$ is a supporting hyperplane for~$F$, then $0\notin H$, and thus there is a unique $\omega_H$ in the dual space $V^*$ 
with $H= \{x\in V\mid \omega_H(x)=1\}$. Then $F= \{x\in B\mid \omega_H(x)=1\}$.

It is easy to show that every closed convex subset of~$V$ is the intersection of the closed halfspaces containing the given convex subset.
\ifwithdetails We refer to \url{https://math.stackexchange.com/questions/2130746/convex-sets-as-intersection-of-half-spaces} for a proof.\fi 
As a consequence $\partial B$ is the union of all exposed faces of $B$.

The definition of an exposed point is tightly related to the
definition of an \emph{extreme point}. 
A point $v$ in a convex body $B$
is called \emph{extreme} if it does not lie in any open line segment
between two different points of $B$. All exposed points are extreme,
but not vice versa. 

The usage of the words ``extreme'' and ``exposed'' may be considered as confusing:
one easily sees that a point $x\in \partial B$ is exposed with supporting hyperplane $H$ if,
and only if, it is the unique maximum of the linear function $\omega_H\colon V\to \R$ restricted to~$B$.
Not all extreme points satisfy this property.
In this article, exposed points will play an important role, but the
definition of extreme points will only be of marginal importance.

Often we will also abbreviate the term ``an exposed edge of $B$'' by simply
saying ``an edge of $B$'', and similarly we will abbreviate ``an exposed face'' by ``a face''.
However, one should keep in mind that usually in the literature  
more general notions of faces and edges are used.

For compact convex polytopes in
finite-dimensional $\R$-vector spaces, all faces (and vertices) are
exposed.

The main focus of the article lies on closed Riemannian manifolds with a finite number of geometrically distinct minimal geodesics. Using a result by Bangert \cite[Theorem~4.4]{bangert:90}, see also Subsection~\ref{subsec.rot.vector}, this assumption  implies that the stable unit ball in $\HdR$ is a polytope. Thus all faces are exposed, and extremality coincides with exposedness in this case.

\subsection{The dual norm}\label{subsec:dual.norm}

Finally, we collect elementary properties of the dual norm of the
stable norm. In fact, the following considerations apply to all
norms on finite-dimensional $\R$-vector spaces. For convenience,
we formulate most of the statements directly for the stable norm.

\begin{definition}[dual norm]
  Let $M$ be a closed connected Riemannian manifold.
  The dual norm of~$\stabnorm{\argu}$ on~$\HdR$
  is denoted by~$\stabnormdual{\argu}$. More
  explicitly: For all~$\omega \in \HdR$, we have
  \begin{equation}\label{def.stabnormdual} \stabnormdual{\omega}
  = \sup_{x \in H_1(M;\R)\setminus\{0\}} \frac{\bigl|\<\omega,x\>\bigr|}{\stabnorm x}\,.
  \end{equation}
\end{definition}
In view of compactness of~$\{x\in H_1(M;\R)\mid \stabnorm{x}=1\}$, the
supremum is attained (if $H_1(M;\R) \not\cong 0$).

For $\omega \in \HdR\setminus\{0\}$ we define the hyperplane 
\begin{equation}\label{eq:H.omega.def}
    H_\omega\ceq \bigl\{ x\in H_1(M;\R) \bigm|  \< \omega,x\> = 1\bigr\}\,.
\end{equation}  
Every hyperplane~$H$ with $0\not\in H$ is obtained this way.

Let us assume in this entire subsection that $M$ is a closed connected Riemannian manifold,
and that $B \subset H_1(M;\R)$ is the stable norm unit ball.

\begin{lemma}
With the definitions above $H_\omega$ is a supporting hyperplane for a face of $B$ if, and only if, $\stabnormdual \omega =1$.
\end{lemma}
\begin{proof}
\ifwithdetails
We have
\begin{align*}
  H\cap \inneres{B}\neq \emptyset 
  &\gdw \exi{x\in\inneres{B}}\<\omega,x\>=1\\
  &\gdw \exi{x\in H_1(M;\R)\setminus\{0\}}\frac{\<\omega,x\>}{\stabnorm{x}}>1\\
  &\gdw \stabnormdual \omega >1\,.
\end{align*} 
and similarly  
\begin{align*}
  H\cap B\neq \emptyset 
  &\gdw \exi{x\in B}\<\omega,x\>=1\\
  &\gdw \exi{x\in H_1(M;\R)\setminus\{0\}}\frac{\<\omega,x\>}{\stabnorm{x}}\geq 1\\
  &\stackrel{(*)}{\gdw} \stabnormdual \omega \geq 1\,,
\end{align*} 
where at $(*)$ the direction $\Rightarrow$ follows directly from~\eqref{def.stabnormdual} and the direction $\Leftarrow$ follows from the fact that the supremum in there is attained.
The condition $H \cap \partial B = F$ and $H \cap \inneres B = \emptyset$ is thus equivalent to $\stabnormdual \omega = 1$.
\else
The proof is straigthforward, using that the supremum in~\eqref{def.stabnormdual} is attained.
\fi
\end{proof}  

\begin{definition}\label{def:face}
  Let $M$ be a closed connected Riemannian manifold,
  let $B \subset H_1(M;\R)$ be the stable norm
  unit ball, and let $\omega \in \HdR$ with~$\stabnormdual \omega =1$.
  We define the hyperplane $H_\omega$ as in \eqref{eq:H.omega.def}.
  Equation~\ref{def.stabnormdual} and~$\stabnormdual \omega =1$ imply $H_\omega \cap \inneres B=\emptyset$. 
  The \emph{exposed face of~$B$ defined by~$\omega$} is
  \[ F_\omega \ceq \bigl\{ x\in H_1(M;\R) \bigm| \stabnorm x = 1
  \text{ and }\< \omega,x\> = 1 \bigr\} = H_\omega \cap \partial B \,, \]
and as in the preceeding subsection, we define \emph{exposed points} and \emph{exposed edges}.
\end{definition}
As $F_\omega$ is the intersection of the closed and convex sets $B$ and~$H_\omega$, the exposed face~$F_\omega$ is closed and convex as well.

\begin{lemma}
  Let again $B \subset H_1(M;\R)$ be the stable norm
  unit ball of a closed connected Riemannian manifold. 
  If $\omega_0,\omega_1 \in \HdR$
      with~$\stabnormdual {\omega_0} = 1 = \stabnormdual {\omega_1}$ and
      $F_{\omega_0} \cap F_{\omega_1} \neq \emptyset$, then for every
      $t\in (0,1)$ and for $\omega_t\ceq (1-t)\omega_0 + t\omega_1\in
      \HdR$ we have
      \[\stabnormdual{\omega_t} = 1 \qand
        F_{\omega_t} = F_{\omega_0} \cap F_{\omega_1}\,.
      \]
\end{lemma}

\begin{proof}
  We first show~$\stabnormdual {\omega_t} =1$: From the triangle
  inequality for $\stabnormdual\argu$ we get $\stabnormdual{\omega_t}
  \leq 1$.  Conversely, let $x \in F_{\omega_0} \cap
  F_{\omega_1}$. Then
  \[ \stabnormdual {\omega_t} \geq \frac{|\<\omega_t,x\>|}{\stabnorm x}
  = 
  \frac{\bigl| (1-t)\<\omega_0, x\> + t\<\omega_1,x\>\bigr|}1
  = 1
  \,.
  \]

  We clearly have~$F_{\omega_0} \cap F_{\omega_1} \subset
  F_{\omega_t}$.  Conversely, let $x \in F_{\omega_t}$. Then for
  $i\in\{0,1\}$ we have $|\<\omega_i,x\>| \leq
  \stabnormdual{\omega_i}\cdot \stabnorm{x} = 1$ by 
  the defining equation~\eqref{def.stabnormdual}. This implies
  \begin{align*}
    1
  & = \<\omega,x\>
    = (1-t)\<\omega_0,x\> + t\<\omega_1, x\> 
  \leq \bigl( (1-t)\stabnormdual{\omega_0} + t\stabnormdual{\omega_1}\bigr)
  \cdot \stabnorm x
  = 1\,.
  \end{align*}
  Therefore, $\<\omega_0,x\> = 1 = \<\omega_1,x\>$, which shows
  that~$x \in F_{\omega_0} \cap F_{\omega_1}$.
\end{proof}

\section{Bounded difference of the minimal length and the stable norm}
\label{sec:findist}

We show for general closed connected Riemannian manifolds that,
on the integral part of homology, the difference between
the minimal length and the stable norm is uniformly bounded.
This generalizes a result of D.~Burago, who
considered the special case of tori~\cite{burago:soviet:92}. 

\begin{theorem}\label{thm:stabnormN}
  Let $M$ be a closed connected Riemannian manifold.
  Then, there is a constant~$D\in \R$ with 
  \[ \fa{x \in \HZR}
  \stabnorm x \leq N(x) \leq \stabnorm x + D
  \,.
  \]
  A concrete bound for~$D$ is provided in Equation~\eqref{eq.D.def}.
\end{theorem}

The theorem can be considered as a continuous analogue of
Remark~\ref{rem:wordlength}. 
A modified version of this theorem is mentioned by Gromov~\cite[Sec.~4.21, Remarks$_+$(b), equation~$(\triangle)$]{gromov:99}.
Gromov knew that it follows from Burago's
techniques~\cite{burago:soviet:92}, but did not provide
details.

Burago's proof in the torus case uses the following splitting lemma
going back to Nazarov, which is also the starting point for us.

\begin{lemma}[splitting lemma~\protect{\cite[Lemma~2]{burago:soviet:92}}]\label{lem:split}
  Let $V$ be an $\R$-vector space of dimension~$b \in \N$ and
  let $\rho \colon [0,\ell] \to V$ be a continuous path. Then there
  exists a number~$d \in \{0,\dots,\lceil b/2 \rceil\}$ and 
  pairwise disjoint closed subintervals~$[\sigma_i,\tau_i]$
  of~$[0,\ell]$ for~$i \in \{1,\dots,d\}$ with
  \begin{equation}
    \sum_{i=1}^d \bigl(\rho(\tau_i) - \rho(\sigma_i)\bigr)
     = \frac12 \cdot \bigl(\rho(\ell) - \rho(0)\bigr)\,.
     \label{eq.half.of.it}
  \end{equation}
  We order these intervals in such a way that~$0\leq \sigma_1 < \tau_1
  < \sigma_2 < \dots < \tau_d\leq \ell$ and call this sequence a
  \emph{splitting partition} for~$\rho$.
\end{lemma}

The original proof by Nazarov was simplified by
Burago~\cite{burago:soviet:92}, using ideas of Burago and Perelman. As
our notation and statement slightly differs, we repeat Burago's
argument for the sake of self-containedness.

\begin{proof}
  Without loss of generality we assume $V=\R^b$. For a point
  $x=(x_1,\ldots,x_{b+1})\in S^b$ and $k\in\{0,1,\ldots,b+1\}$ we
  define
  \[ t_k(x) := \ell\cdot\sum_{i=1}^k x_i{}^2,
  \quad v(x) := \sum_{i=1}^{b+1} \sign x_i \cdot \bigl(\rho\bigl(t_i(x)\bigr)- \rho\bigl(t_{i-1}(x)\bigr)\bigr)\,.
  \]
  Then $v\colon S^b\to \R^b$ is a continuous map with  
  $v(-x)=-v(x)$ for all~$x\in S^b$. By the Borsuk--Ulam theorem, 
  there exists a~$z=(z_1,\ldots,z_{b+1})\in S^b$ with~$v(z)=0$.

  For~$k\in\{1,\ldots,b+1\}$, we consider the interval~$I_k\ceq
  [t_{k-1}(z),t_k(z)]$. Then, $I_k$ has positive length if and only if 
  $z_k\neq 0$. Thus the sets
  \[ A_+\ceq \bigcup_{k \text{ with } z_k>0}I_k
  \qand A_-\ceq \bigcup_{k \text{ with } z_k<0}I_k
  \]
  satisfy
  \[ [0,\ell]=A_+\cup A_-,\quad
  \inneres{A}_+\cap \inneres A_-=\emptyset, \quad
  \overline{\inneres{A}}_\pm=A_\pm
  \,,
  \]
  and both $A_+$ and $A_-$ are
  finite unions of disjoint closed intervals of positive length. Then
  $A_+$ or $A_-$ is the union of $d\leq \lfloor (b+1)/2 \rfloor=
  \lceil b/2 \rceil$ such intervals, denoted as~$[\sigma_i,\tau_i]$
  with~$0\leq \sigma_1 < \tau_1 < \sigma_2 < \dots < \tau_d\leq
  \ell$. By construction, $v(z)=0$ is equivalent to
  Equation~\eqref{eq.half.of.it}.
\end{proof}

\begin{corollary}\label{cor:split}
  Let $M$ be a closed connected Riemannian manifold and $b\ceq \dim_\R  H_1(M;\R)$.  
  Let~$\gamma$ be a continuous loop in~$M$, viewed as a
  periodic continuous curve $\R\to M$ with period~$\ell>0$. Let $\ucov \gamma$
  be a lift of~$\gamma$ to~$\ucov M$. Then, there is a real
  number~$a\in [0,\ell)$, an integer~$d\in \{0,1,\ldots,\lceil b/2 \rceil\}$, 
  and real numbers~$\sigma_i,\tau_i$ with $a= \sigma_1< \tau_1 <
  \sigma_2 < \dots < \tau_d\leq a+\ell$ and 
 \begin{equation*}
    \sum_{i=1}^d \bigl(J(\ucov\gamma(\tau_i)) - J(\ucov\gamma(\sigma_i))\bigr)
    = \frac12 \cdot \bigl(J(\ucov\gamma(\ell)) - J(\ucov\gamma(0))\bigr)
    =  \frac12 \cdot h\bigl(\gamma\stelle{[0,\ell]}\bigr)\,.
  \end{equation*}   
\end{corollary}
\begin{proof}
  We apply Lemma~\ref{lem:split} to~$J\circ \gamma$ and set $a\ceq
  \sigma_1$.  For the last equality, we use the relation between the
  geometric Hurewicz map and the Jacobi map
  (Remark~\ref{rem:jacobiproperties}).
\end{proof}

To prove Theorem~\ref{thm:stabnormN}, we introduce some geometric
quantitites:  Let $M$ be a closed connected Riemannian manifold
and let $b := \dim_\R H_1(M;\R)$. The operator
norm~$\norm{d_xJ}_{g,\stable}$ of~$d_xJ\colon (T_x\ucov M,g)\to
(H_1(M;\R),\stabnorm{\argu})$ is continuous in~$x \in \ucov M$ (see
Remark~\ref{rem:jacobiproperties}) and invariant under deck
transformations of the universal covering~$\ucov M\to M$. Hence,
there is a uniform upper bound and we may define
\begin{equation}\label{eq.def.maJ}
  \maJ\ceq \sup_{x\in \ucov M} \norm{d_xJ}_{g,\stable}<\infty\,.
\end{equation}
\begin{lemma}\label{lemma.h.rho}
  In this situation,
  for every $\Cinftypw$-curve~$\rho\colon[a,b]\to M$, we have
  \[\stabnorm{h(\rho)}\leq \maJ\cdot \maL(\rho)
  \,.
  \]
\end{lemma}

\begin{proof}
  Using Remark~\ref{rem:jacobiproperties}, we obtain for a lift $\tilde \rho$ of $\rho$
  \begin{align*}
    \stabnorm{h(\rho)}
    & =  \stabnorm[big]{J\bigl(\tilde\rho(b)\bigr)-J\bigl(\tilde\rho(a)\bigr)}\\
    &\leq  \int_a^b \stabnorm[Big]{\frac{d}{dt} J\bigl(\tilde\rho(t)\bigr)}\,dt\\
    &\leq  \int_a^b  \norm{\upd_{\tilde\rho(t)}J}_{g,\stable} \cdot \norm{\dot\rho(t)}_g\,dt\\
    &\leq  \maJ\cdot \maL(\rho)\,.
    \qedhere
  \end{align*}
\end{proof}

We now set
\begin{align}
  K\ceq
\sup
\Bigl\{N(y)
\Bigm| y\in \HZR,\, \stabnorm{y}\leq  \maJ\cdot \frac{b+1}2\cdot \diam M
\Bigr\}\,.
\label{eq:Kdef}
\end{align}

The following doubling property will provide the essential step
for the proof of Theorem~\ref{thm:stabnormN}.

\begin{proposition}[Doubling property for the minimal length quasi-norm]\label{prop:doub.prop}
  The minimal length quasi-norm on a closed connected Riemannian manifold~$M$
  with~$b = \dim_\R H_1(M;\R)$ satisfies the doubling property from
  Definition~\ref{def:quasinorm} for the constant
\begin{equation}\label{eq.D.def}
  D \ceq  (b+5)\cdot\diam M + 2\cdot K\,.
\end{equation}
  Here, $K$ is the constant from Equation~\eqref{eq:Kdef}.
\end{proposition}

\begin{proof}
  For all~$x \in \HZR$ and the constant~$D$ from the claim, we have to show
  \[ 2 \cdot N(x) \leq N(2 \cdot x) + D
  \,.
  \]

  For $x \in \HZR$ we abbreviate~$\ell\ceq N(2\cdot x)$. We choose~$\gamma \colon [0,\ell] \to M$ as a closed
  geodesic representing~$2 \cdot x$ with~$\maL(\gamma) = \ell= N(2 \cdot x)$.
  Let $\ucov \gamma$ be a lift of~$\gamma$
  to~$\ucov M$.  To find ``sufficiently short'' representatives for~$x$, we
  split~$\ucov\gamma$ via Corollary~\ref{cor:split}; by possibly
  reparametrizing~$\gamma$ we may assume~$a=0$ in this corollary.
  For~$i \in \{0,\dots,d-1\}$, we define $r_{2i}\ceq \sigma_{i+1}$,
  $r_{2i+1}\ceq \tau_{i+1}$, and $r_{2d}\ceq \ell$. We thus have
  obtained a splitting partition~$0=r_0< r_1<\dots<r_k = \ell$
  of~$[0,\ell]$ with an even number~$k\ceq 2d \leq b+1$.

  Let $\theta_j\ceq \gamma\stelle{[r_j,r_{j+1}]}$. The conclusion from 
  Corollary~\ref{cor:split} and Remark~\ref{rem:jacobiproperties} give
  us that
  \begin{equation*}
    \sum_{i=0}^{d-1}h(\theta_{2i})
    = \sum_{i=0}^{d-1}h(\theta_{2i+1})
    = \frac12 \cdot h(\gamma)
    = x\,.
  \end{equation*}

  We rearrange $\theta_0, \theta_2, \dots$ and $\theta_1, \theta_3,
  \dots$ into two loops: For~$j \in \{0,\dots,k-1\}$, let
  $\varrho_j\colon [0,1]\to M$ be a smooth path from~$\gamma(r_{j})$
  to~$\gamma(r_{j+1})$ with~$\maL(\varrho_j) \leq \diam M$.  We
  consider the rearranged $\Cinftypw$-loops
  \begin{align*}
    \gamma_+
    & \ceq \theta_0 * \rho_1 * \theta_2 * \cdots * \theta_{k-2} * \rho_{k-1}\,,
    \\
    \gamma_-
    & \ceq \rho_0 * \theta_1 * \rho_2 * \theta_3 * \cdots * \theta_{k-1}\,.  
  \end{align*}
  By construction, we have
  \begin{eqnarray*}
    y_+\ceq h(\gamma_+)-\frac12 \cdot h(\gamma)
    &=& \underbrace{\sum_{i=0}^{d-1}h(\theta_{2i})}_{=x} + \sum_{i=0}^{d-1}h(\rho_{2i+1}) -\frac12\cdot2\cdot x\\
    &=&  \sum_{i=0}^{d-1}h(\rho_{2i+1})\,.
  \end{eqnarray*}
  As $\gamma_+$ is closed, we get $y_+=h(\gamma_+)-x\in\HZR$, see
  Remark~\ref{rem:hurewiczint}.  
  
  Our next goal will be to adapt $\gamma_+$ and~$\gamma_-$ in such a way that the resulting loops are ``sufficiently short'' and represent~$x$.
  We use Lemma~\ref{lemma.h.rho} which tells us $\stabnorm{h(\rho_{j})} \leq
  \maJ\cdot \maL(\rho_j) \leq \maJ\cdot \diam M,$ and so 
  \begin{equation*}
    \stabnorm{y_+} \,\leq\, \sum_{i=0}^{d-1}\stabnorm{h(\rho_{2i+1})}
    \,\leq\,   \maJ\cdot d\cdot \diam M\,.
  \end{equation*}
  In particular, $N(y_+) \leq K$, by definition of~$K$. Hence, we may
  find a loop~$\mu_+$ with basepoint~$\gamma_+(0)$ of length at
  most~$K+2\cdot \diam M$ satisfying $h(\mu_+)=-y_+$.  In particular,
  \[ h(\gamma_+*\mu_+)
  = y_+ + \frac12 \cdot h(\gamma) + h(\mu_+)
  = \frac12\cdot h(\gamma)=x
  \,.
  \] 
  Similarly, we construct a loop~$\mu_-$ from $\gamma_-$ 
  with $\maL(\mu_-)\leq K+2\cdot \diam M$ and $h(\gamma_- * \mu_-) =x$. 
  For the corresponding lengths we calculate
  \begin{eqnarray*}
    \maL(\gamma_+*\mu_+)+ \maL(\gamma_-*\mu_-)
    & \leq&  \sum_{i=0}^{k-1}\bigl(\maL(\theta_i) +\maL(\rho_i)\bigr) +\maL(\mu_+)+\maL(\mu_-)\\
    &\leq & \ell+ k\cdot\diam M + 2\cdot K +4\cdot\diam M\,.
  \end{eqnarray*}
  Therefore, 
  \[ 2\cdot\groupnorm{x}
  \leq  \ell+ (k+4)\cdot \diam M + 2 \cdot K
  = \groupnorm{2\cdot x}+ (b+5)\cdot\diam M + 2\cdot K
  \,.
  \]
  This is the desired doubling property and the proposition follows.
\end{proof}

\begin{proof}[Proof of Theorem~\ref{thm:stabnormN}]
  In view of the Propositions~\ref{prop:qndoubling}
  and~\ref{prop:stabnormconstruction} on homogenisations and the
  construction of the stable norm, it follows from the doubling
  property in Proposition~\ref{prop:doub.prop} that
  Theorem~\ref{thm:stabnormN} holds for the same constant~$D$.
\end{proof}

Furthermore, we will make use of the following related
estimate, complementing Lemma~\ref{lemma.h.rho}:

\begin{proposition}\label{prop:stabnormlength}
  Let $M$ be a closed connected Riemannian manifold and let $h$ be a
  geometric Hurewicz map as constructed in
  Section~\ref{subsec:curveshom}.  Then for all compact intervals~$I
  \subset \R$ and all $\Cinftypw$-curves~$\gamma \colon I \to M$, we
  have
  \[ \maL (\gamma)
  \geq \stabnorm{h(\gamma)} - (\maJ +1) \cdot \diam M\,,
  \]
  where~$\maJ$ is defined in Equation~\eqref{eq.def.maJ}.
\end{proposition}
\begin{proof}
  We extend~$\gamma$ to a $\Cinftypw$-loop~$\gamma_0\ceq \gamma *
  \varrho$, where~$\varrho$ is a smooth path from the endpoint
  of~$\gamma$ to the start point of~$\gamma$ with~$\maL(\varrho) \leq
  \diam M$.  Then $\maL(\gamma) = \maL(\gamma_0) - \maL(\varrho)$ and
  $h(\gamma_0) = h(\gamma) + h(\varrho)$.  Moreover, $h(\gamma_0)$ is
  an integral class (Remark~\ref{rem:hurewiczint}).  We compute
  \begin{align*}
    \maL(\gamma)
    & = \maL(\gamma_0) - \maL(\varrho)
    \geq N\bigl(h(\gamma_0)\bigr)- \maL(\varrho)
    \\
    & \geq \stabnorm{h(\gamma_0)} - \maL(\varrho)
    \geq \stabnorm{h(\gamma)} - \stabnorm{h(\varrho)} - \maL(\varrho)\\
    &\geq  \stabnorm{h(\gamma)}  - (\maJ +1) \cdot  \maL(\varrho)
    & \text{(by Lemma~\ref{lemma.h.rho})}\\
    &\geq  \stabnorm{h(\gamma)}  - (\maJ +1) \cdot \diam M\,,
  \end{align*}
  as claimed.
\end{proof}

We end the section with some observations that will help, in Subsection~\ref{subsec:asympt.hom.min}, 
to give a short proof of Proposition~\ref{prop.hom.min.sphere}, but which are also of independent interest.
The following propositions essentially generalize a corresponding result of Burago
for tori~\cite[Theorem~1]{burago:soviet:92}. (We use the word
``essentially'' as Burago's constant is controlled in terms of
different geometric data.)

\begin{proposition}[Finite distance of Riemannian distance and stable norm]\label{prop.finite.dist}
  \hfil
  Let $(M,g)$ be a closed Riemannian manifold with universal covering
  $(\ucov M,\ucov g)$, and let $\ucov d$ be the induced distance
  function on $\ucov M$.  Let $J \colon \ucov M \to H_1(M;\R)$
  be the Jacobi map (Definition~\ref{def:jacobi}).
  \begin{enumarab}
  \item\label{prop.finite.dist.1}  Then for all $p,q\in \ucov M$ we have
    \[ \ucov d(p,q) 
    \geq \stabnorm{J(q)-J(p)} - (\maJ +1) \cdot \diam M
    \,,
    \]
    where~$\maJ$ is defined in Equation~\eqref{eq.def.maJ}.
  \item\label{prop.finite.dist.2} If
    $[\pi_1(M),\pi_1(M)]$ is finite, then there is a constant~$C>0$ auch that for all~$p,q\in \ucov M$ we
    have
    \[ \ucov d(p,q) 
    \leq \stabnorm{J(q)-J(p)} +C
    \,.
    \]
  \item\label{prop.finite.dist.3} Let us now consider the covering
    $\wihat M\to M$ instead of the universal covering, where~$\wihat M$ is either the covering $\homZcov M$ or $\homRcov M$ (or any
    covering in between), defined in
    Section~\ref{subsec.homologically-minimal}.
    The Jacobi map factors to $\wihat J\colon\wihat M\to H_1(M;\R)$.
    Let~$\hat d$ be
    the associated distance function on~$\wihat M$.
    Then there is a
    constant $C>0$ such that for all $\hat p,\hat q\in \wihat M$ we
    have
    \[ \bigl|\hat d(\hat p,\hat q) - \stabnorm[big]{\wihat J(\hat q)-\wihat J(\hat p)}\bigr|
    \leq  C\,.
    \]
\end{enumarab}
\end{proposition}

Upper estimates for the constant $C$ in \ititemref{prop.finite.dist.2}
and~\ititemref{prop.finite.dist.3} can be read off in most cases 
from the proof below.

\begin{proof}
  \Adref{prop.finite.dist.3} As $\wihat M$ is complete, the 
  Hopf--Rinow theorem tells us that for all~$\hat p,\hat q\in \wihat M$ there
  is a curve $\hat\gamma$ of length $\hat d(\hat p,\hat q)$ from $\hat
  p$ to $\hat q$ in $\wihat M$. Proposition~\ref{prop:stabnormlength} then
  yields
  \begin{equation}\label{ineq:hat-d.lower}
    \hat d(\hat p,\hat q) 
  \geq \stabnorm{h(\gamma)} - (\maJ +1) \cdot \diam M\,.
  \end{equation}
  The converse inequality will first be shown for $\wihat M=\homRcov
  M$. Let $\pi\colon\homRcov M\to M$ the corresponding covering, whose
  deck transformation group will be identified canonically with
  $\HZR$.  We may choose a path $\tau$ from $p\ceq \pi(\hat q)$ to
  $q\ceq\pi(\hat p)$, with $\maL(\tau)\leq\diam M$. We choose
  $\hat\gamma$ as above and set $\gamma\ceq \pi\circ\hat\gamma$. The closed
  loop~$\gamma * \tau$ represents a class in $\HZR$ and we choose a
  closed geodesic~$\sigma$ in $M$ that minimizes length in this
  class. Thus $\maL(\sigma)=\groupnorm{[\sigma]}=\groupnorm{[\gamma *
      \tau]}$. We choose a path~$\al$ from $\sigma(0)=\sigma(\ell)$
  to~$p$ in~$M$ of length at most~$\diam M$. Using Theorem~\ref{thm:stabnormN},
  we obtain
  \[ \gamma'
  \ceq \overline\alpha * \sigma * \alpha * \overline\tau
  \] 
  is a path from $p$ to~$q$ such that $\gamma* \overline{\gamma'}$
  represents~$0$ in~$\HZR$. In other words, $\gamma'$ lifts to a path
  from $\hat p$ to $\hat q$.  We obtain
  \begin{eqnarray*}
    \hat d(\hat p,\hat q)&\leq& \maL(\gamma')\,=\, 2\cdot \maL(\alpha)+ \maL(\sigma) +\maL(\tau)\\
    &\leq& \groupnorm{[\gamma * \tau]}+ 3\cdot \diam M\\
    &\leq& \stabnorm[big]{[\gamma * \tau]}+ 3\cdot \diam M + D\,,
 \end{eqnarray*}
 where we used Theorem~\ref{thm:stabnormN} in the last inequality
 for~$D$ given by Equation~\eqref{eq.D.def}. Let $\hat \tau$ be a lift
 of~$\tau$ starting in~$\hat q$; we denote its endpoint by~$\hat
 p_1$. In particular, $\hat \gamma*\hat \tau$ is a path from~$\hat p$
 to~$\hat p_1$, lifting~$\gamma*\tau$.  By construction 
 and Remark~\ref{rem:hurewiczproperties} we have
 \[ [\gamma * \tau]
 = h(\gamma * \tau)
 = h(\gamma)+h(\tau)
 = \wihat J(\hat q)-  \wihat J(\hat p) + h(\tau)\,.
 \] 
 From Lemma~\ref{lemma.h.rho} we see that $\stabnorm{h(\tau)}\leq \maJ
 \cdot \maL(\tau)\leq \maJ \cdot \diam M$. Using the triangle
 inequality for $\stabnorm\argu$, we finally get
 \begin{eqnarray*}
     \hat d(\hat p,\hat q)&\leq&   \stabnorm[big]{ \wihat J(\hat q)-  \wihat J(\hat p)} + \stabnorm{  h(\tau)}+ 3\cdot \diam M + D\\
                          &\leq&   \stabnorm[big]{ \wihat J(\hat q)-  \wihat J(\hat p)} + (3+\maJ) \cdot \diam M + D\,.
 \end{eqnarray*}
 This proves the converse inequality for $\wihat M=\homRcov M$.

 As the torsion group of $H_1(M;\Z)$ is finite, the covering $\homZcov
 M\to\homRcov M$ is finite. We thus can apply
 Lemma~\ref{lem:finite-covering} for $Q_1\ceq \homZcov M$ and $Q_2\ceq
 \homRcov M$, and this provides a constant $C>0$ with
 \begin{eqnarray*}
   \fa{\hat p, \hat q \in \homZcov M} 
   \homZcov{d}(\hat p,\hat q) &\leq & \stabnorm[big]{ \wihat J(\hat q)-  \wihat J(\hat p)} + \stabnorm{  h(\tau)}+ (3+\maJ) \cdot \diam M + D+C\,.
 \end{eqnarray*}
    
 \Adref{prop.finite.dist.1} Let $p,q\in\ucov M$ with images $\hat
 p,\hat q\in \homRcov M$. Obviously, we have $\ucov d(p,q)\geq \hat
 d(\hat p,\hat q)$. Thus, the claimed inequality directly follows from
 Equation~\eqref{ineq:hat-d.lower}.
  
  \Adref{prop.finite.dist.2} The covering $\ucov M\to \homZcov M$ has
  deck transformation goup $[\pi_1(M),\pi_1(M)]$, and thus under the
  assumption of this item, this is a finite covering of bounded
  geometry. Lemma~\ref{lem:finite-covering} implies that there is a 
  constant $C>0$ such that for all~$\hat p, \hat q\in \homZcov M$ with
  lifts $\ucov p,\ucov q\in \ucov M$ we have
  \[ \hat{d}(\hat p, \hat q) \leq \ucov d(\ucov p,\ucov q) \leq \hat{d}(\hat p, \hat q) +C
  \,.
  \]
  The claim thus follows from Item~\ititemref{prop.finite.dist.3}.
\end{proof}

\section{Homological asymptotes}\label{sec:asymptotes}

We introduce homological asymptotes for curves and distinguish
different types of minimal geodesics according to their homological
asymptotes. Moreover, we collect various examples.

\subsection{Definition of homological asymptotes}\label{subsec:def.hom.asymp}

In most of the literature, the asymptotic direction of a minimal geodesic in homology was
termed a ``rotation vector''. This notion is very intuitive in the
case~$b\ceq \dim H_1(M;\R)=2$ and, in particular, in the case of a
$2$-dimensional torus (Example~\ref{exa.torus.asymptotes}).
However, for~$b>2$ it seems more adequate to introduce the
following asymptotes in~$H_1(M;\R)$. These are close to the definition of
rotation vectors by Bangert~\cite{bangert:90} and the first author
(Section~\ref{subsec.rot.vector}).

\begin{definition}[terminal/initial/mixed asymptotes]\label{def.asymptotes}
  Let $M$ be a closed connected Riemannian manifold,
  let $h$ be a geometric Hurewicz map for~$M$, 
  and let $\gamma \colon \R \to M$ be a $\Cinftypw$-curve
  parametrized by arclength.
  \begin{itemize}
  \item We say that $v\in H_1(M;\R)$ is \emph{a terminal
    asymptote of~$\gamma$} if there exists a sequence~$(t_i)_{i\in \N_0}$
    in~$\R$ with~$\lim_{i \to \infty} t_i = \infty$ and
    \begin{equation}\label{eq.v.def}
      v = \lim_{i \to \infty}
           \frac{h\bigl(\gamma\stellle{[t_0,t_i]}\bigr)}{t_i - t_0}\,.
     \end{equation}
      We write $A^+(\gamma)$ for the set of all terminal
    asymptotes of~$\gamma$.
  \item We say that $v\in H_1(M;\R)$ is \emph{an initial
    asymptote of~$\gamma$} if there exists a sequence~$(t_i)_{i\in \N_0}$
    in~$\R$ with~$\lim_{i \to \infty} t_i = -\infty$ and
    \[ v = \lim_{i \to \infty}
           \frac{h\bigl(\gamma\stellle{[t_i,t_0]}\bigr)}{t_0 - t_i}\,.
    \]
    We write $A^-(\gamma)$ for the set of all initial
    asymptotes of~$\gamma$.
 \item We say that $v\in H_1(M;\R)$ is \emph{a mixed
   asymptote of~$\gamma$} if there are sequences~$(s_i)_{i\in \N}$,
   $(t_i)_{i\in \N}$ in~$\R$
   with~$\lim_{i \to \infty} s_i = -\infty$ and~$\lim_{i \to \infty} t_i = \infty$ and
    \[ v = \lim_{i \to \infty}
           \frac{h\bigl(\gamma\stellle{[s_i,t_i]}\bigr)}{t_i - s_i}\,.
    \]
    We write $\Atot(\gamma)$ for the set of all mixed
    asymptotes of~$\gamma$.
  \end{itemize}
  Moreover, we set~$\Apm(\gamma) := A^+(\gamma) \cup A^-(\gamma)$.
\end{definition}

\begin{remark}\label{rem:hastart}
  A straightforward triangle inequality estimate shows
  that terminal/initial asymptotes are independent of
  the start point~$t_0$. 
\end{remark}

\begin{remark}\label{rem:distinctasymptotes}
  Let $M$ be a closed connected Riemannian manifold and let
  $\gamma, \sigma \colon \R \to M$
  be minimal geodesics of~$M$.
  If $\gamma$ and $\sigma$ are geometrically equivalent, say $\sigma=\gamma\circ \phi$, then
  \begin{align*}
    A^+(\gamma) = A^+(\sigma) & \qand A^-(\gamma) = A^-(\sigma) &&\text{if }\phi'>0\,,\\
    \text{or} \quad 
    A^+(\gamma) = -A^-(\sigma) & \qand A^-(\gamma) = -A^+(\sigma)&&\text{if }\phi'<0\,.
  \end{align*}
\end{remark}

\subsection{Classical and elementary results about homological asymptotes}\label{subsec:classical-elem-asymp}

The following lemma shows that the homological asymptotes for minimal
geodesics lie in the (closed) unit ball of the stable norm. If $A$ and $B$
are a subsets of an affine space~$V$, then we define their \emph{convex
  join} as
\[ \convjoin(A,B)
\ceq
\bigl\{t\cdot x + (1-t)\cdot y
\bigm| t\in [0,1]\text{ and } x\in A\text{ and }y\in B\bigr\}
\,, 
\]
which is a subset of the convex hull $\conv(A\cup B)$.

\begin{lemma}\label{lem:haproperties}
  Let $M$ be a closed connected Riemannian manifold and 
  let $\gamma \colon \R \to M$ be a $\Cinftypw$-curve
  parametrized by arclength.
  \begin{enumarab}
  \item\label{lem:haproperties.i}
    Let $(t_i)_{i \in \N_0}$ be a sequence in~$\R$ with~$\lim_{i \to
    \infty} t_i = \infty$. Then 
    \[ \limsup_{i \to \infty}
    \frac{\stabnorm{h(\gamma|_{[t_0,t_i]})}}{t_i - t_0}
    \leq 1
    \,.
    \]
  \item\label{lem:haproperties.ii} Thus, $\Apm(\gamma)$ 
    lies in the closed unit ball
    of~$(H_1(M;\R),\stabnorm\argu)$.
  \item\label{lem:haproperties.zero}
    If $(t_i)_{i \in \N_0}$ is a sequence in~$\R$ with~$\lim_{i \to
    \infty} t_i = \infty$, then there exists a subsequence~$(t'_i)_{i
    \in \N}$ of $(t_i)_{i \in \N}$ such that $\lim_{i \to \infty}
    h\bigl(\gamma\stelle{[t_0,t'_i]}\bigr)/(t'_i - t_0)$ exists.
  \item\label{lem:haproperties.iii}
    In particular, ~$A^+(\gamma) \neq \emptyset$ and~$A^-(\gamma) \neq \emptyset$.
  \item\label{lem:haproperties.iv}
    The sets $A^+(\gamma)$ and $A^-(\gamma)$ are closed.
  \item\label{lem:haproperties.v}
    The sets $A^+(\gamma)$ and $A^-(\gamma)$ are connected. 
  \item\label{lem:haproperties.vi}
    We have 
    \[\Apm(\gamma)\subset \Atot(\gamma)\subset \convjoin\bigl(A^+(\gamma),A^-(\gamma)\bigr)
    \,.
    \]
  \item\label{lem:haproperties.vii}
    The set~$\Atot(\gamma)$
    lies in the closed unit ball of~$(H_1(M;\R), \stabnorm\argu)$
    and is closed and conected.
  \end{enumarab}
\end{lemma}

\begin{proof}
  \Adref{lem:haproperties.i}
  As $\gamma$ is a minimal geodesic, we
  have~$\maL(\gamma|_{[t_0,t_i]}) = t_i - t_0$ for all~$i \in \N$.
  Using the estimate between lengths and the stable norm
  (Proposition~\ref{prop:stabnormlength}), we obtain a constant~$c \in
  \R_{\geq 0}$ such that for all~$i \in \N$:
  \[ \frac{\stabnorm{h(\gamma|_{[t_0,t_i]})}}{t_i -t_0}
    \leq \frac{\maL(\gamma|_{[t_0,t_i]}) + c}{t_i-t_0}
    = \frac{t_i - t_0 + c}{t_i - t_0}
  \]
  Because of~$\lim_{i \to \infty} t_i = \infty$, the claim follows.  

  \Adref{lem:haproperties.ii}
  We first prove the statement for~$A^+(\gamma)$. Item~\ititemref{lem:haproperties.i}
  can be restated as
  \begin{equation}\label{lem:haproperties.i.equiv}
    \limsup_{i \to \infty}\stabnorm{
    \frac{h(\gamma|_{[t_0,t_i]})}{t_i - t_0}}
    \leq 1\,.
  \end{equation}
  Thus if $v \in A^+(\gamma)$ is the limit for such a sequence
  $(t_i)_i$, then $\stabnorm{v}\leq 1$.  The proof for~$A^-(\gamma)$ is
  analogous.

  \Adref{lem:haproperties.zero}
  Because of Estimate~\eqref{lem:haproperties.i.equiv}
  we see that $\bigl(h(\gamma|_{[t_0,t_i]}\bigr) / (t_i
  - t_0))_{i \in \N}$ lies in a stable norm ball of finite
  radius. Therefore, compactness of this ball allows us to select a
  subsequence that converges.   
  
  \Adref{lem:haproperties.iii}
  For~$A^+(\gamma)$, we apply Item~\ref{lem:haproperties.zero}
  to the sequence~$(i)_{i \in \N}$. 
  Reversing the orientation, we also get~$A^-(\gamma) \neq \emptyset$.
  
  \Adref{lem:haproperties.iv} We prove the statement
  for~$A^+(\gamma)$. Let $v$ be in the closure of~$A^+(\gamma)$,
  i.e., we have a sequence~$(v_i)_{i \in \N}$ in~$A^+(\gamma)$
  with~$v_i\to v$ for~$i\to \infty$.  For each~$i$,
  we choose a~$t_i$ with
  \[ \stabnorm[Big]{ v_i-    \frac{h\bigl(\gamma\steLLe{[0,t_i]}\bigr)}{t_i}}
  \leq \frac1i
  \,.
  \]
  One can choose the~$t_i$ such that $t_i\nearrow \infty$.
  Then for this sequence~$(t_i)_i$ and~$v$ we have the
  relation in Equation~\eqref{eq.v.def}; thus,  $v\in A^+(\gamma)$.
  The proof of the closedness of~$A^-(\gamma)$ is analogous.

  \Adref{lem:haproperties.v}
  Assume for a contradiction that
  $A^+(\gamma)$ were not connected, \ie there exists open subsets
  $U_0$ and~$V_0$ of~$H_1(M;\R)$ with 
  \begin{equation}\label{e:cond.non-conn}
    A^+(\gamma)\cap U_0 \cap V_0=\emptyset,\;\; A^+(\gamma)\subset U_0 \cup V_0\;\;\text{and}\;\;
  A^+(\gamma)\cap U_0 \neq\emptyset\neq A^+(\gamma)\cap V_0\,.
  \end{equation}
  Then $A^+(\gamma)\cap U_0 = A^+(\gamma)\setminus V_0$ is closed and thus compact. Similarly $A^+(\gamma)\cap V_0$ is compact.
  We thus can replace $U_0$ and $V_0$ in \eqref{e:cond.non-conn} by open subsets $U$ and $V$ with $U\cap V=\emptyset$.
  We choose $u \in A^+(\gamma)\cap U$ and $v\in A^+(\gamma) \cap V$.
  There are sequences~$(s_i)_{i \in \N}$ and $(t_i)_{i \in \N}$ in~$\R_{>0}$
  with $\lim_{i\to\infty} s_i= \infty= \lim_{i\to\infty} t_i$ such that
  \[
  u = \lim_{i \to \infty} \frac{h\bigl(\gamma\steLLe{[0,s_i]}\bigr)}{s_i}
  \qand
  v = \lim_{i \to \infty} \frac{h\bigl(\gamma\steLLe{[0,t_i]}\bigr)}{t_i}\,.
  \]
  Passing to subsequences we can achieve $s_1<t_1<s_2<t_2<s_2\ldots$ and
  \[\lim_{i \to\infty} \frac{h\bigl(\gamma\steLLe{[0,s_i]}\bigr)}{s_i}\in U
  \qand \lim_{i \to \infty} \frac{h\bigl(\gamma\steLLe{[0,t_i]}\bigr)}{t_i}\in V\,.
  \]
  By continuity, for each~$i \in \N$, we can choose an~$r_i\in (s_i,t_i)$ with
  \[w_i\ceq \frac{h\bigl(\gamma\steLLe{[0,r_i]}\bigr)}{r_i}\notin U \cup V\,.
  \]
  Moreover, by Lemma~\ref{lemma.h.rho}, we have
  $w_i\in \overline B_\maJ(0)$, 
  where~$\maJ$ is defined as in Equation~\eqref{eq.def.maJ}.  Thus, 
  after passing to further subsequences we obtain that
  $w_\infty:=\lim_{i\to\infty} w_i$ exists and lies in~$\overline
  B_\maJ(0)\setminus (U \cap V)$. (Alternatively, we could use
  Item~\ititemref{lem:haproperties.zero} here again.)  By defintion
  of~$A^+(\gamma)$, we also have~$w_\infty\in A^+(\gamma)$, which  
  contradicts the assumption that~$A^+(\gamma) \subset U \cup V$.

  By reversing orientations, we also see that $A^-(\gamma)$ is connected.
  
 \Adref{lem:haproperties.vi}
  We first show the inclusion $\Apm(\gamma)\subset \Atot(\gamma)$.  
  For a given~$v\in A^+(\gamma)$ we choose a sequence~$t_i\nearrow
  \infty$ with~$h\bigl(\gamma\stellle{[0,t_i]}\bigr)/t_i\to v$. After passing to
  a subsequence, we may assume~$t_i\geq i^2$ for all~$i \in \N$. 
  With Lemma~\ref{lemma.h.rho} we see that 
  $\stabnorm[big]{h\bigl(\gamma\stellle{[a,b]}\bigr)}\leq \maJ\cdot (b-a)$.
  We get 
  \begin{align*}
    \stabnorm{\frac{ h\bigl(\gamma\stellle{[-i,t_i]}\bigr)}{t_i+i}-\frac{ h\bigl(\gamma\stellle{[0,t_i]}\bigr)}{t_i}}
    &\leq  \stabnorm{\left(\frac{t_i}{t_i+i}-1\right)\frac{ h\bigl(\gamma\stellle{[0,t_i]}\bigr)}{t_i}}+ \stabnorm{\frac{ h\bigl(\gamma\stellle{[-i,0]}\bigr)}{t_i+i}}\\
    &\leq  \left|\frac{i}{t_i+i}\right|\cdot\frac{\stabnorm[big]{h\bigl(\gamma\stellle{[0,t_i]}\bigr)}}{t_i} + \frac{\maJ\cdot i}{|t_i+i|}\\
    &\leq  \frac{i}{i^2+i}\cdot \maJ + \frac{\maJ\cdot i}{i^2+i} \leq \frac{2\maJ}i
\end{align*} 
and thus 
\[v=\lim_{i\to\infty}\frac{ h\bigl(\gamma\stellle{[0,t_i]}\bigr)}{t_i}=\lim_{i\to\infty}\frac{ h\bigl(\gamma\stellle{[-i,t_i]}\bigr)}{t_i+i}\in \Atot(\gamma)\,.\]
So we have $A^+(\gamma)\in \Atot(\gamma)$. 
The proof for the inclusion $A^-(\gamma)\subset\Atot(\gamma)$ is obtained by reversing the orientation.
In total, we get~$\Apm(\gamma)\subset \Atot(\gamma)$.

  Now, we will prove the inclusion $\Atot(\gamma)\subset \convjoin\bigl(A^+(\gamma),A^-(\gamma)\bigr)$. 
  Suppose that $v\in \Atot(\gamma)$ may be obtained as 
  \[ v = \lim_{i \to \infty}
         \frac{h\bigl(\gamma\stellle{[s_i,t_i]}\bigr)}{t_i - s_i}\,.
  \]
  with~$\lim_{i \to \infty} s_i = -\infty$ and~$\lim_{i \to \infty}
  t_i = \infty$. By Item~\ititemref{lem:haproperties.zero}, after passing
  to subsequences, the individual limits
  \[\tau\ceq \lim_{i\to\infty} \frac{t_i}{t_i-s_i},\quad
  v_+\ceq \lim_{i\to\infty}\frac{h\bigl(\gamma\stellle{[0,t_i]}\bigr)}{t_i},\qand
  v_-\ceq \lim_{i\to\infty}\frac{h\bigl(\gamma\stellle{[s_i,0]}\bigr)}{-s_i}
  \]
  exist and we have $\tau\in[0,1]$, $v_+\in A^+(\gamma)$, and $v_-\in
  A^-(\gamma)$. We calculate
  \begin{align*}
    \frac{h\bigl(\gamma\stellle{[s_i,t_i]}\bigr)}{t_i - s_i}
    =\,&\frac{t_i}{t_i - s_i}\frac{h\bigl(\gamma\stellle{[0,t_i]}\bigr)}{t_i}
    + \frac{-s_i}{t_i - s_i}\frac{h\bigl(\gamma\stellle{[s_i,0]}\bigr)}{-s_i}
    \\
    \xrightarrow{i\to\infty}
    \,&\tau\cdot v_++ (1-\tau) \cdot v_- \in \convjoin\bigl(A^+(\gamma),A^-(\gamma)\bigr) 
    \,.
  \end{align*}

  \Adref{lem:haproperties.vii}  
  It follows from Items \ititemref{lem:haproperties.ii}
  and~\ititemref{lem:haproperties.vi}, and from the convexity of norm balls,
  that $\Atot(\gamma)$ is contained in the stable norm unit ball.

  To prove closedness of $\Atot(\gamma)$ we proceed
  as in Part~\ititemref{lem:haproperties.iv} with $A^+$ replaced by~$\Atot$,
  and with $h(\gamma\steLLe{[0,t_i]})/t_i$ replaced
  by~$h(\gamma\steLLe{[s_i,t_i]})/(t_i-s_i)$. Again the choices
  can be done such that $t_i\nearrow \infty$ and
  $s_i\searrow-\infty$. As before, this implies $v\in \Atot(\gamma)$.
  
  For connectedness, one can use similar arguments as in the proof
  of Item~\ititemref{lem:haproperties.v}.
\end{proof}

\subsection{Asymptotes of homologically minimal geodesics}\label{subsec:asympt.hom.min}
Asymptotes of $\Z$\=/ho\-mo\-lo\-gi\-cal\-ly minimal geodesics
(Definition~\ref{def:homologmin}) lie in the
stable norm unit sphere:

\begin{proposition}\label{prop.hom.min.sphere}
  Let $M$ be a closed connected Riemannian manifold and
  let $\gamma \colon \R \to M$ be a $\Z$\=/ho\-mo\-lo\-gi\-cal\-ly minimal
  geodesic of~$M$.
  Then $A^+(\gamma)$, $A^-(\gamma)$, and~$\Atot(\gamma)$ 
  are subsets of the stable norm unit sphere
  \[\bigl\{x\in H_1(M;\R)\bigm| \stabnorm{x}=1\bigr\}\,.
  \]
\end{proposition}

As $\R$\=/homological minimality implies $\Z$\=/homological minimality,
the proposition also holds for all $\R$\=/ho\-mo\-lo\-gi\-cal\-ly
minimal geodesics.

\begin{proof}
  It suffices to prove the statement for~$\Atot(\gamma)$.
  Let $v \in\Atot(\gamma)$. As a $\Z$\=/ho\-mo\-lo\-gi\-cal\-ly minimal geodesic
  is minimal, we know that~$\stabnorm{v}\leq 1$
  from Lemma~\ref{lem:haproperties}~\itemref{lem:haproperties.vii}.

  We prove~$\stabnorm v \geq 1$: 
  For $L>0$,  we choose~$s,t\in \R$ with $s+L<t$ such that 
  \[\stabnorm[Big]{v-\frac{h\bigl(\gamma\stellle{[s,t]}\bigr)}{t - s}}
  <\epsilon\ceq\frac1L\,.
  \]
  
  Let $\hat\gamma$ be a lift of $\gamma$ to $\homZcov M$. 
  We apply Proposition~\ref{prop.finite.dist} \ref{prop.finite.dist.3} for $\wihat M=\homZcov M$, using the definitions of $\wihat J$, $\hat d$ and $C$ from this proposition, and we obtain:
  \begin{align}
    \left(\stabnorm{v}+\epsilon\right)\cdot (t-s)
    &\geq \stabnorm[big]{h\bigl(\gamma\stellle{[s,t]}\bigr)}\nonumber \\
    &\geq \stabnorm[big]{\wihat J\bigl(\hat\gamma(t)\bigr)-\wihat J\bigl(\hat\gamma(s)\bigr)}\nonumber \\
    &\geq \hat d\bigl(\hat\gamma(t),\hat\gamma(s)\bigr)-C\nonumber \\
    &\geq (t-s)-C\,.\nonumber
  \end{align}
  This yields
  \begin{align*}
    \stabnorm{v}&\geq 1-\frac{C}{t-s}-\epsilon \geq 1-\frac{C}{L}-\epsilon 
  \end{align*} 
  and in the limit $L=\epsilon^{-1}\to \infty$ we finally obtain $\stabnorm{v}\geq 1$.
\end{proof}

\subsection{Comparison to Bangert's rotation vectors}\label{subsec.rot.vector}
  Initial and terminal asymptotes are related to Bangert's ``rotation
  vectors''~\cite{bangert:90,ammann:diploma}. For historical clarity,
  let us explain this, although it is not required for the article's
  main results.
  
  For a curve~$\gamma\colon[a,b]\to M$, Bangert defines the \emph{rotation
  vector} as
  \[R(\gamma)\ceq \frac{h(\gamma)}{\stabnorm{h(\gamma)}}
  \,.
  \]
  Now let $\gamma\colon\R\to M$ be a minimal geodesic.  In Bangert's
  work~\cite{bangert:90}, the accumulation points
  of~$R\bigl(\gamma\steLLe{[s,t]}\bigr)$ with~$t-s\to \infty$ play an
  important role. If additionally $\gamma$ is
  $\Z$\=/ho\-mo\-lo\-gi\-cal\-ly minimal, then he shows 
  that these accumulation points are in
  the intersection of the stable norm unit sphere and a supporting
  hyperplane which implies a statement similar to our
  Proposition~\ref{prop.hom.min.sphere}~\cite[Theorem~3.2]{bangert:90}.
  More precisely:
  
  \begin{theorem}[{Bangert~\cite[Theorem~3.2]{bangert:90}}]\label{thm.bangert.zwei}
    Let $M$ be a closed connected Riemannian manifold and let~$B$
    be the stable norm unit ball in~$H_1(M;\R)$. 
    Let~$\gamma$ be a $\Z$\=/ho\-mo\-lo\-gi\-cal\-ly minimal
    geodesic. Then there exists a supporting hyperplane $H$ to $B$
    with the following property: for every neighborhood $U$ of $H\cap
    B$ there exists $c_U> 0$ such that
    $R\bigl(\gamma\stellle{[s,t]}\bigr)\in U$ whenever $t-s > c_U$.
  \end{theorem}
  
  Conversely, Bangert proves that for every supporting hyperplane a
  $\Z$\=/ho\-mo\-lo\-gi\-cal\-ly minimal geodesic exists such that the accumulation
  points of the rotation vectors lie in the intersection of that
  hyperplane with the stable norm unit
  sphere~\cite[Theorem~4.4]{bangert:90}.

  Rotation vectors and asymptotes are related as follows:
  
\begin{proposition}\label{prop.atot} 
  Let $M$ be a closed connected Riemannian manifold and let $\gamma
  \colon \R \to M$ be a $\Z$\=/ho\-mo\-lo\-gi\-cal\-ly minimal
  geodesic.  Then there are constants~$C_0$, $c_0>0$ such that
  \[
  \fa{s,t \in \R} t-s \geq c_0 \Longrightarrow
  \stabnorm[bigg]{R\bigl(\gamma\stellle{[s,t]}\bigr)
    - \frac{h(\gamma\stellle{[s,t]})}{t-s}}
  \leq \frac{C_0}{t-s}
  \,.
  \]
\end{proposition}

\begin{proof}
  Let $\hat\gamma$ be a lift of $\gamma$ to the covering $\homZcov
  M$. By $\Z$\=/ho\-mo\-lo\-gi\-cal minimality,
  $\hat\gamma$ is a geodesic line, \ie for all $s<t$ we have $\hat d\bigl(\hat
  \gamma(t),\hat \gamma(s)\bigr)=t-s$.

  We calculate for the constant $C$ provided by
  Proposition~\ref{prop.finite.dist}~\ititemref{prop.finite.dist.3} that
  \begin{eqnarray*}
    \bigl|\stabnorm[big]{h\bigl(\gamma\stellle{[s,t]}\bigr)}- (t-s)\bigr|
    &=& \bigl|\stabnorm[big]{\wihat J\bigl(\hat\gamma(t)\bigr)- \wihat J\bigl(\hat\gamma(s)\bigr)}- (t-s)\bigr|\\
    &\leq &  \bigl|\underbrace{\hat d\bigl(\hat \gamma(t),\hat \gamma(s)\bigr)- (t-s)}_{=0}\bigr|+C \,=\, C\,.
  \end{eqnarray*}
  Moreover, we have 
  \begin{eqnarray*}
    \stabnorm[bigg]{R\bigl(\gamma\stellle{[s,t]}\bigr) - \frac{h(\gamma\stellle{[s,t]})}{t-s}}
    &\leq & \biggl|1-\frac{\stabnorm[big]{h(\gamma\stellle{[s,t]})}}{t-s} \biggr|\cdot \stabnorm[big]{R\bigl(\gamma\stellle{[s,t]}\bigr)} \\
    &\stackrel{(*)}{\leq}& \frac{C\cdot(1+\epsilon)}{t-s}\,;
  \end{eqnarray*}
  for the inequality~$(*)$ we used that Bangert's
  Theorem~\ref{thm.bangert.zwei} implies
  $\stabnorm[big]{R(\gamma\stellle{[s,t]})}\leq 1+\ep$ for
  some~$\epsilon$ that goes to~$0$ uniformly in $t-s\to
  \infty$. (Alternatively, we could use
  Proposition~\ref{prop.hom.min.sphere} for that purpose). The
  statement follows for~$C_0\ceq C(1+\epsilon)$.
\end{proof}

\begin{corollary}\label{cor.atot} 
  We assume the conditions of Proposition~\ref{prop.atot}.
  Then
  \begin{equation*}
    \Atot(\gamma)
    = \bigl\{ \lim_{i\to \infty}R\bigl(\gamma\stellle{[s_i,t_i]}\bigr)
    \bigm| t_i\nearrow +\infty\text{ and }s_i\searrow -\infty
    \text{ and the limit exists}
    \bigr\}\,.
  \end{equation*}  
\end{corollary}
\begin{proof}
  This is immediate from Proposition~\ref{prop.atot}. 
\end{proof}

We obtain the following improvement of Proposition~\ref{prop.hom.min.sphere}:

\begin{corollary}\label{cor.prop.hom.min.sphere.improved}
  Let $M$ be a closed connected Riemannian manifold and
  let $\gamma \colon \R \to M$ be a $\Z$\=/ho\-mo\-lo\-gi\-cal\-ly minimal
  geodesic. Let $B$ be the (closed) stable norm unit ball.
  Then there is a supporting hyperplane~$H$ for~$B$ with 
  \[\conv \Apm(\gamma)\subset \bigl\{x\in H_1(M;\R)\bigm| \stabnorm{x}=1\bigr\}\cap H\,.
  \]
\end{corollary}

\begin{proof}
  As the right-hand side is convex, it suffices to prove the statement
  for $A^+(\gamma)$ and $A^-(\gamma)$ instead of $\conv \Apm(\gamma)$.
  We already know from Proposition~\ref{prop.hom.min.sphere} that
  $A^\pm(\gamma)\subset \bigl\{x\in H_1(M;\R)\bigm|
  \stabnorm{x}=1\bigr\}$. From Theorem~\ref{thm.bangert.zwei} and
  Proposition~\ref{prop.atot}, we get a supporting hyperplane~$H$ such
  that $A^\pm(\gamma)\subset H$; the statement is thus shown.
\end{proof}

\begin{remark}
  For general minimal geodesics (\ie without the assumption
  ``$\Z$\=/ho\-mo\-lo\-gi\-cal\-ly''), the relationship of rotation
  vectors and asymptotes is not so immediate and they are out of the
  main focus of Bangert's work~\cite{bangert:90}. We state some
  relations without proof: If a vector~$v\in H_1(M;\R)$ is an
  accumulation point for~$R\bigl(\gamma\stellle{[s_i,t_i]}\bigr)$ for
  some $\R$-valued sequences $(s_i)_{i \in \N}$ and $(t_i)_{i \in \N}$
  with~$t_i\nearrow +\infty$ and $s_i\searrow -\infty$, then there is
  a $w\in \Atot(\gamma)$ auch that
  \begin{equation}\label{eq.relation.asymptote.rotation}
    \stabnorm{w}\cdot v
    = w\,.
  \end{equation}
  Note that we did not exclude the case $w=0$, which arises in
  examples (Example~\ref{ex.higher-genus}).  Conversely, for $w\in
  \Atot(\gamma)\setminus\{0\}$, $v\ceq w/ \stabnorm{w}$ is an
  accumulation point for~$R\bigl(\gamma\stellle{[s_i,t_i]}\bigr)$
  with~$t_i\nearrow +\infty$ and $s_i\searrow -\infty$.
\end{remark}

\subsection{Homologically homoclinic and heteroclinic minimal geodesics}\label{subsec:homoclinic}

The definitions in this section apply in any dimension, but
they are motivated by Hedlund examples, which exist for~$\dim M\geq
3$. Several notions introduced in this subsection come from dynamics, as --
besides the question of existence of minimal geodesics -- the
dynamical properties of the geodesic flow close to minimal geodesics
are of interest.

Hedlund
examples~\cite[Section~9]{hedlund:1932}\cite[Section~5]{bangert:90}\cite{ammann:diploma,ammann:97}
have a finite number of simple closed geodesics; if parametrized by
arclength as a periodic curve~$\tau_i\colon\R\to M$, these simple
closed geodesiscs are also minimal geodesics, see
Section~\ref{subsec:Hedlund.examples}. In the present article, such
curves will be called \emph{minimal closed geodesics}. Let us mention
that the notion ``minimizing closed geodesics'' is also used in the
literature.  If $(M,g)$ is a Hedlund example, one has a finite
number~$\ell$ of minimal closed geodesics, thus we assume~$i\in\{1,2,\cdots,\ell\}$.

\begin{definition}\label{def:homo_heteroclinic}
  Let $(M,g)$ be a closed Riemannian manifold with closed minimal
  geodesics $\tau_1,\dots, \tau_\ell$.

  A minimal geodesic~$\gamma$, parametrized by arclength, is called
  \emph{asymptotic} to~$\tau_i$ at~$\pm \infty$ if there is a
  sign~$\ep\in \{-1,1\}$ and an~$\alpha\in \R$ such that
  \[ \lim_{t\to \pm\infty}
  d\bigl(\gamma(t),\tau_i\bigl(\epsilon\cdot (t-\alpha)\bigr)\bigr)
  = 0
  \,.
  \]
  If $\gamma$ is a minimal geodesic, following Poincar\'e we say that $\gamma$ is
  \begin{itemize}
  \item \emph{homoclinic} if there is an~$i\in\{1,\ldots,\ell\}$ such
    that $\gamma$ is asymptotic to~$\tau_i$ at $+\infty$
    and~$-\infty$. This includes the cases
    $\gamma=\tau_i$ and~$\gamma=\overline\tau_i$.
  \item \emph{heteroclinic} if there are
    $i_+,i_-\in\{1,\ldots,\ell\}$ with~$i_+\neq i_-$ such that $\gamma$ is
    asymptotic to~$\tau_{i_\pm}$ at~$\pm\infty$.
  \end{itemize}
\end{definition}

If the fundamental group~$\pi_1(M)$ of a closed connected manifold~$M$
is virtually abelian (or virtually nilpotent of ``bounded minimal
generation''~\cite[Definition~7.1]{ammann:97}\cite[Definition~IV.1.19]{ammann:diploma}),
then for a suitable Hedlund metric on~$M$, every minimal geodesic is
either homoclinic or heteroclinic~\cite[Hilfsatz~IV.1.9 and
  Korollar~IV.1.21]{ammann:diploma}.
 
\begin{history}
  The existence of homoclinic and heteroclinic minimal geodesics was
  also studied by Bolotin and Rabinowitz, in the special case $M=T^n$,
  under some weak assumptions. We will explain this in Subsection~\ref{subsec:compar_BR} in more details.
\end{history}

On general Riemannian manifolds, this motivates the introduction of
similar terminology, which will be useful to distinguish different
classes of minimal geodesics. All these notions only depend on the
geometric equivalence class of the minimal geodesic.
  
\begin{definition}[types of minimal geodesics]\label{def.homolo.homocli}
  A minimal geodesic~$\gamma$ on a closed connected Riemannian
  manifold~$M$ is called
  \begin{itemize}
  \item \emph{homologically homoclinic} if there is a~$v\in H_1(M;\R)$ with
    \[ A^+(\gamma) = \{v\} = A^-(\gamma)
    \,.
    \]
  \item \emph{homologically heteroclinic} if there
    are~$v,w\in H_1(M;\R)$ with~$v\neq w$ and 
    \[ A^+(\gamma)=\{v\}
    \qand
    A^-(\gamma)=\{w\}
    \,.
    \]
  \item \emph{homologically diverging} if $A^+(\gamma)$ or
    $A^-(\gamma)$ (or both of them) contain more than one point.
  \item \emph{homologically semi\-/con\-ver\-ging} if $A^+(\gamma)$ or
    $A^-(\gamma)$ (or both of them) contain precisely one point.
  \item \emph{homologically exposed} if there are exposed points~$x,y$
    in the stable norm unit ball~$B$ with $A^+(\gamma)=\{x\}$ and
    $A^-(\gamma)=\{y\}$.  We allow both the homoclinic ($x=y$) and the
    hereroclinic ($x\neq y$) case.
  \item \emph{homologically semi\-/ex\-posed}
    if there is an exposed point~$x$ of~$B$ with
    $A^+(\gamma)=\{x\}$ or $A^-(\gamma)=\{x\}$.
  \item \emph{homologically non-homoclinic} if it is not homologically
    homoclinic.  By definition, this is equivalent to being either
    homologically heteroclinic or homologically diverging.
   \end{itemize}
\end{definition}

All homoclinic/heteroclinic minimal geodescis in the sense of
Definition~\ref{def:homo_heteroclinic} are homologically
homoclinic/heteroclinic, respectively, provided that the
homology classes~$[\tau_1], \dots, [\tau_\ell] \in H_1(M,\R)$ are pairwise
different. By definition, every minimal geodesic is either
homologically homoclinic, homologically heteroclinic or homologically
diverging. Furthermore, ``homologically semi\-/ex\-posed'' implies
``homologically semi\-/con\-ver\-ging''.

\begin{example}\label{exam.tk.sl}Let the torus~$T^k=\R^k/\Z^k$
  carry a flat metric and let $N$ be a closed connected Riemannian
  manifold with finite~$\pi_1(N)$.  We equip~$M \ceq T^k\times N$ with
  the product metric. Then a curve
  $\gamma = (\gamma_T,\gamma_N)$ is a geodesic if and
  only if both $\gamma_T$ and~$\gamma_N$ are geodesics, not
  necessarily parametrized by arclength.  Further, $\gamma$ is a
  minimal geodesic if and only if $\gamma_N$ is constant and
  $\gamma_T$ a minimal geodesic. The latter condition is equivalent to
  saying that after identifying $\ucov {T^k}$ with Euclidean~$\R^k$,
  the component~$\gamma_T$ is an affine line.  In particular, all minimal geodesics
  on~$M$ are $\R$\=/ho\-mo\-lo\-gi\-cal\-ly minimal, homologically
  homoclinic, homologically (semi\=/)\allowbreak{}con\-ver\-ging, and homologically
  (semi\=/)\allowbreak{}ex\-posed.
\end{example}  

\section{Finding new minimal geodesics}\label{sec:newmingeod}

We explain how points on the exposed edges of the stable norm unit ball lead
to minimal geodesics and how their terminal/initial asymptotes are
controlled by the corresponding exposed edge. In particular, we will see how
these minimal geodesics can be distinguished via their
terminal/initial asymptotes.

\subsection{Setup for the construction}

Let $M$ be a closed connected Riemannian manifold, let $b := \dim_\R H_1(M;\R)$, and let $B \subset H_1(M;\R)$ be the stable norm unit
ball, i.e., the closed unit ball of~$\bigl(H_1(M;\R),\stabnorm\argu\bigr)$.
The boundary of~$B$ is the stable
norm unit sphere~$\partial B=\bigl\{x\in H_1(M;\R)\bigm|
\stabnorm{x}=1\bigr\}$. An \emph{exposed edge} of~$B$ is a
subset~$E\subset\partial B$ that is the convex hull of two different
points~$x,y\in \partial B$ such that there is an $\omega \in \HdR$
with $\stabnormdual{\omega}=1$ with $E=F_\omega$, where $F_\omega$ is defined in Definition~\ref{def:face}. 
The second condition can be equally expressed by saying that there is 
a hyperplane $H$ of $H_1(M;\R)$ supporting $B$ and satisfying $E=H\cap B=H\cap \partial B$. 
We will write~$[x,y]$ for the convex hull
of $x$ and~$y$, viewed as a $1$-dimensional submanifold with boundary with an orientation from $x$ any $y$. Hence $[x,y]$ and $[y,x]$ are equal as submanifolds, but with opposite orientations.
Moreover, we write $(x,y)\ceq [x,y]\setminus\{x,y\}$
for the relative interior of $[x,y]$. If $\dim_\R H_1(M;\R) \leq 1$,
then $B$ does not have any exposed edges.

Moreover, we choose a $\ZZ$-module basis~$(\beta_1, \dots, \beta_b)$
of~$\HZR$ and closed $1$-forms~$\alpha^1, \dots, \alpha^b$ with
$\<[\alpha^i],\beta_j\>=\delta_{ij}$ as in Section~\ref{subsec:curveshom}; 
in particular, we obtain an associated geometric Hurewicz map~$h$.

\subsection{Construction of minimal geodesics from exposed edges of the stable norm unit ball}
\label{subsec:sinfty}

\def\vertx#1{%
  \fill #1 circle (0.1cm);}

Let $[x,y] \subset \partial B$ be an exposed edge of the stable norm
unit ball~$B$ (Figure~\ref{fig:constrsinfty}) and choose
$z\in(x,y)$. By definition of an exposed edge -- Definition~\ref{def:face} -- there is an~$\omega \in \HdR$
with $\stabnormdual{\omega}=1$ and
\[ \< \omega, z \> = 1 = \stabnorm z 
   \qand
   [x,y] = F_\omega
   \,.
\]
In addition, we choose a closed smooth $1$-form~$\eta \in
\Omega^1(M)$ that is a linear combination of
$\alpha^1,\ldots,\alpha^b$ and that satisfies
\[ \< [\eta] , z \> = 0
   \qand
   \< [\eta], x \> < 0 < \< [\eta], y \>\,.
\]
  
  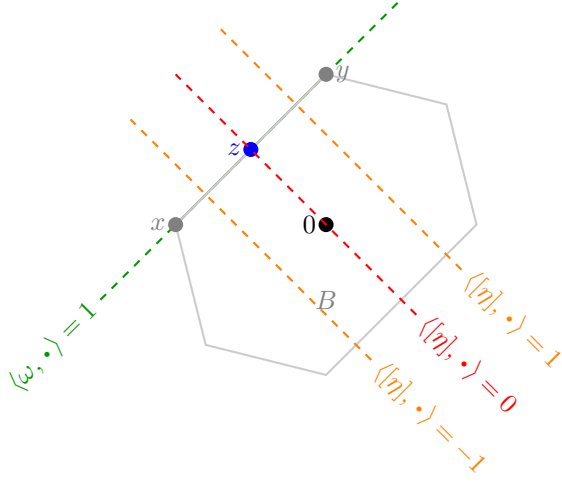
\begin{figure}
  \begin{center}
    \begin{tikzpicture}[x=2cm,y=2cm,thick]
      \begin{scope}[darkgreen]
        \draw[dashed] (-1.5,-0.5) -- +(2,2);
	\draw (-1.5,-0.5) node[anchor=east,rotate=45] {$\<\omega,\args\> = 1$};
      \end{scope}
      \begin{scope}[black!50]
        \draw (-1,0) -- (0,1);
        \draw[black!20] (-1,0) -- (0,1) -- (.8,.8) -- (1,0) -- (0,-1) -- (-.8,-.8) -- cycle;
	\draw (0,-0.5) node {$B$};
        \vertx{(-1,0)}
        \vertx{(0,1)}
	\draw (-1,0) node [anchor=east] {$x$};
	\draw (0,1) node [anchor=west] {$y$};
      \end{scope}
      \begin{scope}[blue]
        \vertx{(-0.5,0.5)}
        \draw (-0.5,0.5) node [anchor=east] {$z$};
      \end{scope}
      \begin{scope}[black]
        \vertx{(-0,0)}
        \draw (-0,0) node [anchor=east] {$0$};
      \end{scope}
      \begin{scope}[red]
        \draw[dashed] (0.6,-0.6) -- +(-1.6,1.6);
	\draw (0.6,-0.6) node [anchor=west,rotate=-45] {$\<[\eta],\args\>=0$};
      \end{scope}
      \begin{scope}[orange]
        \draw[dashed] (0.9,-0.3) -- +(-1.6,1.6);
        \draw (0.9,-0.3) node [anchor=west,rotate=-45] {$\<[\eta],\args\> =1$};
        \draw[dashed] (0.3,-0.9) -- +(-1.6,1.6);
        \draw (0.3,-0.9) node [anchor=west,rotate=-45] {$\<[\eta],\args\> =-1$};
      \end{scope}
    \end{tikzpicture}
  \end{center}

    \caption{The situation in the construction of Section~\ref{subsec:sinfty}}
    \label{fig:constrsinfty}
  \end{figure}

  As starting point for our construction, we pick a controlled
  approximation of~$z$:

  \begin{lemma}\label{lem:constructionzi}
    In this situation, there exists a sequence~$(z_i)_{i \in \N}$
    in~$\HZR \setminus \{0\}$ and a constant~$a \in \R_{\geq 0}$ such
    that 
    \[ \lim_{i \to \infty} \frac{z_i}{\stabnorm{z_i}} = z,\quad
      \fa{i \in \N} \langle \omega,z_i \rangle \geq N(z_i) - a,
      \qand \lim_{i\to\infty}\stabnorm{z_i}=\infty\,.
    \]
  \end{lemma}
  
  The proof even gives the additional information that
  $\left\{\stabnorm{z_i-i\cdot z}\mid i\in \N\right\}$ is bounded, but
  this will not be needed later.
  \begin{proof}
    Because~$\HZR$ is a cocompact lattice in~$H_1(M;\R)$, there
    exists an~$r \in \R_{>0}$ with
    \[ \fa{v \in H_1(M;\R)}
    \exi{v' \in \HZR }
    \stabnorm{v - v'} \leq r
    \,.
    \]
    For each~$i \in \N$, we choose~$z_i \in \HZR$
    with~$\stabnorm{i \cdot z - z_i} \leq r$.
 
    By construction,
    \begin{align*}
      \stabnorm{z_i}
      & \leq \stabnorm{i \cdot z} + \stabnorm{z_i - i \cdot z} \leq i + r\,,
      \\
      \stabnorm{z_i}
      & \geq \stabnorm{i \cdot z} - \stabnorm{z_i - i \cdot z} \geq i -r\,.
    \end{align*}
    Thus, after removing finitely many~$z_i$  we will have $z_i\neq 0$ and 
    $\lim_{i \to \infty} i /\stabnorm{z_i} = 1$.
    Therefore, we obtain
    \[ \lim_{i \to \infty} \frac{z_i} {\stabnorm{z_i}}
    = \lim_{i \to \infty} \frac{i \cdot z}{i}
    + \lim_{i \to \infty} \frac{z_i - i \cdot z}{i}
    = z + 0
    \,.
    \]    
    Moreover, for all~$i \in \N$, we have
    \begin{align*}
      \langle \omega, z_i \rangle - \stabnorm{z_i}
      & \geq \langle \omega,i\cdot z\rangle + \langle \omega, z_i - i \cdot z\rangle
      - \stabnorm{i \cdot z} - \stabnorm{z_i - i \cdot z}
      \\
      & \geq i - \stabnormdual{\omega} \cdot \stabnorm{z_i - i\cdot z}
      - i -r
      \\
      & \geq -2 \cdot r\,.
    \end{align*}
    Because $N$ and $\stabnorm{\args}$ are uniformly close
    (Theorem~\ref{thm:stabnormN}), the claim follows -- with the
    constant~$a \ceq 2 \cdot r + D$, where~$D$ is provided
    by Theorem~\ref{thm:stabnormN}; 
    a possible concrete choice for~$D$ is given by
    Equation~\eqref{eq.D.def}.
  \end{proof}
  
  Let $(z_i)_{i \in \N} \subset \HZR \setminus\{0\}$
  and $a \in \R_{\geq 0}$ be as provided by Lemma~\ref{lem:constructionzi}. 
  As next step, we choose corresponding $1$-forms:
 
  \begin{lemma}\label{lem:constructionetai}
    Choose a number~$\ell\in \{1,\dots, b\}$ such that
    $\<[\alpha^\ell],z\>\neq 0$.  In the situation described above, we
    can find a sequence~$(\lambda_{i})_{i \in \N}$ in~$\R$ 
    with~$\lim_{i \to \infty} \lambda_{i} = 0$ such that
    \[ \eta_i := \eta +  \lambda_{i} \cdot \alpha^\ell
    \]
    satisfies~$\bigl\< [\eta_i], z_i \bigr\> = 0$ for all~$i \in \N$.
    In particular, each~$\eta_i$ is an $\R$-linear combination
    of~$\alpha^1,\ldots,\alpha^b$.
  \end{lemma}
  \begin{proof}
   \newcommand\muu[2]{\mu_{#2}^{#1}}
    For~$i \in \N$, we abbreviate~$\overline z_i := z_i/\stabnorm{z_i}$
    and write
    \[ \overline z_i = \sum_{j=1}^b \muu{j}{i} \cdot \beta_j
    \qand 
     z = \sum_{j=1}^b \mu^j \cdot \beta_j
    \]
    in the chosen basis~$(\beta_1, \dots, \beta_b)$ of~$H_1(M;\R)$.
    Then $(\muu{j}{i})_{i \in \N}$ converges to $\mu^j$. As
    $([\alpha^1],\dots,[\alpha^b])$ is the dual basis we have $\mu^j=
    \<[\alpha^j],z\>$, and thus by assumption~$\mu^\ell\neq 0$.
    Passing to a subsequence, we may achieve~$|\muu{\ell}{i}|\geq
    |\mu^\ell|/2>0$ for all~$i \in \N$.  For~$i \in \N$, we set
    \[ \lambda_i \ceq  -\bafrac1{\muu{\ell}{i}} \cdot \<[\eta],\overline z_i\>\,.\]
    This sequence has the desired properties:
    Because the sequence~$(|1/\muu{\ell}{i}|)_{i \in \N}$ is
    bounded by~$\leq 2/|\mu^\ell|$ 
    and $\lim_{i \to\infty} \<[\eta],\overline z_i\>
    = \<[\eta], z\> = 0$,
    we obtain $\lim_{i \to \infty} \lambda_{i}
    = 0$.
    Moreover, for all~$i \in \N$, duality of the bases leads to 
    \begin{align*}
      \bigl\< [\eta_i], \overline z_i \bigr\rangle
      &
      = \bigl\< [\eta], \overline z_i \bigr\>
      +  \lambda_i
      \cdot \biggl\< [\alpha^\ell], \sum_{j=1}^b \muu{j}{i} \cdot \beta_j \biggr\>
      \\
      & 
      = \bigl\< [\eta], \overline z_i \bigr\>
      + \sum_{j=1}^b \lambda_i \cdot \delta_{j\ell}\cdot \muu{j}{i}
      \\
      & = \bigl\< [\eta], \overline z_i \bigr\>
      + \lambda_i \cdot \muu{\ell}{i}
      \\
      & = 0\,.
      \qedhere
    \end{align*}    
\end{proof}

Our goal now is to construct new minimal geodesics. Additionally we
also want to show that these geodesics are
$\R$\=/ho\-mo\-lo\-gi\-cal\-ly minimal geodesics. We will achieve this
by constructing a geodesic~$\R\to\homRcov M$ that globally minimizes
the distance.  Readers of this construction may replace~$\homRcov M$
by the universal covering~$\ucov M$ in their first reading, as this
case might be more transparent and already allows to understand why
the constructed geodesics are minimal (without the stronger conclusion
of $\R$\=/homological minimality).
  
{}From now on, let $(\eta_i)_{i \in \N}$ be as provided by Lemma~\ref{lem:constructionetai}.  
For each~$i \in \N$, we choose a
closed smooth geodesic~$\sigma_i$ (parametrised by arclength) of
length~$\ell_i := N(z_i)$ representing~$z_i$ and a lift~$\homcov
\sigma_i \colon \R \to \homRcov M$ to~$\homRcov M$. 
By Theorem~\ref{thm:stabnormN} and Lemma~\ref{lem:constructionzi}, we have $\ell_i=\groupnorm{z_i}\geq
\stabnorm{z_i}\to \infty$.  The function
\begin{equation}\label{eq:def.ui}
  u_i \colon \R \to \R.\quad    t \mapsto \int_{\sigma_i\steLLe{[0,t]}} \eta_i
\end{equation}
is smooth and $\ell_i$-periodic; indeed, for all~$t \in \R$, we have
\[ u_i(t + \ell_i) - u(t)
  = \int_{\sigma_i\steLLe{[t,t+\ell_i]}} \eta_i
  = \bigl\< [\eta_i], z_i \bigr\>
  = 0
  \,.
\]
By reparametrizing~$\sigma_i$ (and $\homcov \sigma_i$), we can achieve
that $u_i$ attains its minimum at~$t = 0$. In particular, $u_i(0)=0$
and~$u_i \geq 0$.

As $M$ is compact and as $\homRcov M \to M$ is a normal covering, we
may assume (by translating the lifts via deck transformations and
passing to a suitable subsequence) that the sequence~$(\dot{\homcov
  \sigma}_i(0))_{i \in \N}$ converges to some unit vector~$v_\infty
\in T\homRcov M$.  We then consider the limiting curve
\begin{align*}
     \homcov \sigma_\infty \colon \R & \to \homRcov M \\
     t & \mapsto \exp_{\homRcov M}(t \cdot v_\infty)\,.
\end{align*}
The curve~$\homcov \sigma_\infty$ is a geodesic and the curves~$\homcov
\sigma_i$ converge on all compact intervals uniformly in the
$C^\infty$-topology to~$\homcov \sigma_\infty$
(Lemma~\ref{lemma:limitinggeodesic}, which is applicable in view of
Remark~\ref{rem:defN}). Thus, $\homcov \sigma_\infty$ minimizes
distance between any of its points.  Let $\sigma_\infty \colon \R \to
M$ be the projection of~$\homcov \sigma_\infty$ to~$M$. By
construction, $\sigma_\infty$ is a minimal geodesic on~$M$ and the
$\sigma_i$ converge on all compact intervals uniformly
to~$\sigma_\infty$.  We say that $\sigma_\infty$ is a \emph{minimal
  geodesic constructed from~$([x,y],z,\eta)$}.

By construction, this is even an $\R$\=/ho\-mo\-lo\-gi\-cal\-ly
minimal geodesic.

\subsection{Asymptotes of geodesics from exposed edges of the stable norm ball}

In order to distinguish geodesics constructed from exposed edges as in Section~\ref{subsec:sinfty}, we analyze their
asymptotes with respect to the underlying exposed edge of the
stable norm unit ball. We recall that
$\Apm(\sigma_\infty)= A^+(\sigma_\infty)\cup A^-(\sigma_\infty)$
(Definition~\ref{def.asymptotes}).

\begin{proposition}\label{prop:limits.in.face}
  We assume the situation of Section~\ref{subsec:sinfty}. In
  particular, let $\sigma_\infty$ be a minimal geodesic constructed from~$([x,y],z,\eta)$. Then we have
  \[ \Apm(\sigma_\infty) \subset [x,y]
  \,.
  \]
\end{proposition}
\begin{proof}
  The cases of initial and terminal asymptotes are symmetric; we
  therefore only consider terminal asymptotes.  Let $v \in
  A^+(\sigma_\infty)$. Again choose $\omega\in \HdR$ with $[x,y] =
  F_\omega$. Because $A^+(\sigma_\infty)$ lies in the stable norm unit
  ball (Lemma~\ref{lem:haproperties}~\ititemref{lem:haproperties.ii}),
  it suffices to show that~$\langle \omega, v \rangle = 1$.

  On the one hand, by definition of $\stabnormdual\argu$ we have
  $\langle \omega, v \rangle \leq \stabnormdual \omega \cdot \stabnorm
  v \leq 1$.  For the converse estimate, we argue as follows: Let
  $(t_i)_{i \in \N_0} \subset \R$ be a sequence of parameters with~$t_i
  \to \infty$ realizing~$v$ in Equation~\eqref{eq.v.def} of Definition~\ref{def.asymptotes};
  without loss of generality, we may assume that~$t_0 = 0$
  (Remark~\ref{rem:hastart}).  Passing to a subsequence
  of~$(\sigma_i)_{i \in \N}$ without changing the sequence $(t_i)_{i \in \N}$, we may assume that $t_i \in [0,\ell_i]$
  for all~$i \in \N$ and
  \begin{align*}
    \stabnorm[big]{\hsi{t_i}- \hsi[\infty]{t_i}}
  \leq \frac1i \cdot t_i
  \,.
  \end{align*}
  In particular, this implies that 
  \[ \langle \omega, v \rangle= \lim_{i \to \infty}  \biggl\langle \omega,
           \frac{\hsi[\infty]{t_i}}{t_i}  \biggr\rangle
  = \lim_{i \to \infty}
  \biggl\langle \omega,
  \frac{\hsi{t_i}}{t_i}
  \biggr\rangle
  \]
  and we can thus exploit the properties of the geodesics~$\sigma_i$.
  For all~$i \in \N$, we obtain
  \begin{align*}
    \biggl\langle
    \omega, \frac{\hsi{t_i}}{t_i}
    \biggr\rangle
    & = \frac1{t_i} \cdot
    \bigl(
    \bigl\langle \omega, \hsi{\ell_i})\bigr\rangle
    -
    \bigl\langle \omega, \hsirel[t_i]{\ell_i}\bigr\rangle
    \bigr)
    \\
    & \geq \frac1{t_i} \cdot
    \bigl(\langle \omega, z_i \rangle
    - \underbrace{\stabnormdual{\omega}}_{=1} \cdot \stabnorm[big]{\hsirel[t_i]{\ell_i}}\bigr)\,.
  \end{align*}
Proposition~\ref{prop:stabnormlength} provides $-\stabnorm[big]{\hsirel[t_i]{\ell_i}}\geq - \maL(\sigma)-c$ for $c\ceq
  (\maJ +1) \diam M$ and~$\maJ$ given by Equation~\eqref{eq.def.maJ}. 
  Using this and Lemma~\ref{lem:constructionzi}, we continue our
  estimate with
  \begin{align*}
    \biggl\langle
    \omega, \frac{\hsi{t_i}}{t_i}
    \biggr\rangle 
    & \geq \frac1{t_i} \cdot
    \bigl(\langle \omega,z_i\rangle - (\ell_i - t_i) - c \bigr)
    \\
    & \geq \frac1{t_i} \cdot
    \bigl((\ell_i - a) - (\ell_i - t_i) - c)\bigr)
    \\
    & = \frac{t_i - a - c}{t_i}\,.
  \end{align*}
  We conclude that
  \[ \langle \omega, v \rangle
  \geq \lim_{i \to \infty} \frac{t_i - a - c}{t_i} = 1
  \,,
  \]
  as desired.
\end{proof}

More precisely, we know in which parts of the exposed edge
the terminal/initial asymptotes are located with respect
to the point~$z$:

\begin{proposition}\label{prop.asymptote.semiedge}
  We assume the situation of Section~\ref{subsec:sinfty}. In
  particular, let $\sigma_\infty$ be a minimal geodesic constructed from~$([x,y],z,\eta)$.  
  Then we have
  \[ A^+(\sigma_\infty) \subset [z,y]
  \qand A^-(\sigma_\infty) \subset [x,z]
  \,.
  \]
\end{proposition}
\begin{proof}
  We prove the claim for~$A^+(\sigma_\infty)$, the other case
  being symmetric. We know that $A^+(\sigma_\infty) \subset 
  [x,y]$ by Proposition~\ref{prop:limits.in.face}.
  Therefore, it suffices to show that $\<
  [\eta], v \> \geq 0$ for all~$v \in A^+(\sigma_\infty)$.

  We consider the function
  \begin{align*}
  u_\infty \colon \R & \to \R,\quad
  t \mapsto \int_{\sigma_\infty\steLLe{[0,t]}} \eta\,.
  \end{align*}
  Because $\eta$ is a linear combination of~$\alpha^1, \dots, \alpha^b$, 
  Remark~\ref{rem:hurewicznonint} tells us that 
  \[u_\infty(t)=\<[\eta],\hsi[\infty]{t}\>\,.\]  
  Below, we will show
  that~$u_\infty \geq 0$. Assuming this estimate, we can argue as
  follows: Let $v \in A^+(\sigma_\infty)$ and let $(t_i)_{i \in \N}
  \subset \R_{>0}$ be a sequence for which we obtain $v$ as a limit
  as in Definition~\ref{def.asymptotes}. 
  Then
  \begin{align*}
  \bigl\< [\eta], v \bigr\>
  &
  =\, \lim_{i \to \infty}
    \frac{\bigl\< [\eta], \hsi[\infty]{t_i}\bigr\>}%
         {t_i}\,=\, \lim_{i \to \infty}
    \frac{u_\infty(t_i)}%
         {t_i}
  \geq 0\,,
  \end{align*}
  as desired.
  
  It remains to show that~$u_\infty \geq 0$: Let $t \in \R$ and let $I
  \subset \R$ be a compact interval with~$0,t \in \inneres{I}$.  By
  construction, $\sigma_i\steLLe{I} \to \sigma_\infty\steLLe{I}$
  in~$C^\infty(I,M)$ for~$i \to \infty$.  We assume that an $\ell\in
  \{1,\ldots,b\}$ and a sequence ~$(\lambda_{i})_{i \in \N}$ is chosen
  as in Lemma~\ref{lem:constructionetai}; this provides
  $\eta_i$ and then $u_i$ is defined in Equation~\eqref{eq:def.ui}.
  Furthermore, we have
  \[ \lim_{i \to \infty} \int_{\sigma_i\steLLe{[0,t]}} \alpha^\ell
  = \int_{\sigma_\infty|_{[0,t]}} \alpha^\ell
  \]
  and thus $\bigl(\int_{\sigma_i\steLLe{[0,t]}}
  \alpha^\ell\bigr)_{i\in\N}$ is bounded.  Therefore,
  \begin{align*}
    u_\infty (t)
    & = \int_{\sigma_\infty\steLLe{[0,t]}} \eta
    \,= \lim_{i \to \infty}
    \int_{\sigma_i\steLLe{[0,t]}} ( \eta_i - \lambda_i \cdot \alpha^\ell)\\ 
    &= \lim_{i \to \infty}
      \underbrace{\int_{\sigma_i\steLLe{[0,t]}} \eta_i}_{=u_i(t)}   - \lim_{i \to \infty}\biggl(\underbrace{\lambda_i}_{\to 0} \cdot \underbrace{\int_{\sigma_i\steLLe{[0,t]}}  \alpha^\ell}_{\textrm{bounded}}\;\biggr) \\
    &=  \lim_{i \to \infty} u_i(t) \geq 0\,.\qedhere
  \end{align*}
\end{proof}

\begin{proposition}[homologically homoclinic minimal geodesics]\label{prop:sinftyuniqueasymptote}
  In the situation of Section~\ref{subsec:sinfty},
  let us assume that $\sigma_\infty$ has a unique terminal
  and a unique initial asymptote and that these coincide.
  Then this asymptote is the given edge element~$z$.
\end{proposition}
\begin{proof}
  Because of $[x,z] \cap [z,y] = \{z\}$, this is immediate
  from Proposition~\ref{prop.asymptote.semiedge}. 
\end{proof}

In general, we cannot expect to be in this
situation. We always have the following, which
will help us to distinguish minimal geodesics for
different exposed edges:

\begin{proposition}\label{prop:distinguishedges}
  In the situation of Section~\ref{subsec:sinfty},
  the geodesic~$\sigma_\infty$ constructed above satisfies
  \[ \Atot(\sigma_\infty) 
  \cap (x,y) \neq \emptyset
  \,.
  \]
\end{proposition}
\begin{proof}
  We know that
  $\Apm(\sigma_\infty) \subset [x,y]$
  (Proposition~\ref{prop:limits.in.face}) and that $\Apm(\sigma_\infty) \neq \emptyset$
  (Lemma~\ref{lem:haproperties}~\itemref{lem:haproperties.iii}). 
  
  If $\Apm(\sigma_\infty)$ contains more than one point, then we use
  the connectedness
  of~$\Atot(\sigma_\infty)$, which contains~$\Apm(\sigma_\infty)$
  (Lemma~\ref{lem:haproperties}~\itemref{lem:haproperties.vii}). Then
  $\Atot(\sigma_\infty) \cap (x,y)$ is non-empty because $[x,y]$ is
  an exposed edge.

  Therefore, it suffices to consider the case that
  $\Apm(\sigma_\infty)$ consists of a single point of the stable norm
  unit ball. In this case, we obtain $\Apm(\sigma_\infty) = \{z\}$
  from Proposition~\ref{prop:sinftyuniqueasymptote}; moreover, $z$
  lies in~$(x,y)$.
\end{proof}

We obtain two geometrically distinct geodesics constructed from a
single exposed edge:

\begin{corollary}\label{cor:twomingeod}
  Let $M$ be a closed connected Riemannian manifold and 
  let $[x,y]$ be an exposed edge of the stable norm unit ball of~$M$.
  Then there exist at least two geometrically distinct $\R$-homologically
  minimal geodesics~$\gamma \colon \R \to M$ of~$M$ with
  \begin{equation}\label{display.int.char}
    \begin{split}
      &A^+(\gamma)\text{ and }A^-(\gamma)\text{ are closed connected subsets of }[x,y]\\
      &\text{and }   \Atot(\gamma) \cap (x,y) \neq \emptyset\,.
    \end{split}
  \end{equation}
\end{corollary}
\begin{proof}
  Let $z \in (x,y)$. We choose~$\eta \in \Omega^1(M)$ as in
  Section~\ref{subsec:sinfty}.  Let $\sigma_\infty$ be a minimal
  geodesic constructed from~$([x,y],z,\eta)$
  (Section~\ref{subsec:sinfty}).  Flipping the sign, let $x':=y$,
  $y':=x$, $z' := z$, and $\eta' := -\eta$.  Then the construction
  from Section~\ref{subsec:sinfty} also provides us with a minimal
  geodesic~$\sigma'_\infty$ constructed from~$([y,x],z,-\eta)$.

  We obtain (Proposition~\ref{prop.asymptote.semiedge}
  and Proposition~\ref{prop:distinguishedges})
  \begin{align*}
  & A^+(\sigma_\infty) \subset [z,y],
  \quad A^-(\sigma_\infty) \subset [x,z],
  \quad \Atot(\sigma_\infty) \cap (x,y) \neq \emptyset,\\
  & A^+(\sigma'_\infty) \subset [x,z],
  \quad A^-(\sigma'_\infty) \subset [z,y],
  \quad \Atot(\sigma'_\infty) \cap (x,y) \neq \emptyset\,.
  \end{align*}

  As $[x,y]\subset \partial B$ is an exposed edge, we have $[x,y] \cap
  [-x,-y]=\emptyset$.  In particular,
  Remark~\ref{rem:distinctasymptotes} yields that $\sigma_\infty$ and
  $\sigma'_\infty$ are geometrically distinct unless
  \begin{equation}\label{eq:unless}
    A^+(\sigma_\infty) = \{z\} = A^-(\sigma_\infty)\,.
  \end{equation}
  If this holds (\ie if $\sigma_\infty$ is homologically homoclinic),
  then we may pick~$z'' \in (z,y)$ and an appropriate
  $1$-form~$\eta''$ and construct a minimal geodesic~$\sigma''_\infty$
  constructed from~$([x,y], z'', \eta''\>)$.  
  This minimal geodesic satisfies (Proposition~\ref{prop.asymptote.semiedge} and
  Proposition~\ref{prop:distinguishedges})
  \[ A^+(\sigma''_\infty)\subset [z'',y]
  \qand
  \Atot(\sigma''_\infty) \cap (x,y) \neq \emptyset
  \,.
  \] 
  This implies $z \in A^+(\sigma_\infty) \setminus A^+(\sigma''_\infty)$, 
  and because of $[x,y] \cap [-x,-y]=\emptyset$ we conclude that $\sigma_\infty$
  and $\sigma''_\infty$ are geometrically distinct. 

  By construction $\sigma_\infty$,  $\sigma'_\infty$, and  $\sigma''_\infty$ are $\R$-homologically minimal.
  The sets~$A^\pm(\sigma_\infty)$, $A^\pm(\sigma'_\infty)$, and $A^\pm(\sigma''_\infty)$ are closed and connected by
  Lemma~\ref{lem:haproperties}~\itemref{lem:haproperties.iv} and~\itemref{lem:haproperties.v}.
\end{proof}

A stronger version of Corollary~\ref{cor:twomingeod}
is provided by the following:

\begin{proposition}\label{prop:twomingeod.stronger}
  Let $M$ be a closed connected Riemannian manifold and 
  let $[x,y]$ be an exposed edge of the stable norm unit ball of~$M$;
  and let $x$ and $y$ be exposed points of the stable norm unit ball.
  Then at least one of the following statements holds:
  \begin{enumAlpha}
  \item\label{cor:twomingeod.stat.i} There are uncountably many
    geometrically distinct homologically homoclinic minimal
    geodesics of~$M$
    satisfying Equation~\eqref{display.int.char}.
  \item\label{cor:twomingeod.stat.ii} There are infinitely many
    geometrically distinct homologically non-homoclinic minimal
    geodesics of~$M$ satisfying Equation~\eqref{display.int.char}.
  \item\label{cor:twomingeod.stat.iii}
    There are geometrically
    distinct homologically non-homoclinic minimal
    geodesics $\gamma_1,\dots,\gamma_m \colon \R \to M$ satisfying
    Equation~\eqref{display.int.char}
    with at least one of the following additional properties:   
    \begin{enumC}
    \item\label{cor:twomingeod.stat.iii.a} $m=2$, and $\gamma_1$ and
      $\gamma_2$ are homologically heteroclinic and homologically
      exposed.
    \item\label{cor:twomingeod.stat.iii.b} $m =3$, and $\gamma_1$ is
      homologically heteroclinic and homologically exposed. The
      geodesics $\gamma_2$ and~$\gamma_3$ are homologically
      semi\-/ex\-posed, but not homologically exposed.
    \item\label{cor:twomingeod.stat.iii.c} $m=4$, and
      $\gamma_1,\ldots,\gamma_4$ are homologically semi\-/ex\-posed,
      but not homologically exposed.
   \end{enumC}
  \end{enumAlpha}
  All minimal geodesics detected in the above cases are
  $\R$\=/ho\-mo\-lo\-gi\-cal\-ly minimal. Furthermore each of these
  geodesics is homologically non-exposed or it is homologically
  heteroclinic.
\end{proposition}

\begin{remark}
 The following proof remains valid if we remove the assumption that the boundary points $x$ and $y$ are exposed. Note that the exposedness of $[x,y]$ already implies that $x$ and $y$ are extremal, which is a bit weaker than being exposed, see Subsection~\ref{subsec:def.convex.bodies}. If $x$ is no longer exposed, then $A^\pm(\gamma)=\{x\}$ no longer means that~$\gamma$ is ``homologically semi\-/ex\-posed'', see Definition~\ref{def.homolo.homocli}, and similarly for~$y$. Nevertheless Proposition~\ref{prop:twomingeod.stronger} still holds
if we replace the properties 
    ``homologically exposed'',   ``homologically semi\-/ex\-posed'' in Proposition~\ref{prop:twomingeod.stronger} \ref{cor:twomingeod.stat.iii}  by  ``homologically extremal'',   ``homologically semi\-/ex\-tremal'' which are defined similarly by replacing ``exposed points'' by ``extremal points'' in Definition~\ref{def.homolo.homocli}. We will not elaborate on this in detail, as the main focus of the article lies on the case that the stable norm unit ball is a polytope in which case all faces of this ball are exposed, see the comment at the end of Subsection~\ref{subsec:def.convex.bodies}.
\end{remark}

This proposition and its proof show, in particular: If the
stable norm unit ball contains an exposed edge~$[x,y]$, then  
\begin{itemize}
\item Condition~\ref{cor:twomingeod.stat.i} holds or 
\item the two geodesics provided by Corollary~\ref{cor:twomingeod} may be chosen to be homologically non-homoclinic.
\end{itemize}
For both (non-exclusive) alternatives, there are examples, \eg  Example~\ref{example:Hedlung-cont} for the
first alternative,
and for the second alternative Examples~\ref{exam.hedlund.b2}.

\begin{proof}[Proof of Proposition~\ref{prop:twomingeod.stronger}]
  For simplicity of notation we introduce a linear ordering~$\leq$,
  resp.~$<$, on~$[x,y]$, characterized by ``$z\leq w \gdw w\in
  [z,y]$''. All minimal geodesics in this proof will be
  $\R$\=/ho\-mo\-lo\-gi\-cal\-ly minimal.
  
  The proof of this proposition starts as the proof
  of Corollary~\ref{cor:twomingeod}. For every given~$z\in (x,y)$ we
  obtain $\R$\=/ho\-mo\-lo\-gi\-cal\-ly minimal geodesics $\gamma_1^z\ceq
  \sigma_\infty$ and~$\gamma_2^z\ceq \sigma'_\infty$ with
  \begin{align*}
  & A^+(\gamma_1^z) \subset [z,y],
  \quad A^-(\gamma_1^z) \subset [x,z],
  \quad \Atot(\gamma_1^z) \cap (x,y) \neq \emptyset,\\
  & A^+(\gamma_2^z) \subset [x,z],
  \quad A^-(\gamma_2^z) \subset [z,y],
  \quad \Atot(\gamma_2^z) \cap (x,y) \neq \emptyset\,.
  \end{align*}
  As all~$A^\pm(\gamma_i^z)$ are closed and connected subsets
  of~$[x,y]$, there are~$v_i^\pm(z)\in[x,y]$ and~$w_i^\pm(z)\in[x,y]$
  with~$A^\pm(\gamma_i^z)=[v_i^\pm(z),w_i^\pm(z)]$, and we get
  \begin{align*}
    & x\leq v_1^-(z)\leq w_1^-(z)\leq z \leq v_1^+(z)\leq w_1^+(z)\leq y\,,\\
    & x\leq v_2^+(z)\leq w_2^+(z)\leq z \leq v_2^-(z)\leq w_2^-(z)\leq y\,.
  \end{align*}
  We study the geodesics~$\gamma_1^z$ with~$z\in (x,y)$ first, which
  are not necessarily geometrically distinct for different values
  of~$z$. We claim that in this family~$(\gamma_1^z)_{z\in (x,y)}$ of geodesics 
  \begin{enumalpha}
  \item\label{claim.i} there is at least one homologically heteroclinic and
    homologically exposed one,
  \item\label{claim.ii} or there are at least two geometrically distinct homologically
    semi\-/ex\-posed ones that are not homologically exposed and that
    are homologically non-homoclinic,
  \item\label{claim.iii} or this family provides an infinite number of
    geometrically distinct homologically non-homoclinic minimal geodesics,
  \item\label{claim.iv}
    or there are uncountably many geometrically distinct homologically homoclinic minimal geodesics.
  \end{enumalpha}
  If we carry out the same discussion for~$\gamma_2^z$, the
  proposition follows.
  
  We consider several cases:

  \begingroup 
  \case{1}{For all~$z \in (x,y)$, we have~$v_1^+(z)<y$.}
  In this case, we
  choose a~$z_1\in (x,y)$ and then we inductively select 
  \[z_{i+1}\in \bigl(v_1^+(z_i),y\bigr)
  \,.
  \]  
  By construction, the geodesics~$\gamma_1^{z_i}$ with~$i\in \N$ are
  geometrically distinct minimal geodesics.  We have thus obtained an
  infinite number of minimal geodesics.  If infinitely many of them
  are homologically non-homoclinic, then we have
  Item~\itemref{claim.iii} of the claim. 
  If only finitely
  many of them are geometrically non-homoclinic with
  Equation~\eqref{display.int.char}, then the union of their $\Atot$-sets is a
  compact subset of~$[x,y)$ and thus its complement contains a set
  of the form~$[w,y)$ for some~$w<y$. 
  Then for every~$z\in [w,y)$,
  the geodesic~$\gamma_1^z$ is a homologically homoclinic minimal geodesic,
  and they are pairwise geometrically distinct.
  We thus get Item~\itemref{claim.iv} of the claim.
  
  \case{2}{For all~$z\in (x,y)$, we have~$x<w_1^-(z)$.}
  In this case, we argue analogously to Case~1.
 
  \case{3}{There exists a~$z\in (x,y)$ with $x=w_1^-(z)$
    and~$v_1^+(z)=y$.}
  We choose such a~$z$. Then
  $A^+(\gamma_1^z)=\{y\}$ and $A^-(\gamma_1^{z})=\{x\}$.  The
  geodesic~$\gamma_1^z$ is homologically heteroclinic and
  homologically exposed, and Item~\itemref{claim.i} of the claim
  holds.
  \case{4}{Remaining case.}
  In this case, we have
  \begin{align*}
    \exi{z\in (x,y)} & \bigl( x<w_1^-(z)\text{ and }v_1^+(z)=y\bigr) \qand\\
    \exi{z'\in (x,y)} & \bigl(x=w_1^-(z')\text{ and } v_1^+(z')<y\bigr)\,.
  \end{align*}
  The minimal geodesics $\gamma_1^z$ and~$\gamma_1^{z'}$ are
  homologically semi\-/ex\-posed and geometrically distinct. They are
  not homologically exposed. Thus, we have Item~\itemref{claim.ii} of
  the claim.
  \endgroup 
\end{proof}

\begin{example}\label{example:Hedlung-cont}
  Let $M=T^2\times S^2$, where $T^2=\R^2/\Z^2$. 
  We choose the $1$-forms $\alpha^1$ and $\alpha^2$ in Setup~\ref{setup:bases} as pullbacks of $\upd x^1$
  and $\upd x^2$ on $\R^2/\Z^2$. The projection $p_1\colon M\to T^2$ gives an isomorphism 
  $H_1(M;\R)\to H_1(T^2;\R)$ and we identify $H_1(M;\R)\cong H_1(T^2;\R)\cong \R^2 \cong \ucov T^2$.
  The Jacobi maps $J^M$ and $J^{T^2}$ of $M$ and $T^2$ satisfy $J^M = J^{T^2}\circ \ucov p_1$, where
  $\ucov p_1\colon \R^2\times S^2\to \R^2$ is again projection to the first factor.
  
  Let $g_s$ be a Riemannian metric with constant coefficients on~$\R^2$, 
  depending smoothly on~$s\in S^2$; we also write~$g_s$ for the associated metric on~$T^2$. 
  Denoting the standard metrik on~$S^2$ by~$g^{S}$ we define the 
  Riemannian metric $g\ceq g_s+g^{S}$ on~$M$. One may choose a smooth family $f_s\in \Iso(\R^2)$,
  such that $f_s^*g_s$ is the standard metric on $\R^2$. The closed unit norm ball of~$g_s$ 
  is thus $B_s\ceq f_s(B_1(0))$, where $B_1(0)$ is the standard closed unit ball in $\R^2$.

  Then $B=\bigcup_{s\in S^2} B_s$ is the union of a smooth family of ellipses centered in~$0$; however, in general, $B$ is not convex.
  But for suitable choices of~$g_s$, one can achieve that $B$ is convex and such that there is an 
  embedded smooth path $\sigma\colon [0,1]\to S^2$ with  $B=\bigcup_{t\in [0,1]} B_{\sigma(t)}$. 
  These conditions imply
  \[\stabnorm{v}= \min_{t\in[0,1]}\sqrt{g_{\sigma(t)}(v,v)}\]
  and that $B$ is the stable norm unit ball. By choosing the family~$(g_s)_s$ in a suitable way,
  we can additionally achieve that~$\partial B$ contains an exposed edge, \ie a non-trivial line segment $e\subset \partial B$.
  Let $\inneres{e}\neq\emptyset$ be the open line segment associated to $e$.

  Under these assumptions we have on $(M,g)$:
  \begin{itemize}  
  \item all minimial geodesics $\gamma$ with $\Atot(\gamma)\subset e$ are homologically homoclinic,
  \item all points $v$ in $\inneres{e}$  have a minimal geodesic~$\gamma_v$ with $A^+(\gamma_v)= \{v\} = A^-(\gamma_v)$ 
  (we do not claim uniqueness of~$\gamma_v$).
 \end{itemize}
Obviously all minimal geodesics are $\R$-homologically minimal. Thus we have verified Condition~\ref{cor:twomingeod.stat.i}. 
\end{example}

\section{Lower bounds for centrally symmetric polytopes}\label{sec:cs}

We give a (crude) lower bound for the number of vertices and
edges of centrally symmetric polytopes, which will be
used to prove Theorem~\ref{thm:main} in Section~\ref{sec:proofs}.

\begin{proposition}\label{prop:cslowerbound}
  Let $n \in \N$ and let $P \subset \R^n$ be a centrally symmetric
  polytope with non-empty interior. Let $V$ and $E$ denote the
  number of vertices and edges of~$P$, respectively. Then,
  $V \geq 2 \cdot n$ and 
  \[ \frac12 \cdot V + E
     \geq
     \min(n^2 + 2\cdot n +1, 2 \cdot n^2 - n)
     = \begin{cases}
     2 \cdot n^2 - n& \text{if $n \leq 3$}
     \\
     n^2 + 2\cdot n +1 & \text{if $n \geq 4$}
     .
     \end{cases}  
     \]
\end{proposition}
\begin{proof}
  We distinguish two cases:
  \begin{enumarab}
  \item The polytope~$P$ is simplicial, i.e., all faces are simplices.
    Because $P$ is centrally symmetric, the number of vertices and
    edges of the $n$-dimensional cross-polytope is a lower bound
    for~$V$ and $E$, respectively~\cite[Section~7]{novik}. Hence,
    \[ V \geq 2 \cdot n
    \qand
       E \geq 2^{1+1} \cdot {n \choose {1+1}} = 2 \cdot n \cdot (n-1)\,.
    \]
    Therefore, $1/2 \cdot V+E \geq 2 \cdot n^2 - n$.
  \item
    The polytope~$P$ is \emph{not} simplicial, i.e., $P$ has at least
    one facet~$\sigma$ (i.e., at least one face of dimension~$n-1$ in the boundary)
    that is not an $(n-1)$-simplex. 
    In particular, $\sigma$ has at least $n+1$ vertices.
    Because $P$ is centrally symmetric, also $-\sigma$ is a facet
    of~$P$; as $P$ has non-empty interior, these parallel faces~$\sigma$
    and~$-\sigma$ have no common vertex. Thus, $P$ has at
    least $V \geq 2 \cdot (n+1)$ vertices. 

    Because of~$\inneres{P} \neq\emptyset$, each vertex of~$P$ 
    has edge-degree at least~$n$. Therefore,~$P$ has at least
    \[ \frac12 \cdot n \cdot V
       \geq \frac12 \cdot n \cdot 2 \cdot (n+1)
       = n \cdot (n+1)
    \]
    edges.
    Hence, $1/2 \cdot V+E \geq (n+1) \cdot (n+1) = n^2 + 2 \cdot n +1$.
  \end{enumarab}	
  Taking the minimum over both cases gives the claim.
\end{proof}

We can reformulate this estimate in terms of the constants
from Definition~\ref{def:Emin}:

\begin{corollary}\label{cor:VEmin}
  For  all~$b \in \N$, we have  
  \[ \VEmin(b) \geq
     \begin{cases}
     2 \cdot b^2 - b& \text{if $b \leq 3$}
     \\
     b^2 + 2\cdot b +1 & \text{if $b \geq 4$}
     .
     \end{cases}  
  \]
Furthermore, $\VEmin(b) \geq b + \Emin(b)$; 
if $\VEmin(b) = b + \Emin(b)$, then $\Emin(b)$ is attained by the cross-polytope,
and we thus have $\Emin(b)=2\cdot b\cdot (b-1)$.
\end{corollary}
\begin{proof}
  The first lower bound on $\VEmin(b)$ is a direct consequence of Proposition~\ref{prop:cslowerbound}.  Let $P$ be a polytope that attains~$\VEmin(b)$. 
  By elementary geometry, $P$ has at least $2b$~vertices, which implies
  that $\VEmin(b) \geq b + \Emin(b)$.
  If this inequality is an equality, then $P$ has precisely $2b$~vertices, 
  and is then (up to an affine transformation) a cross-polytope.
\end{proof}

In view of the square and the octahedron, these bounds are sharp in
dimensions $2$ and~$3$.  Improved bounds for~$\Emin$ in higher
dimensions are discussed in Remark~\ref{rem:csimproved}~\ref{rem:csimproved.i}.

\section{Proofs of the results from the introduction}\label{sec:proofs}

We combine the results from the previous sections to
complete the proofs of the main results. We will always
use the following notation:

\begin{setup}
  Let $M$ be a closed connected (and non-empty) Riemannian manifold.
  Let $b := \dim_\R H_1(M;\R)$ and let $B \subset H_1(M;\R)$ be the
  stable norm unit ball.
\end{setup}

\begin{remark}\label{rem:exposedpolytope}
  We know that $B$ is closed, convex, centrally symmetric,
  and has non-empty interior.
  If $B$ has only finitely many exposed points, then $B$
  has only finitely many extreme points (by Straszewicz's
  theorem~\cite{straszewicz}) and so $B$ is a compact
  centrally symmetric polytope in~$H_1(M;\R)$ with non-empty
  interior. 
\end{remark}

\subsection{Proof of Theorem~\ref{thm:main}}
\label{subsec:proofthmmain}

Let $M$ be a closed connected Riemannian manifold. The stable norm
ball~$B \subset H_1(M;\R)$ is centrally symmetric and convex.

\begin{enumarab}
  \item\label{proof.thmmain.1:1} If $v$ is an exposed point of~$B$,
    then Bangert shows that there exists an
    $\R$\=/ho\-mo\-lo\-gi\-cal\-ly minimal geodesic~$\gamma$
    with~$A^+(\gamma) = \{v\} =
    A^-(\gamma)$~\cite[Theorem~4.4]{bangert:90}. This proves the
    theorem in the case of~$b = 1$ or in the case that $B$ has
    infinitely many exposed points.
\item\label{proof.thmmain.1:2} We may therefore assume that $b \geq 2$
  and that $B$ has only finitely many exposed points. Then $B$ is a
  compact centrally symmetric polytope with non-empty interior
  (Remark~\ref{rem:exposedpolytope}); in particular, all
  vertices/edges of~$B$ are exposed. Let $V(B)$ denote the number of
  vertices of~$B$ and $E(B)$ the number of edges of~$B$.

Applying Bangert's construction to each antipodal pair of vertices
of~$B$ provides $1/2 \cdot V(B)$ geometrically distinct minimal
geodesics (which are homologically homoclinic and homologically
exposed).

For each exposed edge~$[x,y]$ of~$B$, we obtain two geometrically
distinct minimal geodesics~$\gamma$ with~$\Atot(\gamma) \cap (x,y)
\neq \emptyset$ (Corollary~\ref{cor:twomingeod}).  In particular, such
minimal geodesics for different antipodal edge pairs are geometrically
distinct. Thus we have at least $E(B)$ many geometrically distinct
minimal geodesics of this type.  Moreover, these minimal geodesics for
exposed edges do not have the vertex asymptote behaviour of Bangert's vertex
examples: they are homologically heteroclinic or they are not
homologically exposed.  
Thus the geodesics constructed from the edges are all geometrically
distinct from the ones obtained from the vertices.

The total number of $\R$\=/ho\-mo\-lo\-gi\-cal\-ly minimal geodesics
is therefore at least~$1/2 \cdot V(B) + E(B)$.

\item\label{proof.thmmain.1:3}
  Theorem~~\ref{thm:main} now follows from Corollary~\ref{cor:VEmin}.
\end{enumarab}

Moreover, we record the following observation from
the proof above:

\begin{remark}\label{rem:finpolytope}
  If $M$ has only finitely many geometrically distinct
  minimal geodesics, then $B$ is a compact convex polytope
  in~$H_1(M;\R)$ with non-empty interior and $M$
  has at least~$V(B)/2 + E(B)$ geometrically distinct
  minimal geodesics.
\end{remark}

\subsection{Proof of Theorem~\ref{thm:main.2}}
\label{subsec:proofthmmain.2}

If $B$ is not a polytope, then $B$ has infinitely
many different antipodal pairs of exposed points
(Remark~\ref{rem:exposedpolytope}) and so the
corresponding minimal geodesics by Bangert show
that we are in case~\ititemref{thm:main.alt.1} of Theorem~\ref{thm:main.2}.

If $B$ is a polytope, then the arguments in
Section~\ref{subsec:proofthmmain} can be refined. For every edge we apply
Proposition~\ref{prop:twomingeod.stronger}:
\begin{itemize}
  \item If we are in case~\ititemref{cor:twomingeod.stat.i} of
    Proposition~\ref{prop:twomingeod.stronger} for some edge, then we
    have uncountably many geometrically distinct homologically
    homoclinic minimal geodesics. In particular, we are again in
    case~\ititemref{thm:main.alt.1} of Theorem~\ref{thm:main.2}.
  \item If we are in case~\ititemref{cor:twomingeod.stat.ii} of
    Proposition~\ref{prop:twomingeod.stronger} for some edge, then we
    have infinitely many geometrically distinct homologically
    non-homoclinic minimal geodesics. These geodesics are
    geometrically distinct from at least $b$ homologically homoclinic
    minimal geodesics, detected with Bangert's method. This shows that
    we are in case~\ititemref{thm:main.alt.2} of
    Theorem~\ref{thm:main.2}.
  \item If we are in case~\ititemref{cor:twomingeod.stat.iii} of
    Proposition~\ref{prop:twomingeod.stronger} for every edge, then
    every pair of antipodal edges contributes at least two
    homologically non-homoclinic minimal geodesics. The statament for
    the geodesics via Bangert's method is as above. Again, we are in
    case~\ititemref{thm:main.alt.2} of Theorem~\ref{thm:main.2}.
  \end{itemize}
As before, all minimal geodesics in this list are
$\R$\=/ho\-mo\-lo\-gi\-cal\-ly minimal.

The last claim of Theorem~\ref{thm:main.2} follows from
Remark~\ref{rem:finpolytope} and the proof is thus complete.

\subsection{Proof of Theorem~\ref{thm.equality.gen}}
\label{subsec:proofEminequal}

Let $M$ have exactly~$b + \Emin(b)$ minimal geodesics. In particular,
in combination with Remark~\ref{rem:finpolytope}, we obtain that the
stable norm unit ball~$B$ is a compact convex centrally symmetric
polytope in~$H_1(M;\R)$ with non-empty interior and $E(B) = \Emin(b)$
as well as~$V(B) = 2\cdot b$. As vertices come in antipodal pairs, there is a
$b$-tuple $\maV=(v_1,\ldots,v_b)$ of vectors in~$H_1(M;\R)$ such that
$v_1,\ldots,v_b,-v_1,\ldots,-v_b$ are precisely the vertices of $B$.
Because $B$ is the convex hull of its vertices and because $B$ has non-empty
interior, $\maV$ is a linearly independent family, and thus a basis
of~$H_1(M;\R)$. It follows that $B$ has the combinatorial type of
a cross-polytope.

In particular, $B$ is a simplicial polytope and as in the proof of
Proposition~\ref{prop:cslowerbound} we have~$\Emin(b) = E(B)$. Therefore,
\[ \Emin(b) = E(B) = 2 \cdot b^2 -2 \cdot b
\,.
\]
Further, every antipodal pair of edges contributes at least two
$\R$\=/ho\-mo\-lo\-gi\-cal\-ly minimal geodesics
(Proposition~\ref{prop:twomingeod.stronger}) and every antipodal pair
of vertices contributes one minimal geodesic. Thus, the equality
discussion implies that each edge pair contributes precisely two, and
each vertex pair precisely one. Hence, for the edges we will always be
in case~\ititemref{cor:twomingeod.stat.iii.a} of this proposition. It
follows that we have $\Emin(b)$ of $\R$\=/ho\-mo\-lo\-gi\-cal\-ly minimal
geodesics that are homologically heteroclinic and homologically exposed,
and $b$ of $\R$\=/ho\-mo\-lo\-gi\-cal\-ly minimal geodescis that are
homologically homoclinic and homologically exposed, and that there are
no other minimal geodesics (up to geometric equivalence).

\subsection{Proof of Theorem~\ref{thm.equality.polytope-fixed}}
\label{subsec:proofVEminequal}

Let $M$ have exactly~$V(B)/2 + E(B)$ geometrically distinct minimal
geodesics. Then a similar equality discussion implies that precisely
$E(B)$ minimal geodesics are determined by the edges and that
precisely $V(B)/2$ minimal geodesics are determined by the vertices.
The types of the minimal geodesics are then determined by
Proposition~\ref{prop:twomingeod.stronger}, similarly to the proof of
Theorem~\ref{thm.equality.gen} in Subsection~\ref{subsec:proofEminequal}.

\section{Examples}\label{sec:examples}

In the following section, we discuss existence of minimal geodesics on surfaces, as well as examples of Hedlund metrics on manifolds of dimension at least~$3$. 
We also discuss further examples and discuss more related literature.

\subsection{Surfaces}

We discuss minimal geodesics on closed connected surfaces~$M$.
On the one hand, the $2$-dimensional case is special as
topological conditions often imply intersections of geodesics and the
results are in some sense different from higher dimensions, see
Remark~\ref{rem:2dim.untypical}.
On the other hand, this situation
nicely illustrates possible behaviours of minimal
geodesis and their asymptotes.

By passing to connected components and to the orientation covering, it
is sufficient to consider connected and orientable
surfaces~$M$.
Obviously, minimal geodesics exist if and only if the genus of~$M$ is at least~$1$.

Let $M$ be such a surface and let $\gamma$ be a simple closed
geodesic on~$M$.
Recall that \emph{simple} means by definition that~$\gamma$ defines an embedding of a circle~$S^1$ into~$M$.
We also assume that~$\gamma$ minimizes length within its free homotopy class. 
Such minimizing simple closed geodesics~$\gamma$ always exist, in particular, in every primitive class in~$H_1(M;\Z)$~\cite{meeks.patrusky::Pacific}.
An element $\neq 0$ in a free abelian group is called \emph{primitive}
if it is not a non-trivial multiple of another element. We declare~$0$ to be \emph{non-primitive}.

By a classical argument, carried out for the torus already by Hedlund
\cite[before Theorem~II]{hedlund:1932}, the geodesic~$\gamma$ is
minimal\ifwithappendixtwodim; see Proposition~\ref{prop:closed-geo-on-surfaces} in Appendix~\ref{app:twodim} for the
general case\fi.

\begin{example}[$2$-tori]\label{exa.torus.asymptotes}
  We consider the $2$-dimensional torus $T^2=\R^2/\Z^2$, equipped with
  an arbitrary (smooth) Riemannian metric~$g$. This example has many
  features essentially different from the higher dimensional case.
  Minimal geodesics on $T^2$ were already discussed in many aspects in
  articles by Hedlund~\cite{hedlund:1932} and Morse~\cite{morse:1924},
  where minimal geodesics are called ``unending
  geodesics of class~A''. Some considerations even seem to go back to
  Bliss~\cite{bliss:1902} who considered the special case of tori embedded in Euclidean~$\R^3$
  by rotating a circle in the $(x>0,z\in \R)$-half plane around
  the $z$-axis. We refer to Bangert's articles~\cite{bangert_DynamicReported1988,bangert:90}
  for a presentation in modern language.
  
  We will again use the choices from Example~\ref{example:torus.basics}; 
  in particular, the Jacobi map corresponds to the identity under
  the canonical isomorphism~$\R^2\cong H_1(T^2;\R)$.
  We identify $\pi_1(T^2)\cong H_1(T^2;\Z)\cong
  \HZR[T^2]\cong \Z^2$, and elements thereof are written as pairs~$(k,\ell)$ of
  integers.  A pair is non-primitive if and only if it can
  be written as $(rk,r\ell)$, with $k,\ell,r\in \Z$, $r\geq 2$,
  and~$(k,\ell)\neq 0$.

  By an intersection argument, 
  \ifwithappendixtwodim see Lemma ~\ref{lemma:closed-non-prim-geod} 
  in Appendix~\ref{app:twodim} for details, \fi%
  every closed curve representing a
  non-primitive non-zero element has to have a self-intersection, which 
  implies that every closed geodesic of minimal length representing
  such a multiple~$(rk,r\ell)\in \HZR[T^2]$ is in fact an $r$-fold covering of
  a closed geodesic geodesic representing~$(k,\ell)$. This implies
  that $N\bigl((rk,r\ell)\bigr)=r \cdot N\bigl((k,\ell)\bigr)$ and so 
  $N\bigl((k,\ell)\bigr)=\stabnorm{(k,\ell)}$.

  If $(k',\ell')\notin \mathbb{Q}\cdot(k,\ell)$, then closed curves
  $\gamma$ and~$\gamma'$ representing the homology classes
  $(k,\ell)$ and~$(k',\ell')$, respectively, have to intersect. This implies 
  \[ N\bigl((k+k',\ell+\ell')\bigr)
  <N\bigl((k,\ell) \bigr)+ N\bigl((k',\ell') \bigr)
  \,.
  \] 
  We conclude that the stable norm unit sphere~$\partial B$ does not
  contain any straight lines, \ie the stable norm unit ball~$B$ is
  strictly convex.

  By definition, $\gamma\colon\R\to T^2$ is a minimal geodesic if a
  lift~$\tilde\gamma\colon\R\to \R^2$ is a line, \ie a globally
  minimizing geodesic, parametrized by arclength.  It follows from
  Theorem~\ref{thm.bangert.zwei}, or alternatively from Hedlund's work,
  that for every minimal geodesic~$\gamma$ there is a
  unique~$v_\gamma\in H_1(T^2;\R)\cong \R^2$ with $\stabnorm{v_\gamma}=1$,
  such that
  \[ A^+(\gamma)
  = A^-(\gamma)
  = \Atot(\gamma)
  =\bigl\{\lim_{t-s\to\infty}R\bigl(\gamma\stellle{[s,t]}\bigr)\bigr\}
  =\{v_\gamma\}\,.
  \] 
  Furthermore, $\tilde\gamma$ remains within bounded distance from the
  line $t\mapsto tv_\gamma$. Conversely, for every~$v\in H_1(T^2;\R)$
  with~$\stabnorm{v}=1$, there is a minimal geodesic~$\gamma$
  with~$v_\gamma=v$~\cite[Theorem~4.4]{bangert:90}; however, it may
  happen that $v_\gamma\=v_{\gamma'}$ for geometrically distinct
  minimal geodesics $\gamma$ and~$\gamma'$.

  We briefly indicate how \emph{rotation} vectors are related to
  actual rotations.
  
  There is a closed curve~$\tau$ that has minimal length among all
  closed curves representing non-zero elements in~$H_1(T^2;\Z)$. One
  can show that this curve is a simple closed geodesic, and by the
  arguments above, it is also a minimal geodesic. Further, $[\tau]\in
  H_1(T^2;\Z)$ is primitive; after pulling back by a diffeomorphism
  we may assume~$[\tau]=(0,1)$. If $\gamma$ is a minimal geodesic 
  with~$v_\gamma=(v_1,v_2)\notin \R\cdot (0,1)$, then $v_\gamma$ and $\tau$
  have infinitely many intersections, they all have the same
  orientation, and they are located at some~$v_\gamma(t_i)$ for~$i\in
  \Z$ with~$t_i<t_{i+1}$. The map $v_\gamma(t_i)\mapsto v_\gamma(t_{i+1})$
  defines a bijection $\phi\colon A\to A$ for some subset $A$ of the
  circle $S\ceq\image(\tau)$, and $\phi$ extends to a
  circle homeomorphism $\phi\colon S\to S^1$ that lifts canonically to
  a periodic orientation-preserving homeomorphism $\Phi\colon \R\to
  \R$ \cite{bangert_DynamicReported1988}.  Its rotation number
  \[ \rho(\Phi)\ceq \lim_{j\to\infty}\frac1j \cdot \bigl(\Phi^j(x)-x\bigr)
  \] 
  exists and does not depend on $x\in\R$~\cite[Sec.~2.2]{navas:2011}.
  One can show through appropriate choices that~$\rho(\Phi)=v_2/v_1$.
  Thus the rotation vector~$v_\gamma$ determines the rotation number~$\rho(\Phi)$ and
  conversely the rotation vector~$v_\gamma$, which is on the stable
  norm unit circle, is determined up to sign by~$\rho(\Phi)$.
\end{example}

Recently, also the stable norm on slit $2$-tori has been
studied~\cite{montealegre}.

\begin{example}\label{ex.higher-genus}
  Let $M$ be an orientable closed connected surface of genus~$k\geq
  2$ with an arbitrary Riemannian metric.
  Let~$\gamma$ be a closed
  geodesic representing a non-trivial $\pi_1$-conjugacy class, see
  Section~\ref{subsec:based-and-free-homotopies}, contained in the
  commutator subgroup.  Such geodesics exist, \eg if $\langle
  a_1,\dots,a_k, b_1, \dots, b_k \mid [a_1,b_1] \cdot \dots \cdot
  [a_k,b_k] \rangle$ is a presentation of~$\pi_1(M)$ arising from
  realizing~$M$ in the standard way as a quotient of a $4k$-gon, then
  a curve minimizing~$[a_1,b_1] \cdot \dots \cdot [a_\ell,b_\ell]$ 
  with~$\ell \in \{1,\dots, k-1\}$ in the associated free homotopy class 
  has this property. 
  It is not so straigthforward to decide whether~$\gamma$ is a
  simple closed curve, but one can easily see that there are $\pi_1$-conjugacy classes and Riemannian metrics~$g$ such that $\gamma$ is a simple closed curve.
  Let us mention here -- we will not use this fact in this paper -- that the simpleness of~$\gamma$ is independent of the choice of Riemannian metric~$g$, thus only depends on a suitably chosen $\pi_1$-conjugacy class \cite[Theorem~2.1]{freedman-hass-scott:82}.
  From the arguments at
  the beginning of this subsection, it follows that
  simple closed geodesics~$\gamma$
  are minimal geodesics. As $\gamma$ represents~$0\in
  H_1(M;\Z)$, the set
  \[\bigl\{\,h(\gamma\stellle{[s,t]})
  \bigm| s,t\in\R,\;\, s<t \,\bigr\}
  \]
  is bounded in~$H_1(M;\R)$ and thus
  $\Apm(\gamma)=\Atot(\gamma)=\{0\}$. In particular, both sides of
  Equation~\eqref{eq.relation.asymptote.rotation} vanish.  The
  associated periodically parametrized geodesic~$\R\to M$ is a minimal
  geodesic, as explained above. Such geodesics are obtained in the
  work of Morse~\cite{morse:1924},
  Klingenberg~\cite{klingenberg:1971}, and
  Gromov~\cite[Section~7.5]{gromov_hyperbolic_groups:87}.  However,
  the minimal geodesics detected in Bangert's existence
  theorem~\cite{bangert:90} and in the present article
  satisfy~$\Apm(\gamma)\subset \{x\in H_1(M\;\R)\mid
  \stabnorm{x}=1\}$, see Proposition~\ref{prop.hom.min.sphere}. Thus,
  the above geodesics will not arise this way.
\end{example}  

\begin{remark}\label{rem:2dim.untypical}
  The situation in dimension at least $3$ is very different from the
  results sketched above. Topological conditions do no longer force
  intersections of curves, and this allows the construction of Hedlund
  examples, see Subsection~\ref{subsec:Hedlund.examples}. One of the
  consequences is that in dimension $2$ the stable norm unit ball is
  never a polytope, see the discussion in an article by
  Massart~\cite[after
    Proposition~6]{massart_GAFA:97}\,\cite{massart_CRAS:97}.  In
  dimension at least~$3$, however, Hedlund examples always exist and
  for such examples the stable norm unit ball is a polytope.
\end{remark}

\subsection{Hedlund examples}\label{subsec:Hedlund.examples}

We now turn to the situation in higher dimensions.  We briefly recall
the construction of Hedlund examples on closed connected manifolds~$M$
of dimension~$n\geq 3$.  The classical constructions of such metrics
only treated the case of the $n$-dimensional torus with~$n\geq 3$, the
first -- and name-giving -- reference being Hedlund's
work~\cite[Section~9]{hedlund:1932}. The construction on the torus was
also discussed in modified versions by
Bangert~\cite[Section~5]{bangert:90} and the first
author~\cite{ammann:97}.  To our knowledge, Hedlund metrics on
arbitrary closed connected manifolds of dimension~$\geq 3$ were
constructed and discussed first in~\cite[IV.1.a and
  IV.1.b]{ammann:diploma} under the name ``Autobahnmetriken'', also
briefly sketched in~\cite[Sec.~6]{ammann:97} as ``Express-way
metrics''. Later they were discussed
further~\cite{babenkobalacheff,jotz}.

In order to construct a Hedlund metric on $M$, we take closed curves
$\sigma_1,\ldots,\sigma_\ell$ based in $x_0\in M$, whose based
homotopy classes $[\sigma_i]$ generate $\pi_1(M,x_0)$. One can easily
achieve that these curves are closed regular curves (\ie immersions of
$S^1$), and because $\dim M=n\geq 3$ we may assume that the
curves~$\sigma_i$ are embeddings. Now we perturb the closed
curves~$\sigma_i\colon S^1\to M$ in the class of free (\ie unbased) closed
curves to curves, denoted as~$\tau_i\colon S^1\to M$, which are
embeddings of circles with pairwise disjoint images. One can easily
construct a Riemannian metric~$g_1$ on~$M$ such that each~$\tau_i$ is
a closed geodesic of length~$1$ 
and such that the distance from $x_0$
to~$\image(\tau_i)$ is ``well-controlled and sufficiently small''.

Now let $L_1, \dots, L_\ell \in \R_{>0}$ be given. We choose
a smooth function~$f\colon M\to \R$ such that
\begin{itemize}
\item $f\stelle{\image(\tau_i)} = L_i$, and $f>L_i$ in a tubular
  neighborhood of~$\image(\tau_i)$, away from $\image(\tau_i)$, 
\item $f$ grows ``rapidly'' on a tubular neighborhood of the
  $\image(\tau_i)$'s, and
\item $f$ is ``very large'' suffiently far away from~$\bigcup_{i=1}^\ell
  \image(\tau_i)$.
\item
  It is also helpful to keep the value of $f$ controlled (and of
  medium size) close to chosen paths $\alpha_i$ from $x_0$ to the set
  $\image(\tau_i)$.
\end{itemize}
One then defines the \emph{Hedlund metric} $\gHed\ceq f^2g_1$. We
refer to the literature~\cite[IV.1.a and IV.1.b]{ammann:diploma} for a precise
desciption of this metric for~$\epsilon\ceq L_1=L_2=\ldots=L_\ell$, which
extends to the generality sketched above.

The curves $\tau_i$ are then geodesics for $\gHed$. One can show
that for suitable choices of~$L_i$ the curves~$\tau_i$ are minimal
geodesics. Furthermore, every minimal geodesic~$\gamma$ stays within a
small tubular neighborhood of~$\bigcup_{i=1}^\ell \image(\tau_i)$ ``most of the
time''. The precise meaning of ``most of the time'' is a bit
subtle. In the case that $\pi_1(M)$ is virtually nilpotent with the
bounded minimal generation property~\cite[Def.~7.1]{ammann:97}\cite[IV.~Def~1.19]{ammann:diploma} --
which includes the case that $\pi_1(M)$ is abelian -- ``most of the time''
means up to a bounded number of short intervals during which the
geodesic~$\gamma$ jumps from a neighborhood of some~$\image(\tau_i)$
to the neighborhood of some~$\image(\tau_j)$ with~$j\neq i$~\cite[IV.1.
  items c, e, f,~g]{ammann:diploma}. The geodesic~$\gamma$
then gives rise to a ``symbol sequence'', which is a function
\[ \Z\to \{\tau_1,\ldots,\tau_\ell,\ol\tau_1,\ldots,\ol\tau_\ell\}
\] 
unique up to shift in~$\Z$; it describes the curves $\tau_i$ followed
by $\gamma$ one after the other. See items
\itemref{exam.hedlund.b2.1}--\itemref{exam.hedlund.b2.6} in the
following example for a list of such sequences for the
case~$\pi_1(M)\cong \Z^2$.

Note that $h(\tau_1), \dots, h(\tau_\ell)$ span~$H_1(M;\R)$ as a vector
space.  One can now show that for all~$v\in H_1(M;\R)$ we have 
\[ \stabnorm{v}=\min\biggl\{\,\sum_{i=1}^\ell |a_i|\cdot L_i \biggm|
a_1, \dots, a_\ell \in \R,\ 
v
= \sum_{i=1}^\ell  a_i \cdot h(\tau_i) \,\biggr\}\,.
\] 
If the geodesic is even $\Z$\=/ho\-mo\-lo\-gi\-cal\-ly minimal, then
the property to have only a ``bounded number of jumps'' holds for arbitrary fundamental groups. 
We consider this in more detail in some special cases.

\begin{example}[Hedlund examples with~$b=2$ and $B$ a parallelogram]\label{exam.hedlund.b2}
  \newcommand\generatorname{\tau} Let $M$ be a closed connected
  manifold of dimension at least $3$ with $b=\dim H_1(M;\R)=2$.  
  Let~$\generatorname_1$ and~$\generatorname_2$ be two simple closed
  curves, representing a $\Z$-basis of~$\HZR \cong \Z^2$. We choose a
  Hedlund metric~$\gHed$ as in~\cite{ammann:diploma} or as above
  with~$L_1=L_2=1$. Then $H_1(M;\R)$ with the stable norm is isometric
  to~$\R^2$ with the $L^1$-norm with respect to the basis consisting
  of~$x_\R\ceq [\generatorname_1]$ and $y_\R\ceq [\generatorname_2]$.
  We write~$x_\Z\ceq [\generatorname_1],y_\Z\ceq [\generatorname_2]$
  for the corresponding integral classes classes in~$H_1(M;\Z)$. With
  our conventions we have ~$-x_\Z = [\overline\generatorname_1]$
  and $- y_\Z = [\overline\generatorname_2]$.  
  As discussed above, every minimal geodesic~$\gamma$ on~$M$ with respect to the Hedlund
  metric~$\gHed$ is described by a suitable word (of two-sided
  infinite length) over~$\{x_\Z, y_\Z,x_\Z^{-1}, y_\Z^{-1}\}$. One can
  show that the following sequences give a full list of all symbol sequences of 
  $\R$\=/ho\-mo\-lo\-gi\-cal\-ly minimal geodesics (modulo geometric
  equivalence). In this list, a negative number~$k$ has to be read as
  taking $|k|$~times the inverse:
  \begin{enumarab}
  \item\label{exam.hedlund.b2.1} for $k\in \Z$: $\cdots x_\Z \cdots
    x_\Z\underbrace{y_\Z\cdots y_\Z}_{k \text{ times}}x_\Z \cdots x_\Z
    \cdots$
  \item\label{exam.hedlund.b2.2} for $k\in \Z$: $\cdots y_\Z \cdots
    y_\Z\underbrace{x_\Z\cdots x_\Z}_{k \text{ times}}y_\Z \cdots y_\Z
    \cdots$
  \item\label{exam.hedlund.b2.3}  $\cdots x_\Z \cdots  x_\Z y_\Z\cdots y_\Z \cdots$
  \item $\cdots y_\Z \cdots y_\Z x_\Z\cdots x_\Z \cdots$
  \item  $\cdots x_\Z \cdots  x_\Z \bar y_\Z\cdots \bar y_\Z \cdots$
  \item\label{exam.hedlund.b2.6}  $\cdots \bar y_\Z \cdots  \bar y_\Z x_\Z\cdots x_\Z \cdots$
  \end{enumarab}
  Note that for each such symbol class there is an $\R$\=/ho\-mo\-lo\-gi\-cal\-ly minimal geodesic. Furthermore, one can show -- but this is supposedly not yet worked out in the literature -- that for suitably chosen Hedlund metrics this minimal geodesic is uniquely determined (up to geometric equivalence) by its symbol sequence.
  If $\pi_1(M)\cong \Z^2$, then every minimal geodesic is
  $\R$\=/ho\-mo\-lo\-gi\-cal\-ly minimal and thus this is the list of
  all minimal geodesics (up to geometric equivalence).

  Minimal geodesics of type~\ref{exam.hedlund.b2.1} and~\ref{exam.hedlund.b2.2} are homologically homoclinic
  and homologically exposed. The case~$k=0$ for~\ref{exam.hedlund.b2.1} resp.~\ref{exam.hedlund.b2.2}
  describes $\generatorname_1$ and~$\generatorname_2$. 
  Bangert's existence results detect one minimal geodesic of type~\ref{exam.hedlund.b2.1} and one of type~\ref{exam.hedlund.b2.2}.

  The four minimal geodesics~\ref{exam.hedlund.b2.3}--\ref{exam.hedlund.b2.6} are homologically heteroclinic
  and homologically exposed.
  Note that the stable norm is the $\ell^1$-norm with respect to this
  generating set. Thus, the stable norm unit ball is a polytope, more
  precisely a parallelogram.  In the case that~$b=2$ and $B$ is a polytope, our
  existence theorem, Theorem~\ref{thm:main.2}~\itemref{thm:main.alt.2}
  predicts the existence of at least four homologically non-homoclinic
  and $\R$\=/ho\-mo\-lo\-gi\-cal\-ly minimal geodesics. In the example
  above, these are the minimal geodesics (3)--(6).  Thus, these
  examples show that our existence result provides an optimal lower
  bound -- the number $4$ -- for homologically non-homoclinic and
  $\R$\=/ho\-mo\-lo\-gi\-cal\-ly minimal geodesics.
\end{example}

The next example is an extension to this, in the sense that the case
$k=2$ in the following example reduces to the previous one.

\begin{example}[Hedlund examples with~$b=2$ and $B$ a $2k$-gon]\label{exam.hedlund.b2.extended}
  \newcommand\generatorname{\tau} Let again $b=2$ and $\dim M\geq
  3$. Let $\generatorname_1,\ldots,\generatorname_k\in H_1(M;\Z)$ be
  given with images~$\generatorname_1^\R,\ldots,\generatorname_k^\R$
  in~$H_1(M;\R)$. We further assume that the convex hull
  of~$\{\generatorname_1^\R,\ldots,\generatorname_k^\R,-\generatorname_1^\R,\ldots,-\generatorname_k^\R\}$
  is a symmetric set~$B$ whose boundary is a convex $2k$-gon, and thus
  $E(B)=V(B)=2k$. Then the Hedlund metric based on these
  curves~$\tau_i$ with~$L_i=1$ has $B$ as stable norm unit ball.  In
  this case, the last phrase in Theorem~\ref{thm:main.2} predicts the
  existence of
  \begin{itemize}
    \item $2k$ homologically non-homoclinic and
      $\R$\=/ho\-mo\-lo\-gi\-cal\-ly minimal geodesics, and 
    \item $k$ homologically homoclinic and
      $\R$\=/ho\-mo\-lo\-gi\-cal\-ly minimal geodesics.
    \end{itemize} 
 This estimate is sharp: for the Hedlund metric in this example there are precisely
 $3k$~geodesics that are $\R$\=/ho\-mo\-lo\-gi\-cal\-ly minimal; and among them
 there are $k$~homologically homoclinic ones, and $2k$~homologically
 heteroclinic (and thus homologically non-homoclinic) ones.
\end{example}

\begin{example}[homologically homoclinic or heteroclinic, but not homologically semi\-/ex\-posed]
  We describe the construction of a Hedlund metric with a finite number of 
  $\R$\=/ho\-mo\-lo\-gi\-cal\-ly minimal geodesics that are not
  homologically semi\-/ex\-posed, but that are either homologically
  heteroclinic or homologically homoclinic. For simplicity of 
  presentation, we specialize to the case of the $3$-torus~$M=T^3=\R^3/\Z^3$,
  but similar constructions are also possible for
  every closed connected manifold~$M$ with $n=\dim M\geq 3$
  and~$b=\dim_\R H_1(M;\R)\geq 2$.
  
  Again as in Example~\ref{example:torus.basics} we assume that the
  Jacobi map~$J$ corresponds to the identity from~$\ucov
  {T^3}=\homZcovtorus3=\homRcovtorus3\cong \R^3$ to~$H_1(T^3;\R)\cong
  \R^3$. In particular, the notion of ``minimal geodesics'' and of 
  ``$\R$- or $\Z$\=/ho\-mo\-lo\-gi\-cal\-ly minimal geodesics''
  coincide. We make the choices such that $H_1(T^3;\Z)\cong\HZR[T^3]$ is
  the lattice generated by the canonical basis~$e_1,e_2,e_3$
  of~$\R^3$. For $i\in\{1,2,3\}$, we define $j(1)=2$, $j(2)=3$, $j(3)=1$, and
  let
  \[t \mapsto
  \tau_i(t)\ceq \Bigl[t\cdot e_i+\frac12 \cdot e_{j(i)}\Bigr]\in T^3 
  \qand
  L_i\ceq 1\,.
  \]
  Now we consider vectors $v_4=\sum_{k=1}^3v_4^ke_k$ and
  $v_5=\sum_{k=1}^3v_5^ke_k$ with $v_4^1,v_5^1\in\Z$,
  $v_4^2,v_5^2\in \Z\setminus\{0\}$, and $v_4^3, v_5^3\in \Z_{>0}$; moreover,
  we set~$L_i\ceq
  \left|v_i^1\right|+ \left|v_i^2\right| + \left|v_i^3\right|$.  We
  assume that $v_i^1$, $v_i^2$, and $v_i^3$ have no common divisor
  $>1$, \ie $v_4$ and $v_5$ are primitive in $\Z^3$. Furthermore
  assume $v_4\neq v_5$. We define for $i=4,5$
  \[t \mapsto \tau_i(t)\ceq \Bigl[\frac{t}{L_i}\cdot v_i + w_i\Bigr]
  \in T^3\,,
  \]
  where $w_4$ and $w_5$ are chosen such that $\tau_i$ and~$\tau_j$
  have disjoint images for~$i, j \in \{1,\dots, 5\}$ with~$i < j$.
  All curves $\tau_1, \dots, \tau_5$
  are closed curves of periodicity~$L_i$. We define~$v_i\ceq e_i$ for~$i\inupto{3}$.

  For these $\tau_i$ and $L_i$ and some basepoint $x_0\in T^3\setminus \bigcup_{i=1}^5 \image(\tau_i)$, we
  construct a Hedlund metric~$\gHed$ as described above; let $\dHed$
  be the associated distance function and let $D_{i,j}\ceq
  \min\left\{\dHed\bigl(\tau_i(t),\tau_j(s)\bigr)\mid t,s\in \R\right\}>0$ for $i\neq j$.
  Let again $\alpha_i$ be a suitable path from~$x_0$ to~$\tau_i$ as in the construction of the function~$f$ at the beginning of this subsection.
  We assume that $f$ is much smaller  along~$\alpha_4 $ and $\alpha_5$ than along~$\alpha_1, \alpha_2,
  \alpha_3$; and we also assume that outside tubular neighborboods of the $\tau_i$ and $\alpha_i$ the function $f$ is much larger.
  In particular, we may assume $3 D_{4,5}<D_{i,j}$ for all~$i\neq j$ with~$\{i,j\}\neq\{4,5\}$.

  Then the stable norm on~$H_1(T^3;\R)\cong \R^3$ is the
  $\ell^1$-norm with respect to~$(e_1,e_e,e_3)$, and
  $\stabnorm{v_i}=L_i$ for~$i\inupto{5}$.
  
  The closed curves $\tau_1, \dots, \tau_5$ are minimal geodesics
  for~$\gHed$.
  We have $A^+(\tau_i)=A^-(\tau_i)=\{(1/L_i)v_i\}$.
  Thus $\tau_1, \dots, \tau_5$ are homologically homoclinic, and they are
  the only minimal geodesics of this type.

  For $i\inupto{3}$, the unique initial and terminal asymptote $v_i=e_i$ of $\tau_i$ is an
  exposed point, thus $\tau_1, \tau_2, \tau_3$ are homologically ex\-posed;
  furthermore, these minimal geodeiscs can be detected by Bangert's
  method. However, the unique asymptote of $\tau_4$ and~$\tau_5$ is
  not exposed, and these $\tau_4$, $\tau_5$ are not homologically
  (semi\=/)\allowbreak{}ex\-posed. They are minimal geodesics and they are homologically
  heteroclinic, but there are no homologically diverging minimal
  geodesics. For example, there is a minimal geodesic~$\gamma$ with
  $A^-(\gamma)=\{(1/L_4)v_4\}$
  and~$A^+(\gamma)=\{(1/L_5)v_5\}$. Geodesics of such a type cannot 
  be detected by Bangert's method. If $(1/L_4)v_4$ or $(1/L_5)v_5$ is
  not on an edge of the stable norm unit ball~$B$, then this geodesic
  is also not detected by our method. 
  The number of $\R$-homologically minimal geodesics is finite 
  as previously shown by the first author~\cite[Korollar~1.16 in IV.1.e]{ammann:diploma}.
  
  Our method provides the
  existence of homologically heteroclinic minimal geodesics with
  $A^-(\gamma)=\{\pm e_i\}$ and $A^+(\gamma)=\{\pm e_j\}$
  with~$i,j\inupto{3}$ and $i\neq j$, where the two choices of sign are
  independent. This provides $12$~geometrically distinct minimal
  geodesics.

  Now let us consider the case that $(1/L_4)v_4$ lies on an edge of~$B$, say
  in the relative interior of the edge $[e_2,e_3]$, \eg
  $v_4=e_2+e_3$, $L_4=2$. We also assume~$(1/L_5)v_5\notin[e_2,e_3]$
  Then it depends on some finer information about $f$ and some choices
  in our proof whether our method, applied to the edge $[e_2,e_3]$
  will provide one of the following non-exclusive items (the number
  corresponds to the cases in
  Proposition~\ref{prop:twomingeod.stronger}) 
  \begin{itemize}
  \item[\itemref{cor:twomingeod.stat.iii.a}]
    minimal geodesics $\gamma_1$ and $\gamma_2$ with
    \[ A^-(\gamma_1)=A^+(\gamma_2)=\{e_2\},\quad A^+(\gamma_1)= A^-(\gamma_2)=\{e_3\}\,.\]
  \item[\itemref{cor:twomingeod.stat.iii.b}]
    minimal geodesics  $\gamma_1, \dots, \gamma_4$ with
    \begin{align*}
      &A^-(\gamma_1)= A^+(\gamma_2)=\{e_2\},\quad  A^-(\gamma_3)=A^+(\gamma_1)=\{e_3\},\\
      & A^-(\gamma_4)= A^+(\gamma_4)=A^-(\gamma_2)= A^+(\gamma_3)=\Bigl\{\frac1{L_4}v_4\Bigr\}\,.
    \end{align*}
    (Proposition~\ref{prop:twomingeod.stronger} only provides
    $\gamma_1$ and two out of the three remaining ones.)
  \item[\itemref{cor:twomingeod.stat.iii.c}]
    minimal geodesics $\gamma_1, \dots, \gamma_5$ with
    \begin{align*}
      &A^-(\gamma_1)= A^+(\gamma_4)=\{e_2\},\quad A^-(\gamma_3)= A^+(\gamma_2)=\{e_3\},\\
      &A^+(\gamma_5)= A^-(\gamma_5)= A^-(\gamma_2)=  A^-(\gamma_4)=A^+(\gamma_1)=A^+(\gamma_3)=\Bigl\{\frac1{L_4}v_4\Bigr\}\,.
    \end{align*}
    (Proposition~\ref{prop:twomingeod.stronger} only provides the four semi-exposed 
    geodesics $\gamma_1,\ldots,\gamma_4$.)
\end{itemize}
\end{example}

\subsection{Some (expected) examples with minimal geodesics of other types}

Examples of Riemannian manifolds with homologically diverging minimal
geodesics with abelian or nilpotent fundamental group are difficult to
construct, see Remark~\ref{rem:typeconj}.  However, if the fundamental
group is Gromov-hyperbolic, their existence follows from known facts,
see Example~\ref{ex.surf.times.stwo}.

Let us first formulate a candidate for a counterexample in
Remark~\ref{rem:typeconj}, which would imply that the bound in
Proposition~\ref{prop:twomingeod.stronger}~\itemref{cor:twomingeod.stat.iii.c}
cannot be improved to~$m=5$. Then we discuss the Gromov-hyperbolic case.

\begin{remark}\label{rem:typeconj}
  We expect that examples with the following properties exist,
  although to our knowledge the construction of such examples
  has not yet been worked out.  These examples consist of a closed connected
  Riemannian manifold~$(M,g)$, an exposed edge~$[x,y]$ of the stable norm unit
  ball, and points~$z_1 \in (x,y)$, $z_2 \in (z_1,y)$ such that there
  are \emph{precisely} four geometrically distinct minimal geodesics
  $\gamma_1, \dots, \gamma_4 \colon \R\to M$ with~$\Atot(\gamma_i)
  \cap (x,y) \neq \emptyset$ such that 
  \begin{align*}
    A^-(\gamma_1)&=\,\{x\},\quad
    A^+(\gamma_1)\= [z_2,y],\quad
    A^-(\gamma_2)\=[x,z_1],\quad
    A^+(\gamma_2)\=\{y\}
      \\
    A^+(\gamma_3)&=\,\{x\},\quad
    A^-(\gamma_3)\= [z_2,y],\quad
    A^+(\gamma_4)\=[x,z_1],\quad
    A^-(\gamma_4)\=\{y\}\,.
 \end{align*}
  Hence, these four geodesics would be homologically semi\-/ex\-posed amd
  homologically semi\-/con\-ver\-ging.
\end{remark}
  
\begin{example}\label{ex.surf.times.stwo}
  Let $N$ be an oriented closed connected surface of genus~$k \geq 2$
  and let $M\ceq N\times S^2$. We choose a Hedlund metric on~$M$, with
  respect to standard generators~$a_1, \dots, a_k, b_1, \dots, b_k$
  coming from writing~$M$ as a quotient of a $4k$-gon in the standard
  way. This gives rise to disjoint simple closed curves~$\tau_1, \dots, \tau_{2k}$,
  for $L_i\ceq 1$; we construct the associated Hedlund metric.
  Then $(H_1(M;\R),\stabnorm\argu)$ is isometric to~$\R^{2k}$ with the
  $\ell^1$-norm, and all~$\tau_i$ are minimal geodesics. In every exposed edge
  of the stable norm unit ball, we find minimal geodesics~$\gamma_i$,
  $i\inupto{4}$ with all the properties in Remark~\ref{rem:typeconj},
  but without the word ``precisely''. In the present example,  infinitely
  many geometrically distinct minimal geodesics~$\gamma$
  with~$\Atot(\gamma) \subset [x,y]$ and~$\Atot(\gamma) \cap (x,y)
  \neq \emptyset$ do exist. Furthermore, we also can show the
  existence of homologically diverging, homologically
  semi\-/con\-ver\-ging, homologicyally semi\-/ex\-posed geodesics,
  and many more types.
\end{example}  

\subsection{Lagrangian approach to detecting minimal geodesics}\label{subsec:Lagrangian}

Geodesics can be viewed as the flow lines of a Lagrangian or Hamiltonian system. Let $\omega$ be a closed $1$-form on a Riemannian manifold~$(M,g)$. We define the \emph{Lagrange function} $L\colon TM\to \RR$ as
$$L(v)=\tfrac12 g(v,v) - \omega(v).$$
Let $\maC_T(x_1,x_2)$ be the space of piecewise $C^1$-curves from $x_1\in M$ to $x_2\in M$, parametrized on $[0,T]$.
For any $\gamma\in  \maC_T(x_1,x_2)$ we define its \emph{action} as
  \[A(\gamma)\ceq \int_0^T L(\dot\gamma(t))\,\upd t.\]
  A curve $\gamma\colon[0,T]\to M$ with $\gamma(0)=x_1$ and $\gamma(T)=x_2$ is a minimizer of the functional $A\colon \maC_T(x_1,x_2)\to \RR$
  if, and only if, $\gamma$ is a geodesic. If the stable norm unit ball $B\subset H_1(M,\RR)$ is a polytope, we fix an edge $[x,y]$ of $B$ where $x,y\in B$. We choose a closed $1$-form $\omega$ such that $\alpha=[\omega]\in \HdR$ satisfies $\<\alpha,z\>=1 $ for all $z\in [x,y]$ and $\<\alpha,z\><1$ for all $z\in B\setminus[x,y]$,
  \ie we assume $F_\alpha=[x,y]$. In particular $\stabnormdual{\alpha}=1$.               

  A referee of a previous version of this article informed us that he expects that the theory of static classes for Lagrangian systems can be used for an alternative approach to finding minimal geodesics.
  In particular, the referee suggested that the main results of this article may also be deduced from the work of Contreras and Paternain \cite[Theorems~A and~B]{contreras.paternain:2002}, see also Fathi \cite{fathi:1998}.
 We checked that Ma\~n\'e's critical value $c(L)$ defined in this theory satisfies
  $$c(L)=\tfrac12 \left(\stabnormdual{[\omega]}\right)^2=\tfrac12.$$
  We think that applying the theory by Contreras--Paternain and Fathi would require that $\RR$-homologically minimal geodesics with asymptotes in $[x,y]$ are static curves in the sense of \cite{contreras.paternain:2002}, at least after reparametrisation.  However, we have not found a reference for this relation so far, and we expect additional hypotheses.
  In particular, this will require a careful choice of representatitve $\omega$  within the class $\alpha=[\omega]\in \HdR$. %

  
This Lagrangian approach could be interesting for several reasons, in particular as this might help to find new types of homologically homoclinic minimal geodesic.

\subsection{Comparison to results by Bolotin and Rabinowitz}\label{subsec:compar_BR}

For the case that $M=T^n$ is an $n$-dimensional torus, Bolotin and Rabinowitz~\cite{bolotin.rabinowitz:99} show the existence of homoclinic and heteroclinic minimal geodesics in the sense of Poincar\'e, see Subsection~\ref{subsec:homoclinic}. They fix a non-trivial class~$v\in  H_1(T^n;\Z)\cong\pi_1(T_n)\cong \Z^n$, and two conditions (S1) and (S2) are assumed -- see below for details.

Let $\maC_v$ be the set of all $\Cinftypw$-loops representing $v$. For simplicity we identify loops in  $\maC_v$ that are reparametrizations of each other. 
According to Lemma~\ref{lemma.length.minimizing.loop} the length functional $\maL\colon\maC_v\to[0,\infty)$ attains its infimum, and each minimizer $\tau\in \maC_v$
is a closed geodesic (if parametrized proportionally to arclength).
The results in the article~\cite{bolotin.rabinowitz:99} assume that at least one minimizer -- and thus all minimizers -- are minimal geodesics.
Let us rescale the metric $g$ on $T^n$ such that $\stabnorm{v}=1$, \ie for every minimizer $\tau$ we have $\maL(\tau)=\groupnorm{[\tau]}=\stabnorm{v}=1$.

Assumption (S2)~\cite{bolotin.rabinowitz:99} claims that the number~$p$ of minimizers in~$\maC_v$ is finite. We omit the precise definition of~(S1), but let us mention that it follows from some mild conditions, e.g. if $T^n$ has a symmetry that reflects ``orthogonally to~$v$''. The condition~(S1) implies that length minimizers in~$\maC_v$ are (closed) minimal geodesics.  

The main results  of Bolotin and Rabinowitz~\cite{bolotin.rabinowitz:99} yield homoclinic and heteroclinic minimal geodesics~$\gamma$. All these geodesics~$\gamma$ have~$\Atot(\gamma)=\{v\}$, thus they are homologically homoclinic in the sense of Subsection~\ref{subsec:homoclinic}, and $t\mapsto\gamma(t)$ is asymptotic to some~$\tau_+$ for~$t\to+\infty$ and  asymptotic to some~$\tau_-$ for~$t\to-\infty$, where $\tau_\pm$ are length minimizers in~$\maC_v$.
In the case $p=1$, \ie if there is a unique closed minimal geodesic~$\tau_1$ in~$\maC_v$, then
there exist at least~$n$ homoclinic minimal geodesics~\cite[Theorem 1.7]{bolotin.rabinowitz:99}, all of them are asymptotic to~$\tau_1$ for $t\to +\infty$ and~$t\to -\infty$. For $p>1$, there exist at least~$n+p-1$ heteroclinic minimal geodesic, with the above properties~\cite[Theorem 1.8]{bolotin.rabinowitz:99}. 

All these minimal geodesics are homologically homoclinic,
thus they are of a different type than the ones detected by our methods.
This fact sheds some light on the minimal number~$\Nmin(T^3)$ of minimal geodesics on~$T^3$.

To be more precise, let $\maN(T^3,g)$ be the number of geometrically distinct minimal geodesics on~$(T^3,g)$, and define 
\begin{align*}
  \Nmin(T^3)
  \ceq \min\bigl\{
  & \maN(T^n,g)  \mid \text{$g$ is a Riemannian metric on~$T^3$}\bigr\}.
\end{align*}
We stated in Subsection~\ref{subsec:introsharp} that we expect~$15\leq \Nmin(T^3)\leq 27$, and it would be interesting to determine the precise value of~$\Nmin(T^3)$.

Suppose that $g$ is a Hedlund metric on~$T^n$ with finitely many minimal geodesics. Let us assume that $g$ also satisfies (S1) and~(S2)~\cite[Theorem 1.8]{bolotin.rabinowitz:99} -- all explicit constructions of Hedlund metrics known to the authors do so.

Then the results by Bangert~\cite{bangert:90} and Bolotin--Rabinowitz~\cite{bolotin.rabinowitz:99} provide at least~$9$ homologically homoclinic minimal geodesics and our method yields at least~$12$ homologically non-homoclinic minimal geodesics. Thus one has $\maN(T^3,g)\geq 21$ for all Hedlund metrics~$g$ satisfying the assumptions of Bolotin and Rabinowitz~\cite{bolotin.rabinowitz:99}.

As a conclusion we see that $\Nmin(T^3)$ is either essentially larger than~$15$, or one has to use much finer construction methods in order to obtain a metric~$g$ with~$\maN(T^3,g)$ close to~$15$. It thus would be desirable to find a refined existence theorem for homologically homoclinic minimal geodesics in the style of Bolotin--Rabinowitz~\cite{bolotin.rabinowitz:99} without assumptions (S1) and~(S2), strengthening Bangert's existence result~\cite[Theorem~4.4]{bangert:90}. Such an improved existence result would imply that $\Nmin(T^3)$ is essentially larger than~$15$.

In summary, we see that our lower bound on the number of homologically non-homoclinic minimal geodesics is optimal on~$T^3$, but it is still unknown whether the existing lower bound on the number of homologically homoclinic minimal geodesics is optimal.


\appendix
\section{Length and minimizing geodesics}\label{app.minimzing.geodesics}

We collect some facts about the length functional and
on distance minimizing geodesics that are well-known or
straightforward consequences of well-known facts.

\subsection{Length of curves}

Let $(X,d)$ be a metric space. For a continuous curve~$\gamma\colon I
\to X$ on a non-empty interval~$I \subset \R$, we define the
\emph{length} by
\[   \maL(\gamma)
\ceq \sup \Bigl\{ \sum_{i=1}^kd(\gamma(t_i),\gamma(t_{i-1})
          \Bigm| k \in \N_0,\ t_0, \dots,t_k\in I, \ t_0<t_1\cdots <t_k
          \Bigr\} \in [0,\infty]\,.
\]

If $(X,d)$ arises from a Riemannian manifold and if~$\gamma$ is
piecewise~$C^1$, then $\maL(\gamma)=\int_I \|\dot\gamma(t)\|\,\upd t\in
[0,\infty]$.

\subsection{Limiting geodesics}

We provide a proof of Lemma~\ref{lemma:limitinggeodesic}: 

\begin{proof}[{Proof of Lemma~\ref{lemma:limitinggeodesic}}]
  It follows from the theorem of Picard and Lindel\"of for suitable
  ordinary differential equations, in particular from the smooth
  dependence of the solution on the initial conditions, that
  $\sigma_i$ converges to~$\sigma_\infty$, in the sense of uniform
  convergence on all compact intervals.

  The induced distance function~$\hat d \colon \wihat M \times \wihat M \to
  \R_{\geq 0}$ is continuous; thus, as all curves~$\sigma_i$ are
  minimizing, this allows the following calculation for all~$s, t\in \R$
  with~$s < t$:
  \begin{eqnarray*}
    \hat d\bigl(\sigma_\infty(t),\sigma_\infty(s)\bigr) &=& \hat d\bigl(\lim_{i\to\infty}\sigma_i(t),\lim_{i\to\infty}\sigma_i(s)\bigr)\\
    &=& \lim_{i\to\infty}\underbrace{\hat d\bigl(\sigma_i(t),\sigma_i(s)\bigr)}_{=t-s}\,=\, t-s\,.
  \end{eqnarray*}
  Therefore, $\sigma_\infty$ is minimizing as well.
\end{proof}

\subsection{Existence of minimal geodesics and the fundamental group}

\begin{proposition}\label{prop.infinite.pi}
  Let $(M,g)$ be a closed connected Riemannian manifold. Then there
  exists a minimal geodesic on~$M$ if and only if $\pi_1(M)$ is
  infinite.
\end{proposition}

\begin{proof}
  The universal covering~$\ucov M$ of~$M$ has infinite diameter~$\diam
  \ucov M$ if and only if $\pi_1(M)$ is infinite.
  If there exists a minimal geodesic on~$M$, then clearly, $\diam \ucov M=\infty$.

  Conversely, let $\diam \ucov M = \infty$.
  Then there are points $p_i,q_i\in \ucov M$ with $\ell_i\ceq \ucov d(p_i,q_i)/2\to \infty$.
  We choose a minimizing geodesic~$\sigma_i\colon [-\ell_i,\ell_i]\to \ucov M$ with~$\sigma_i(-\ell_i)=p_i$ and~$\sigma_i(\ell_i)=q_i$, parametrized by arclength. 
  As the deck transformation group acts cocompactly on $\ucov M$ and on the unit tangent bundle of $\ucov M$, we can assume -- without loss of generality -- 
  that there is a compact set~$K$ containing~$\dot\sigma_i(0)$ for all~$i\in \N$.
  This allows us to pass to a subsequence such that $v_\infty\ceq \lim_{i\to\infty}\dot\sigma_i(0)$ exists.
  We apply Lemma~\ref{lemma:limitinggeodesic} for $a_i\ceq
  -\ell_i$ and~$b_i\ceq \ell_i$. The resulting
  geodesic~$\sigma_\infty\colon\R\to \ucov M,\ t \mapsto \exp(tv_\infty)$ is globally distance
  minimizing, and thus -- by definition -- its projection to~$M$ is a
  minimal geodesic.
\end{proof}

\subsection{Length-minimizing curves}

Next, we prove that the infimum in Definition~\ref{def:N} is 
attained and that it does not matter whether the infimum ranges over
curves of high or low regularity, as long as they are continuous. Let us
recall a lemma that is a reformulation of a lemma in Sakai's book.

\begin{lemma}[{\cite[Lemma~V.1.5 (1)]{sakai_Riemannian_Geometry}}]\label{lemma.length.minimizing.loop}
  Let $(M,g)$ be a closed connected Riemannian manifold. Then, in every
  nontrivial free homotopy class of closed curves~$S^1\to M$, there exists 
  a curve of minimal length. This curve can be parametrized by
  arclength and then it is a closed geodesic.
\end{lemma}

In the following, if $\gamma$ is a closed curve in~$M$, then
$h(\gamma)$ denotes the element in~$\HZR$ represented by~$\gamma$
via the Hurewicz homomorphism.

\begin{lemma}\label{lem.loop.attain.inf}
  Let $M$ be a closed connected Riemannian manifold
  and let $x\in \HZR\setminus\{0\}$.
  Then there is a closed geodesic $\gamma:[0,N(x)]\to M$ with $h(\gamma)=x$ and 
     \[\maL(\gamma)=  \inf
    \bigl\{ \maL (\tau)
    \bigm| \text{$\tau$ is a continuous loop with~$h(\tau) = x$}
    \bigr\}
    \,.
    \]
\end{lemma}    

The lemma can be proved via arguments analogous to a related proof in
Sakai's book~\cite[Lemma~V.1.5~(2)]{sakai_Riemannian_Geometry},
enriched with standard smoothing techniques~\cite[\Paragraph 16
  and~17]{milnor::morse-theory}.

\subsection{Distance estimates in finite-sheeted coverings}

\begin{lemma}\label{lem:finite-covering}
  Let $Q_1\xrightarrow{p} Q_2\xrightarrow{\pi} M$ be a sequence of
  Riemannian coverings, where~$M$ is a closed Riemannian manifold with
  metric~$g$. (Then also $Q_1\xrightarrow{\pi \circ p} M$ is a
  covering.) We assume that the coverings $p$, $\pi$, and $\pi \circ p$
  are normal, \ie the deck transformation groups act transitively (and
  freely) on the fibres of the respective coverings. Let $g_i$ be the pullback of~$g$
  to~$Q_i$ and let $d_i$ be the induced distance function on~$Q_i$.
  \begin{enumarab}
  \item
    Then, for all~$x,y\in Q_1$ we have
    \[ d_2\bigl(p(x),p(y)\bigr)\leq d_1(x,y)
    \,.
    \]
  \item
    If $Q_1\xrightarrow{p} Q_2$ is a finite covering, then there is
    a~$C<\infty$, such that for all~$x,y\in Q_1$:
    \[ d_1(x,y)\leq d_2\bigl(p(x),p(y)\bigr)+C
    \]
  \end{enumarab}
\end{lemma}
\begin{proof}
  The first part is clear because $p$ is a local isometry.
  We show the second part:
  For~$\hat x\in Q_2$ we define
  \[ m(\hat x)
  \ceq \diam{}_{d_1} p^{-1}(\hat x)
  = \max \left\{d_1(x,y)\mid x,y \in p^{-1}(\hat x)\right\}
  \,.
  \]
  It is easy to see that $m\colon Q_2\to \R$ is a
  continuous function. We show that $m$ factors over~$\pi \colon Q_2 \to M$: 
  As $\pi$ is normal,
  for all~$\hat x_1, \hat x_2 \in Q_2$ with~$\pi(\hat x_1) = \pi(\hat x_2)$,
  there exists a deck transformation~$\hat f$ of~$Q_2 \to M$ with~$\hat f(\hat x_1)=\hat
  x_2$. Because $\pi \circ p$ is normal, $\hat f$
  lifts to a deck transformation~$f$ of~$p$; in particular,
  $\hat f$ maps $p^{-1}(\hat x_1)$ bijectively to~$p^{-1}(\hat
  x_2)$. Thus $m$ is invariant under the deck transformations of~$Q_2\to M$. 
  Therefore, $m$ factors as $m=m'\circ \pi$ for some continuous map~$m'\colon
  M\to \R$. In particular, the function~$m$ attains
  its maximum and we define
  \[ C\ceq \max_{\hat x\in Q_2} m(\hat x)
  \,.
  \]
  Let $x, y \in Q_1$. 
  The Riemannian manifold~$Q_2$ is complete and so by the Hopf--Rinow
  theorem there is a path~$\gamma$ in~$Q_2$ from $p(x)$ to~$p(y)$ of
  length~$\ell\ceq d_2\bigl(p(x),p(y)\bigr)$. Let~$\gamma'$ be the
  lift of~$\gamma$ to~$Q_1$ starting in~$x$. If $y'$ is the endpoint
  of~$\gamma'$, then $p(y')$ is the endpoint of~$\gamma$ and so~$p(y')=p(y)$.
  By construction of~$C$, we have
  $d_1\bigl(y',y\bigr) \leq C$. Furthermore, the existence of~$\gamma'$,
  which has length~$\ell$, shows $d_1\bigl(x,y'\bigr)\leq
  \ell$. Then the statement follows from the triangle inequality.
\end{proof}

\ifwithappendixstable
\section{Other aspects of the stable norm}\label{app:stable-norm}

We recall (special cases of) Federer's notion of mass~$\gnorm{\argu}{K}$,
which is a group norm on~$H_1(M;K)$ for~$K\in
\{\Z,\Q,\R\}$. We then show that the norms~$\gnorm\argu \R$ and
$\stabnorm\argu$ are equal.

This equivalence is well-known, and -- as already mentioned in
Section~\ref{subsec:stable-norm} -- sometimes attributed to Gromov's
book on metric structures on Riemannian manifolds~\cite{gromov:99} or
its French version~\cite{gromov::structures-metriques}.  The
equivalence is based on work by Federer, requiring involved
notation. We include a proof to keep our presentation self-contained.

\begin{definition}[mass of $1$\=/ho\-mo\-lo\-gy classes]\label{def:federer-mass}
  Let $(M,g)$ be a Riemannian manifold and $K\in \{\Z,\Q,\R\}$.  For a
  singular $1$-chain~$a=\sum_{i=1}^m a_i \sigma_i\in C_1(M;K)$ with
  coefficients~$a_i$ in~$K$, we define
  \[ \gnorm{a}{}\ceq \sum_{i=1}^m |a_i|\cdot \maL(\sigma_i)
  \,.
  \]
  For~$\alpha \in H_1(M;K)$, we define the \emph{$K$-mass} of~$\alpha$ as
  \[
  \gnorm{\alpha}K\ceq \inf_{a\in \alpha} \gnorm{a}{}
  \,.
  \]
\end{definition}

It follows directly from the definition that the map
$\gnorm{\argu}K\colon H_1(M;K)\to \R$ is symmetric and satisfies the
triangle inequality; i.e., it is what is sometimes referred to as a
\emph{group semi\-/norm}. The
manifold~$\R^2\setminus \{0\}$ with the Euclidean metric shows that in
general we do not get positive definiteness of~$\gnorm{A}K$ without
further assumptions. In the following, we therefore often
require that the manifold is closed.

The group semi\-/norm~$\gnorm{\argu}\R$ is $\R$\=/homogeneous and
thus a semi\-/norm on a vector space in the usual sense.

\begin{lemma}\label{lem:gnormrational}
  Let $(M,g)$ be a Riemannian manifold. Then, the
  change-of-coefficients map~$\args_\R \colon H_1(M;\Q) \to H_1(M;\R)$
  is isometric: For all~$\alpha \in H_1(M;\Q)$, we have
  \[ \gnorm {\alpha_\R}\R = \gnorm \alpha \Q
  \,.
  \]
\end{lemma}
\begin{proof}
  The change-of-coefficients map is induced by the inclusion~$i \colon
  C_*(M;\Q) \to C_*(M;\R)$ of the singular chain complexes. The
  map~$i$ is isometric in each degree and has dense image with respect
  to~$\|\args\|_g$.
  Approximating boundaries therefore shows that $\arg_\R = H_1(i)
  \colon H_1(M;\Q) \to H_1(M;\R)$ is isometric with respect to the
  induced semi\-/norms~$\gnorm{\args}\Q$ and~$\gnorm{\args}\R$
  on homology~\cite[Lemma~2.9]{mschmidt}\,\cite[Lemma~1.7]{loehthesis}
  (the cited proofs carry over to this slightly more general setting). 
\end{proof}

\begin{lemma}
  Let $(M,g)$ be a Riemannian manifold. Then $\gnorm\argu \R$
  on~$H_1(M;\R)$ extends the homogenisation (in the sense of
  Proposition~\ref{prop:homog}) of~$\gnorm\argu \Z$.
\end{lemma}
\begin{proof}
  Let $\alpha\in H_1(M;\Z)$ and we denote its image in~$H_1(M;\R)$
  with respect to the change-of-coefficients map~$H_1(M;\Z)\to
  H_1(M;\R)$ by~$\alpha_\R$ (and similarly for~$\Q$). Using
  Lemma~\ref{lem:gnormrational} and the definition of~$\gnorm\args\Q$,
  we obtain
  \[ \gnorm{\alpha_\R}\R
  = \gnorm{\alpha_\Q}{\Q}
  = \inf_{k\in\N} \frac1k \cdot \gnorm{k \cdot \alpha}\Z
  \,.
  \] 
  According to Proposition~\ref{prop:homog}, the right-hand side
  equals to the homogenisation of~$\gnorm\argu \Z$ on~$\alpha$.
\end{proof}

\begin{proposition}[mass and stable norm]\label{prop.gnorm-equals-stabnorm}
  Let $(M,g)$ be a closed connected Riemannian manifold. Then 
  $\gnorm\argu \R$ and $\stabnorm\argu$ coincide on~$H_1(M;\R)$.
\end{proposition}

\begin{proof}
  We first show~$\gnorm{\argu}\R\leq\stabnorm\argu$:
  Let $\alpha\in \HZR$.
  Then there is a loop~$\gamma$ of length~$\groupnorm{\alpha}$
  representing~$\alpha$.
  If we re-interpret~$\gamma$ as a $1$-cycle, we get
  \[ \gnorm{\alpha_\R}\R
  \leq \gnorm{\alpha}\Z
  \leq \gnorm{\gamma}{}
  = \maL(\gamma)
  =  N(\gamma)
  \,.
  \]
  Applying homogenisation (Proposition~\ref{prop:homog}) and
  using the fact that $\gnorm\argu\R$ is homogeneous, we
  obtain~$\gnorm{\argu}\R\leq\stabnorm\argu$.

  We now prove the converse inequality. As rational classes
  are dense in~$H_1(M;\R)$ with respect to $\gnorm\args\R$ and $\stabnorm{\args}$,
  it suffices to prove the estimate~$\gnorm\args\R \geq \stabnorm\args$
  on rational classes. Moreover, in view of Lemma~\ref{lem:gnormrational},
  we may replace~$\gnorm\args\R$ with~$\gnorm\args\Q$. 

  Let $\alpha \in H_1(M;\Q)$ and $\varepsilon \in \R_{>0}$. Then,
  there is a singular chain~$a= \sum_{i=1}^m a_i \sigma_i \in
  C_1(M;\Q)$ representing~$\alpha$ with~$\gnorm a {} \leq \gnorm
  \alpha\Q + \varepsilon$. We will replace~$a$ with loops
  that efficiently represent multiples of~$\alpha$.
  More precisely, let $k_0 \in \N$ be a common multiple
  of the denominators appearing in~$a_1, \dots, a_m \in \Q$.
  There exists a chain~$b = \sum_{i=1}^{m'} \sigma_i' \in C_1(M;\Z)$,
  consisting of a sum of \emph{loops}~$\sigma'_i$, with the
  following properties:
  \[ [b] = k_0 \cdot \alpha
  \qand
  \gnorm{b}{}
  =\sum_{i=1}^{m'} \maL(\sigma_i')
  \leq k_0\cdot\bigl(\gnorm{\alpha}\Q+ \epsilon\bigr)
  \]
  Indeed, because $k_0 \cdot a$ is an integral cycle, 
  we can re-organise the singular $1$-simplices appearing in~$k_0 \cdot a$
  (with multiplicities/orientations given by their coefficients)
  into a sum of loops. This procedure does not affect the homology class.

  We parametrize the singular simplices~$\sigma_i'$ on~$[0,\ell_i]$;
  moreover, for each~$i$, we choose a path~$\tau_i$ from~$\sigma_i(0)$
  to~$\sigma_{i+1}(0)$ with~$\maL(\tau_i) \leq \diam M$. For~$k_1 \in \N$,
  let us consider 
  \begin{multline*}
    w_{k_1}
    \ceq \underbrace{\sigma_1'*\cdots*\sigma_1'}_{k_1\text{-times}}*\,\tau_1*\underbrace{\sigma_2'*\cdots*\sigma_2'}_{k_1\text{-times}}*\,\tau_2*\cdots * \underbrace{\sigma_{m'}'*\cdots*\sigma_{m'}'}_{k_1\text{-times}}*\,\bar\tau_{m'-1}*\cdots *\bar\tau_1, 
  \end{multline*}
  which provides a loop based at~$\si_1'(0)$ with $[w_{k_1}]= k_1\cdot
  [k_0\cdot a]\in\HZR$ and of length
  \[ \maL(w_{k_1})\leq k_1\cdot \sum_{i=1}^{m'} \maL(\sigma_i') + 2(m'-1)\cdot\diam M
  \,.
  \] 
  Thus
  \[ \frac1{k_1}\cdot\groupnorm{  k_1\cdot [k_0\cdot a]}
  \leq \frac1{k_1}\cdot \maL(w_{k_1})
  \leq  \sum_{i = 1}^{m'} \maL(\sigma_i') + \frac{2(m'-1)}{k_1} \cdot \diam M
  \,.
  \]
  As $\stabnorm\argu$ was defined as the homogenisation of
  $\groupnorm\argu$, the left-hand side converges for~$k_1\to\infty$
  to~$\stabnorm[big]{ [k_0\cdot a]}=k_0\cdot \stabnorm{\alpha}$.
  On the other hand, the
  limit of the right-hand side is bounded from above
  by~$k_0 \cdot (\gnorm{\alpha}\Q+ \epsilon)$. Dividing by~$k_0$ and
  taking~$\varepsilon \to 0$, we get $\stabnorm{\alpha}\leq \gnorm{\alpha}\R$.
\end{proof}

If $M$ is closed, 
for~$\omega\in \Omega^1(M)$, we define the supremum norm (or
$L^\infty$-norm or $C^0$-norm) as
\[ \innorm{\omega}{L^\infty}{}
\ceq \sup_{x\in M} \bigl|\omega(x)\bigr|_g\,.
\] 
For $1$-forms, this coincides with Gromov's notion of
\emph{comass}~\cite[4.15]{gromov:99}.
  
\begin{proposition}[duality and the comass]
  Let $M$ be a closed connected Riemannian manifold.
  Then, for all~$\varphi \in H^1_{\derham}(M)$ we have:
  \[ \stabnormdual{\varphi}
  =\inf_{\omega\in \varphi} \innorm{\omega}{L^\infty}{}\,.
  \] 
\end{proposition}
  
\begin{proof}[Proof of the inequality $\leq$]
  Let $\omega\in \Omega^1(M)$. Take $v\in \HZR$ and a closed
  geodesic~$\gamma$ representing~$v$ with~$\maL(\gamma)=\groupnorm{v}$.
  Then
  \begin{eqnarray*}
    \bigl|\<[\omega],v\>\bigr|
    &=& \Bigl|\int_\gamma \omega\,\Bigr|
    \,\leq\, \maL(\gamma) \cdot \innorm{\omega}{L^\infty}{}
    \,=\, \groupnorm{v}\cdot \innorm{\omega}{L^\infty}{}\,.
  \end{eqnarray*}
  We apply homogenisation. 
  As the left-hand side is homogeneous in~$v$, we obtain 
  \begin{eqnarray}\label{ineq.stab}
    \bigl|\<[\omega],v\>\bigr|
    &\leq& \stabnorm{v} \cdot \innorm{\omega}{L^\infty}{}\,.
  \end{eqnarray} 
  By homogeneity of this inequality, this estimate~\eqref{ineq.stab}
  also holds for all
  rational classes~$v\in H_1(M;\Q)$; and then by density for all~$v\in
  H_1(M;\R)$.  Using the defining Equation~\eqref{def.stabnormdual}
  of the dual of the stable norm, this implies $\stabnormdual{[\omega]}\leq
  \innorm{\omega}{L^\infty}{}$, which proves the claimed inequality.
\end{proof}

The proof of the inequality~``$\geq$'' requires the construction of
suitable $1$-forms~$\omega$. A proof is given by Gromov~\cite[4.35]{gromov:99}.

\begin{comparisonliterature*}
  Gromov~\cite[Proposition~2.22]{gromov::structures-metriques} calls
  the vector space norm obtained from a homogenisation of a group
  norm~$\innorm\argu{}{}$ in the sense of Proposition~\ref{prop:homog}
  \emph{norme limite} and denotes it as~$\innorm\argu{}{lim}$.  In the
  English edition~\cite[4.17]{gromov:99}, the norm is denoted in the same way
  and called \emph{limit norm}, but the ``Proposition~2.22''
  to which he refers differs in the English version.

  Gromov uses the symbol $\innorm\argu{}{}$ or $\innorm\argu{H_1}{}$
  for all the norms $\gnorm\argu{}\colon C_1(M;K)\to K$,
  $\gnorm\argu{\Z}\colon H_1(M;\Z)\to \Z$, and $\gnorm\argu{\R}\colon
  H_1(M;\R)\to \R$. This does not lead to ambiguities as the argument~$\alpha$
  of~$\innorm{\alpha}{}{}$ will determine, which (group) norm is
  meant.
\end{comparisonliterature*}

\fi 

\ifwithappendixtwodim
\section{More on minimal geodesics in dimension~$2$}\label{app:twodim}

In this appendix let~$M$ be an orientable, closed, connected surface of genus at
  least~$1$ with Riemannian metric~$g$. 
For $\alpha\in \pi_1(M)$ let $\freehom{\alpha}$ be its free homotopy class.
The length functional attains its minimum in any free homotopy class $\freehom{\alpha}$, 
and we write $N_0\bigl(\freehom{\alpha}\bigr)$ for the minimum of this length.
Let us recall the following classical lemma, for which one may find a proof, for example, in an article by Bleecker~\cite{bleecker:74}.

\begin{lemma}\label{lemma:closed-non-prim-geod}
  Let $M$ be an orientable, closed, connected surface of genus at
  least~$1$ with Riemannian metric~$g$.
  Let~$\gamma\colon[0,L]\to M$ 
  be a closed geodesic in $M$ minimizing length within its 
  free homotopy class~$\freehomn{\alpha^k}$ where $\alpha\in \pi_1(M)$ is primitive and $k\in\N$.
  Then $\gamma\colon[0,L/k]\to M$ is a closed geodesic in $M$ minimizing length within its 
  free homotopy class~$\freehom{\alpha}$, and $\gamma$ extends to a periodic curve $\R\to M$ of period~$L/k$.
  in particular, we have for every~$\beta\in \pi_1(M)$:
    \[N_0\bigl(\freehom{\beta^k}\bigr)=k\cdot N_0\bigl(\freehom{\beta}\bigr)\,.\]
\end{lemma}  

\begin{proof}We choose $x_0\ceq \gamma(0)$ and a lift $\ucov x_0\in \ucov M$ as basepoints. 
This allows to identify $\pi_1(M,x_0)$ with the deck transformations of $\ucov M\to M$.
We consider $\wihat M\ceq \ucov M/\<\alpha^k\>$, which is diffemorphic to~$S^1\times \R$. 
Then $\gamma$ lifts to a simple closed geodesic $\wihat\gamma\colon [0,L]\to\wihat M$ with $\wihat\gamma(0)=[\ucov x_0]\qec \wihat x_0$.
The closed geodesic $\wihat\gamma$ generates the infinite cyclic group $\pi_1(\wihat M)\cong H_1(\wihat M;\Z)$. 
The group $\{\alpha^m\mid m\in \Z\}\subset\pi_1(M,x_0)$ acts freely and isometrically on $\wihat M$, and $\alpha^k$ acts trivially. 

We claim that the curves $\wihat\gamma$ and $\alpha\cdot \wihat\gamma$ necessarily have to intersect in at least two points. The Jordan curve theorem implies that
$\wihat M\setminus\image(\wihat\gamma)$ has two components, and orientation allows us to call them the right and the left connected component of~$\wihat\gamma$ 
(this requires the choice of some convention, but any choice will be fine). If~$\wihat\gamma$ and $\alpha\cdot \wihat\gamma$ do not intersect,
then  $\alpha\cdot \wihat\gamma$ has to run entirely in the left (or right) connected component of~$\wihat\gamma$. 
As~$\alpha$ preserves orientation, 
$\alpha^2\cdot \wihat\gamma$ has to run entirely in the left (or right) connected component of $\alpha \cdot\wihat\gamma$ and thus in the  left (or right) connected component of  $\wihat\gamma$.
By induction, we get the same statement for $\alpha^k\cdot \wihat\gamma$ instead of $\alpha^2\cdot \wihat\gamma$, 
but this is obviously in contradiction to $\alpha^k\cdot \wihat\gamma =\wihat\gamma$. 

If the curves~$\wihat\gamma$ and $\alpha\cdot \wihat\gamma$ intersect in precisely one point, then the Jordan curve theorem requires that away from this point  $\wihat\gamma$ is on one side of
$\alpha\cdot \wihat\gamma$ and thus a similar argument yields a contradiction. The claim is thus proven.

We now prove that $\wihat\gamma$ and $\alpha\cdot \wihat\gamma$ coincide (up to reparametrization): If they do not coincide, 
we can use cut and paste constructions at the intersection points to produce 
a loop in $\hat M$ shorter than $L$, freely homotopic to $\wihat\gamma$, which is again a contradiction.

Thus, $\hat\gamma$ is invariant under the action of $\alpha$ (up to shift in the parameter) and all statements of the lemma follow directly from this.
\end{proof}

\begin{proposition}\label{prop:closed-geo-on-surfaces}
  Let $M$ be an orientable, closed, connected surface of genus at
  least~$1$ with Riemannian metric~$g$.
  Let~$\gamma\colon\R\to M$ be a
  (non-constant) closed geodesic with respect to~$g$, that
  minimizes length in its free homotopy class. Then, $\gamma$ is a minimal
  geodesic.
\end{proposition}
\begin{proof}
  As always, we assume that the geodesic $\gamma$ is
  parametrized by arclength.
  Let $\ell>0$ be the period of~$\gamma$.  As $\gamma$ is not
  constant, the free homotopy class of~$\gamma\stelle{[0,\ell]}$ is
  non-trivial and thus also the class~$\alpha\ceq [\gamma\stelle{[0,\ell]}]
  \in\pi_1\bigl(M,\gamma(0)\bigr)$
  is non-trivial. 

  Let $\ucov\gamma\colon\R\to \ucov M$ be a lift of~$\gamma$.
  Because the fundamental group of the surface~$M$ is torsion-free,
  $\alpha$ has infinite order. Moreover,
  $\alpha^k\cdot \ucov\gamma(0)$ converges for~$k\to\infty$ to
  a point~$p_\infty\in \partial\ucov M$, where $\partial\ucov M$ is
  the boundary at infinity of~$\ucov M$.  Similarly, the sequence
  converges to a point~$p_{-\infty}\in \partial\ucov M$ for~$k\to-\infty$.

  Assume for a contradiction that $\ucov\gamma$ is not a line.
  Then there exist~$t_0,t_1\in \R$ with $t_0<t_1$ and 
  \[ \ucov d\bigl(\ucov\gamma(t_0),\ucov\gamma(t_1)\bigr)<t_1-t_0= \maL\bigl(\gamma\stelle{[t_0,t_1]}\bigr)
  \,.
  \] 
  By increasing~$t_1$ and decreasing~$t_0$ we may assume that~$t_i=k_i\ell$
  with~$k_i\in \Z$ for~$i \in \{0,1\}$, and $k\ceq k_1-k_0\geq 2$. By shifting in the domain
  by~$t_0$ and in the range by $\alpha^{k_0}$, we may assume $t_0=0$ and
  $t_1=k\ell$ with~$k\in \N_{\geq 2}$. 
  There is a curve~$\ucov\tau$
  of length $L$ less than~$k\ell$ from~$\ucov\gamma(0)$ to~$\ucov\gamma(k\ell)$. 
  The projections of $\ucov \gamma\stelle{[0,k\ell]}$ and $\ucov \tau$ to~$M$ are loops and 
  will be denoted as $\gamma'$ and $\tau$. Both of them represent the free homotopy class defined by $\alpha^k$.
  As $\tau$ is shorter than $\gamma'$, 
  we have $N_0\bigl(\freehom{\alpha^k}\bigr)\leq \mathcal{L}(\tau)<k\ell$. This 
  contradicts Lemma~\ref{lemma:closed-non-prim-geod} and our proposition is shown.
\end{proof}
Note that orientability is crucial for these results. For example, if~$M$ is the Klein bottle, one can construct a Riemannian metric $g$ on $M$ such that Lemma~\ref{lemma:closed-non-prim-geod}
and Proposition~\ref{prop:closed-geo-on-surfaces} do no longer hold for this non-orientable surface~$M$ and the metric~$g$.

\fi 

\clearpage

\section*{Statements and declarations}
\noindent\emph{Funding.} This work was supported by the CRC 1085 Higher Invariants (Universität Regensburg, funded by the DFG). B.\thinspace A.\ was also supported
by the SPP~2026 (Geometry at infinity), funded by the DFG

\noindent\emph{Conflicts of interest.} There are no conflicts of interest to declare.

\noindent\emph{Author contribution.} Both authors (Bernd Ammann and Clara Löh) contributed equally to all aspects of the project: research, first draft, revisions, final manuscript. Both authors read and approved the final manuscript.

\noindent\emph{Data availability statement.} No external data is processed in this project.


\begin{thebibliography}{10}

\bibitem{ammann:diploma}
{\sc Ammann, B.}
\newblock Minimale {G}eod\"atische auf {M}annigfaltigkeiten mit nilpotenter
  {F}undamentalgruppe.
\newblock Diplomarbeit, Universit\"at Freiburg, Germany, 1994.
\newblock \doi{10.5283/epub.53183}.

\bibitem{ammann:97}
{\sc Ammann, B.}
\newblock Minimal geodesics and nilpotent fundamental groups.
\newblock {\em Geom. Dedicata 67\/} (1997), 129--148.

\bibitem{ammann.loeh_edge}
{\sc Ammann, B., and Löh, C.}
\newblock Edges of symmetric polytopes.
\newblock Work in progress.

\bibitem{babenkobalacheff}
{\sc Babenko, I., and Balacheff, F.}
\newblock Sur la forme de la boule unit{\'e} de la norme stable
  unidimensionnelle.
\newblock {\em Manuscripta Math. 119}, 3 (2006), 347--358.

\bibitem{bangert_DynamicReported1988}
{\sc Bangert, V.}
\newblock Mather sets for twist maps and geodesics on tori.
\newblock In {\em Dynamics reported, {V}ol. 1}, vol.~1 of {\em Dynam. Report.
  Ser. Dynam. Systems Appl.} Wiley, Chichester, 1988, pp.~1--56.

\bibitem{bangert:90}
{\sc Bangert, V.}
\newblock Minimal geodesics.
\newblock {\em Ergod. Th. \& Dynam. Sys. 10}, 2 (1990), 263--286.

\bibitem{bangert:94}
{\sc Bangert, V.}
\newblock Geodesic rays, {B}usemann functions and monotone twist maps.
\newblock {\em Calc. Var. Partial Differential Equations 2}, 1 (1994), 49--63.

\bibitem{bleecker:74}
{\sc Bleecker, D.~D.}
\newblock The {G}auss-{B}onnet inequality and almost-geodesic loops.
\newblock {\em Advances in Math. 14\/} (1974), 183--193.

\bibitem{bliss:1902}
{\sc Bliss, G.~A.}
\newblock The geodesic lines on the anchor ring.
\newblock {\em Ann. of Math. (2) 4}, 1 (1902), 1--21.

\bibitem{bolotin.rabinowitz:99}
{\sc Bolotin, S.~V., and Rabinowitz, P.~H.}
\newblock Minimal heteroclinic geodesics for the {$n$}-torus.
\newblock {\em Calc. Var. Partial Differential Equations 9}, 2 (1999),
  125--139.

\bibitem{burago:soviet:92}
{\sc Burago, D.}
\newblock Periodic metrics.
\newblock {\em Advances in Soviet Mathematics 9\/} (1992), 205--210.

\bibitem{burago.burago.ivanov:01}
{\sc Burago, D., Burago, Y., and Ivanov, S.}
\newblock {\em A course in metric geometry}, vol.~33 of {\em Graduate Studies
  in Mathematics}.
\newblock American Mathematical Society, Providence, RI, 2001.

\bibitem{burago.ivanov.kleiner:97}
{\sc Burago, D., Ivanov, S., and Kleiner, B.}
\newblock On the structure of the stable norm of periodic metrics.
\newblock {\em Math. Res. Lett. 4}, 6 (1997), 791--808.

\bibitem{contreras.paternain:2002}
{\sc Contreras, G., and Paternain, G.~P.}
\newblock Connecting orbits between static classes for generic {L}agrangian
  systems.
\newblock {\em Topology 41}, 4 (2002), 645--666.

\bibitem{fathi:1998}
{\sc Fathi, A.}
\newblock Orbites h\'et\'eroclines et ensemble de {P}eierls.
\newblock {\em C. R. Acad. Sci. Paris S\'er. I Math. 326}, 10 (1998),
  1213--1216.

\bibitem{federer.VP}
{\sc Federer, H.}
\newblock Real flat chains, cochains and variational problems.
\newblock {\em Indiana Univ. Math. J. 24\/} (1974/75), 351--407.

\bibitem{freedman-hass-scott:82}
{\sc Freedman, M., Hass, J., and Scott, P.}
\newblock Closed geodesics on surfaces.
\newblock {\em Bull. London Math. Soc. 14}, 5 (1982), 385--391.

\bibitem{gromov::structures-metriques}
{\sc Gromov, M.}
\newblock {\em Structures m\'{e}triques pour les vari\'{e}t\'{e}s
  riemanniennes}, vol.~1 of {\em Textes Math\'{e}matiques [Mathematical
  Texts]}.
\newblock CEDIC, Paris, 1981.

\bibitem{gromov_hyperbolic_groups:87}
{\sc Gromov, M.}
\newblock Hyperbolic groups.
\newblock In {\em Essays in group theory}, vol.~8 of {\em Math. Sci. Res. Inst.
  Publ.} Springer, New York, 1987, pp.~75--263.

\bibitem{gromov:99}
{\sc Gromov, M.}
\newblock {\em Metric structures for {R}iemannian and non-{R}iemannian spaces}.
\newblock Birkh\"auser Boston Inc., Boston, MA, 1999.

\bibitem{hedlund:1932}
{\sc Hedlund, G.~A.}
\newblock Geodesics on a two-dimensional {R}iemannian manifold with periodic
  coefficients.
\newblock {\em Ann. of Math. (2) 33}, 4 (1932), 719--739.

\bibitem{jotz}
{\sc Jotz, M.}
\newblock Hedlund metrics and the stable norm.
\newblock {\em Differential Geom.\ Appl. 27}, 4 (2009), 543--550.

\bibitem{klingenberg:1971}
{\sc Klingenberg, W.}
\newblock Geod\"{a}tischer {F}luss auf {M}annigfaltigkeiten vom hyperbolischen
  {T}yp.
\newblock {\em Invent. Math. 14\/} (1971), 63--82.

\bibitem{loehthesis}
{\sc L{\"o}h, C.}
\newblock {\em $\ell^1$-{H}omology and {S}implicial {V}olume}.
\newblock PhD thesis, WWU~M{\"u}nster, 2007.
\newblock \url{http://nbn-resolving.de/urn:nbn:de:hbz:6-37549578216}.

\bibitem{massart_CRAS:97}
{\sc Massart, D.}
\newblock Normes stables des surfaces.
\newblock {\em C. R. Acad. Sci. Paris S\'{e}r. I Math. 324}, 2 (1997),
  221--224.

\bibitem{massart_GAFA:97}
{\sc Massart, D.}
\newblock Stable norms of surfaces: local structure of the unit ball of
  rational directions.
\newblock {\em Geom. Funct. Anal. 7}, 6 (1997), 996--1010.

\bibitem{meeks.patrusky::Pacific}
{\sc Meeks, III, W.~H., and Patrusky, J.}
\newblock Representing codimension-one homology classes by embedded
  submanifolds.
\newblock {\em Pacific J. Math. 68}, 1 (1977), 175--176.

\bibitem{milnor::morse-theory}
{\sc Milnor, J.}
\newblock {\em Morse theory}, vol.~No. 51 of {\em Annals of Mathematics
  Studies}.
\newblock Princeton University Press, Princeton, NJ, 1963.
\newblock Based on lecture notes by M. Spivak and R. Wells.

\bibitem{montealegre}
{\sc Montealegre, P.}
\newblock On the stable norm of slit tori and the {F}arey sequence, 2023.
\newblock Preprint, \arxiv{2310.05570}.

\bibitem{morse:1924}
{\sc Morse, H.~M.}
\newblock A fundamental class of geodesics on any closed surface of genus
  greater than one.
\newblock {\em Trans. Amer. Math. Soc. 26}, 1 (1924), 25--60.

\bibitem{navas:2011}
{\sc Navas, A.}
\newblock {\em Groups of circle diffeomorphisms}, spanish~ed.
\newblock Chicago Lectures in Mathematics. University of Chicago Press,
  Chicago, IL, 2011.

\bibitem{novik}
{\sc Novik, I.}
\newblock A tale of centrally symmetric polytopes and spheres.
\newblock In {\em Recent trends in algebraic combinatorics}, vol.~16 of {\em
  Assoc.\ Women Math.\ Ser.} Springer, 2019, pp.~305--331.

\bibitem{polyaszego}
{\sc P{\'o}lya, G., and Szeg{\H o}, G.}
\newblock {\em Problems and theorems in analysis.~{I}}.
\newblock Classics in Mathematics. Springer, 1998.
\newblock Series, integral calculus, theory of functions, Translated from the
  German by Dorothee Aeppli, Reprint of the 1978 English translation.

\bibitem{sakai_Riemannian_Geometry}
{\sc Sakai, T.}
\newblock {\em Riemannian geometry}, vol.~149 of {\em Translations of
  Mathematical Monographs}.
\newblock American Mathematical Society, Providence, RI, 1996.
\newblock Translated from the 1992 Japanese original by the author.

\bibitem{mschmidt}
{\sc Schmidt, M.}
\newblock {\em $L^2$-{B}etti {N}umbers of $\mathcal{R}$-{S}paces and the
  {I}ntegral {F}oliated {S}implicial {V}olume}.
\newblock PhD thesis, Westf{\"a}lische Wilhelms-Universit\"at M{\"u}nster,
  2005.
\newblock \url{http://nbn-resolving.de/urn:nbn:de:hbz:6-05699458563}.

\bibitem{straszewicz}
{\sc Straszewicz, S.}
\newblock {\"U}ber exponierte {P}unkte abgeschlossener {P}unktemengen.
\newblock {\em Fund.\ Math. 24\/} (1935), 139--143.

\end{thebibliography}

\end{document}